\documentclass[10pt]{article}
\usepackage{amsmath,amssymb,amsthm}
\usepackage{booktabs,graphicx}
\usepackage[round]{natbib}
\usepackage[hidelinks]{hyperref}
\hypersetup{pdftitle={Exact Calibration and Sharp Risk Geometry for Volume-Sampled Ridge Regression},pdfauthor={Kihun Rhee}}
\newtheorem{theorem}{Theorem}[section]
\newtheorem{lemma}[theorem]{Lemma}
\newtheorem{proposition}[theorem]{Proposition}
\newtheorem{corollary}[theorem]{Corollary}
\theoremstyle{definition}

\theoremstyle{remark}

\numberwithin{equation}{section}

\title{Exact Calibration and Sharp Risk Geometry\\for Volume-Sampled Ridge Regression}
\author{Kihun Rhee\\\href{mailto:rheekh00@snu.ac.kr}{rheekh00@snu.ac.kr}}
\date{October 2026}
\makeatletter
\renewcommand{\maketitle}{%
  \begin{center}
  {\Large\bfseries\@title\par}\vspace{1em}
  {\normalsize\@author\par}\vspace{.6em}
  {\normalsize\@date\par}
  \end{center}\vspace{1em}}
\makeatother
\begin{document}
\maketitle
\begin{abstract}
We study ridge regression from exactly $s$ distinct rows of a fixed design.
Responses are fixed, and only the subset is random.
The determinant law and selected ridge fit share one positive definite penalty.
Established mean identities and exponential-family duality give the unique
penalty that matches a prescribed full-data ridge fit in expectation.
It exists exactly when $s$ exceeds the target's effective dimension.

Our main result concerns centered covariance risk normalized by full-data
penalized loss. For balanced signed coordinate replicas, a strict sector
inequality gives the sharp risk and all maximizing responses at every budget
from the dimension to one below the row count. This holds for any nonzero
positive semidefinite query. With the target and query fixed, the maximizing
response space is unchanged across these budgets.
For general designs, we characterize attainment of a leave-one-out envelope.
For existing real equiangular tight frames, flat row query energy characterizes
when every nonzero residual response maximizes at two deletions.
At three deletions, we give the sharp risk and complete maximizing space
for isotropic queries, using unequal triangle weights.

The balanced geometry yields a same-sample unbiased ridge--Horvitz--Thompson
mixture with lower sharp risk and an exact mean-share improvement boundary.
Under full recalibration after feature changes, we prove quadratic regret
from searching the complete old maximizing space and a query-uniform bound
on the mixture's risk gain. The strongest sector inequalities have exact
computer-assisted proofs.

\end{abstract}
\section{Introduction}\label{sec:introduction}
A sampled regression fit may match a full-data fit in expectation and still
have substantial variance. Two questions therefore arise. Can an exact row
budget preserve a prescribed ridge target? If so, which responses make the
remaining covariance risk largest?

We study these questions for a fixed design and fixed responses.
We sample distinct original rows under a determinant law.
The sampling law and selected ridge fit share one positive definite penalty.
The full-data target, which we call Full, has its own prescribed penalty.
We choose the sampled penalty from features alone so that the selected fit
has mean Full for every response. We call this choice \emph{calibration}.

We first derive calibration from established selected-inverse identities
and exponential-family duality. This gives the basis for our covariance analysis.
Calibration exists exactly when the row budget exceeds the effective
dimension of the target. When it exists, the penalty is unique.
The result applies to any positive definite target.
It also allows feasible budgets below the feature dimension.
Thus a fixed number of observed labels can preserve a target even when
individual sampled designs are rank deficient.

Calibration makes the next question precise.
For any fixed query inputs, expected squared-loss excess over Full is a
covariance trace. We normalize it by Full's penalized loss and maximize over
nonzero responses. Our goal is both the sharp value and the complete set of
responses that attain it. These are different tasks from bounding a
coefficient covariance matrix or locating its zero-variance responses.

A balanced replica design has the same number of signed copies of each
coordinate vector. After undoing the row signs, each response group
splits into its mean and zero-sum contrasts. Our main technical result is a strict inequality
between the contrast and mean sectors. It holds at every budget from the dimension to one
below the row count. The strict coefficient inequality forces every worst-case response
to have zero group means.
The resulting deficit gives the sharp risk and every maximizing response
for each nonzero positive semidefinite query, including all ties.
The maximizing space in the original response coordinates stays fixed
as the budget changes, with the target and query held fixed.
The proof of the strict inequality is exact and computer-assisted.
A separate quantitative bound describes the minimum-budget, small-penalty corner.
An explicit lower-budget phase change shows why the budget condition matters.

Two complementary regimes extend the geometric picture.
For general designs, leave-one-out sampling gives an explicit risk envelope.
A row-dependence condition determines exactly when it is attained.
For an existing real equiangular tight frame (ETF), flat query energy means
that the query assigns the same squared length to every frame row.
Residual responses are orthogonal to the columns of the design.
At two deletions, flat query energy is necessary and sufficient for every
nonzero residual response to maximize risk. For isotropic queries at three deletions, we find the sharp value and
complete residual maximizing space. The proof uses the actual unequal
triangle weights.
These are separate exact classifications with distinct assumptions.
Table~\ref{tab:geometry-scope} records their scopes.

\begin{table}[ht]
\centering
\caption{Exact response geometry under the calibrated law.
The target and physical query are fixed while the response varies.}
\label{tab:geometry-scope}
\small
\begin{tabular}{@{}p{.25\textwidth}p{.29\textwidth}p{.40\textwidth}@{}}
\toprule
\raggedright Design and budget & \raggedright Target and query & \raggedright Conclusion\tabularnewline
\midrule
\raggedright General full rank; $s=m-1$ & \raggedright Any SPD target; any nonzero PSD query
& \raggedright Envelope and exact attainment test. The sharp value is smaller when the test fails.\tabularnewline[3pt]
\raggedright Balanced replicas; $d\le s<nd$ & \raggedright Positive scalar target; any nonzero PSD query
& \raggedright Sharp value and complete maximizing space at every stated budget.\tabularnewline[3pt]
\raggedright Existing real ETF; $s=m-2$ & \raggedright Positive scalar target; any nonzero PSD query
& \raggedright Flat row energy iff every residual maximizes. Exact value and space under flatness.\tabularnewline[3pt]
\raggedright Existing real ETF; $s=m-3$ & \raggedright Positive scalar target; isotropic query; $d,m-d\ge3$
& \raggedright Sharp value and complete residual maximizing space under unequal triangle weights.\tabularnewline
\bottomrule
\end{tabular}
\end{table}

The geometry has two further uses.
First, it guides estimator choice on the same sampled rows.
We mix calibrated ridge with a Horvitz--Thompson fit.
One feature-selected interval of positive weights improves the sharp risk
for every nonzero positive semidefinite query at a balanced base.
The maximizing space remains unchanged there.
This does not contradict calibration uniqueness: the mixture lies outside
the matched ridge family.
We also solve the response optimization under a lower bound on the fraction
of Full loss carried by group means. The resulting phase boundary states when all-query improvement
is possible within the chosen weight interval, which keeps the contrast
gap positive.
An exact adverse query gives the corresponding failure mechanism.

Second, the full maximizing space controls what happens when features change.
We recalibrate the penalty and update the law, selected fits and Full loss.
Reoptimizing over the complete old maximizing space has quadratic restriction
regret near a balanced base. Keeping only one old maximizing vector can incur
linear regret. Enlarging the space to all base contrasts gives a radius that
is uniform over normalized positive semidefinite queries.
A separate argument preserves the mixture's gain with error proportional to
both the feature change and the mixture weight.
This conclusion compares two moving sharp optima.
It does not require their maximizing responses to coincide.

\paragraph{Mathematical contributions.}
Calibration fixes the sampling procedure for the covariance question.
It follows from established mean identities and duality after a spectral reduction.
The main technical work proves strict sector inequalities and resolves
response equality in the stated balanced, LOO and ETF regimes.
The mixture and mean-share results turn this geometry into an estimator comparison.
The perturbation results then control the fully recalibrated procedure.
The mean-map, Rayleigh-quotient and generic spectral tools are standard.
The required covariance operators, sign arguments and coupled estimates
are specific to the sampling procedure studied here.

\paragraph{Foundations and related work.}
The sampling law is a fixed-size determinantal point process
\citep{kulesza2012}. Regularized determinant designs are established
\citep{derezinski2020design}.
A selected-inverse identity gives the fixed-penalty ridge mean
\citep{schreurs2020}. We invert that mean at a hard row budget using
finite exponential-family duality \citep{wainwright2008}.
Exact ridge debiasing also exists for surrogate sketches with random row
count \citep{derezinski2020debiasing}.
Here the number of distinct original rows is exact, and the same penalty
enters sampling and fitting. Fixed-size DPPs with prescribed row inclusions
also have a survey-sampling precedent \citep{loonis2019}.
Our calibration instead matches the full ridge response operator.

Volume sampling has exact fixed-response OLS identities
\citep{derezinski2018}. Normalized worst-response OLS design is studied by
\citet{derezinski2019minimax}, and OLS covariance contact by
\citet{rhee2026phase}.
Our denominator is the positive penalized Full ridge loss.
Our covariance is centered over the subset draw with deterministic responses.
Classical frame erasure results provide the pair and triangle geometry
\citep{goyal2001,bodmann2005}.
They concern different losses from the centered ridge risk studied here.
Likewise, Horvitz--Thompson weighting \citep{horvitz1952}, estimator averaging
\citep{lavancier2016}, and spectral perturbation \citep{yu2015}
provide foundations for our consequences.
Appendix~\ref{app:sources} makes these source relations explicit.

\paragraph{Reading order.}
Section~\ref{sec:calibration} establishes the matched sampling procedure.
Section~\ref{sec:general} defines normalized covariance risk, compares
feasible interior budgets, and gives the general leave-one-out equality test.
Section~\ref{sec:geometry} proves the all-budget balanced classification.
Section~\ref{sec:frames} treats two deletion budgets on existing real ETFs.
Sections~\ref{sec:mixture} and~\ref{sec:transfer} develop the estimator
and perturbation consequences. The appendices contain exact sign
certificates, explicit bounds, finite floating-point evaluations, and
source comparisons.

\section{Exact target calibration}\label{sec:calibration}
\subsection{Fixed target and fixed row budget}\label{sec:setup}
Let $X\in\mathbb R^{m\times d}$ have full column rank, $d\ge1$, with rows
$x_i^\top$, and put $G=X^\top X$. The response $y\in\mathbb R^m$
is fixed. Fix a positive definite target penalty $A_0$ and define
\begin{align}
 w_0&=(G+A_0)^{-1}X^\top y,\\
 J_0(X,y)&=\min_w\{\lVert y-Xw\rVert^2+w^\top A_0w\}
          =y^\top K_Xy,\\
 K_X&=(I+XA_0^{-1}X^\top)^{-1}.
\end{align}
Thus $J_0(X,y)>0$ for $y\ne0$. This denominator is the minimized
penalized full-data objective.

For an integer $1\le s\le m$ and a positive definite matrix $R$, draw
exactly $s$ distinct original rows using
\begin{equation}\label{eq:matched-law}
\begin{aligned}
 p_{s,R}(S)&=\frac{\det(X_S^\top X_S+R)}{Z_s(G,R)},\\
 w_S&=(X_S^\top X_S+R)^{-1}X_S^\top y_S.
\end{aligned}
\end{equation}
Write $G_S=X_S^\top X_S$. Here $Z_s$ sums the numerator over all
size-$s$ subsets. The same $R$
enters the law and the fit. Only $S$ is random. We call $R$ calibrated
when $\mathbb E w_S=w_0$ for every $y$, with $R$ independent of $y$.

\subsection{The inverse calibration theorem}
The law in \eqref{eq:matched-law} is a fixed-size determinantal point
process with $L$-ensemble kernel $I+XR^{-1}X^\top$
\citep[Section~5.1]{kulesza2012}.
Indeed, its principal determinant on $S$, multiplied by $\det R$, is
$\det(X_S^\top X_S+R)$. We ask which matched penalty recovers the
prescribed full-data operator at an exact original-row budget.

\begin{theorem}[Exact calibration]\label{thm:calibration}
Let $X$ have full column rank and $A_0\succ0$. For $1\le s\le m$,
a calibrated $R_s\succ0$ exists if and only if
\begin{equation}\label{eq:feasibility}
 s>d_{\mathrm{eff}}(A_0):=\operatorname{tr}\{G(G+A_0)^{-1}\}.
\end{equation}
It is unique among all positive definite matrices and depends only on
$G,A_0,m,s$. At $s=m$, $R_m=A_0$. For every feasible $s<m$,
\begin{equation}\label{eq:penalty-upper}
 0\prec R_s\preceq \frac{s}{m}A_0\prec A_0.
\end{equation}
If $d\le s<m$, then
\begin{equation}\label{eq:penalty-sandwich}
 \frac{s-d+1}{m-d+1}A_0\preceq R_s\preceq\frac{s}{m}A_0.
\end{equation}
Both inequalities in \eqref{eq:penalty-sandwich} are strict when
$d>1$. If $A_0$ is scalar, $R_s$ commutes with $G$.
\end{theorem}

The strict threshold can allow $s<d$, so sampled designs need not have
full column rank. Uniqueness concerns the joint law-and-fit requirement.
It is not identification of a determinantal kernel from subset
probabilities. Exact ridge debiasing with scaled regularization and an
effective-dimension threshold also occurs in random-size surrogate
sketches \citep[Theorems~1 and~6]{derezinski2020debiasing}.
Here the number of distinct, unrescaled original rows is exactly $s$.

The calibration theorem follows from two established results after a
spectral reduction. The forward mean formula follows from the fixed-size
selected-inverse identity of \citet[Lemma~1 and Corollary~1]{schreurs2020}.
The inverse mean map is a direct finite-support instance of
\citet[Proposition~3.2 and Theorem~3.3]{wainwright2008}.
The reduction identifies the spectral support and its feasible mean set.
The budget threshold, matrix uniqueness and penalty bounds then follow.
We give the derivation to make these steps and their assumptions explicit.
A change of coefficient coordinates reduces any SPD target to the identity
penalty. We retain the matrix form to state the answer in the original coordinates.

All gradients of functions of a symmetric matrix use the trace inner
product. We first establish the needed mean formula without assumptions
on simultaneous diagonalization.

\begin{lemma}[Determinant and mean identities]\label{lem:mean-identity}
For $R\succ0$,
\begin{align}
 \sum_{s=0}^m t^sZ_s(G,R)
   &=(1+t)^m\det\left(R+\frac{t}{1+t}G\right),\label{eq:partition-generating}\\
 \mathbb E_{s,R}w_S
   &=\{\nabla_G\log Z_s(G,R)\}X^\top y.\label{eq:mean-identity}
\end{align}
\end{lemma}
\begin{proof}
Expand the determinant into distinct rank-one row contributions. A term
using $k$ rows belongs to $\binom{m-k}{s-k}$ size-$s$ subsets. Summing in
$s$ proves \eqref{eq:partition-generating}. For the mean, allow nonsymmetric
rank-one contributions and replace $x_ix_i^\top$ by
$x_i(x_i+t y_i v)^\top$. On a selected subset the derivative at zero is
$\det(G_S+R)v^\top(G_S+R)^{-1}X_S^\top y_S$. On the total Gram matrix
the derivative is $v^\top(\nabla_G Z_s)X^\top y$. Equating these
derivatives and dividing by $Z_s$ proves \eqref{eq:mean-identity}.
\end{proof}

For $q_j>0$, put $c_k=\binom{m-k}{s-k}$, with impossible
binomial coefficients equal to zero, and define
\begin{equation}\label{eq:cal-partition}
 P_s(q)=\sum_{J\subseteq[d]}c_{|J|}\prod_{j\in J}q_j,
 \qquad \rho_j=q_j\partial_{q_j}\log P_s(q).
\end{equation}
Under the auxiliary law with mass proportional to
$c_{|J|}\prod_{j\in J}q_j$, $\rho_j$ is the probability that $j\in J$.
These are spectral-coordinate means, not original-row inclusion
probabilities.

\begin{lemma}[Finite inverse moment map]\label{lem:moment-map}
Let $1\le s\le m$ and $m\ge d$. The gradient of
$A_s(z)=\log P_s(e^z)$, with $P_s$ in \eqref{eq:cal-partition}, is a
bijection onto
$\{\rho:0<\rho_j<1,\ \sum_j\rho_j<s\}$.
\end{lemma}
\begin{proof}
The distribution on $J\subseteq[d]$ with mass proportional to
$c_{|J|}\exp(\sum_{j\in J}z_j)$ has support all zero-one vectors with at
most $s$ ones. Its convex hull is
$\mathcal C_s=\{\rho:0\le\rho_j\le1,\ \sum_j\rho_j\le s\}$.
To check the hull, any vertex with two fractional coordinates admits a
feasible perturbation; a vertex with one fractional coordinate cannot
satisfy an active integral sum constraint and also admits a perturbation.
Thus every vertex is a stated zero-one vector.

The gradient of $A_s$ is the mean indicator, and its Hessian is the
indicator covariance. A zero-variance linear combination must agree on
the empty set and every singleton, so all its coefficients are zero.
The Hessian is therefore positive definite. Every finite parameter gives
positive probability to every support point, so its mean is interior.
Conversely, for an interior $\rho$, a ball of radius $\delta>0$ about
$\rho$ lies in $\mathcal C_s$. The support function then gives
\[
 A_s(z)-\rho^\top z\ge\delta\lVert z\rVert+
             \min_{k:c_k>0}\log c_k.
\]
This function is coercive and strictly convex. Its unique minimizer
solves $\nabla A_s(z)=\rho$. This also proves injectivity. When a sum
constraint is redundant, the displayed strict inequality remains true
through $0<\rho_j<1$.
\end{proof}

\begin{proof}[Proof of Theorem~\ref{thm:calibration}]
Use coefficient coordinates $\widetilde w=G^{1/2}w$ and design
$\widetilde X=XG^{-1/2}$, whose Gram matrix is $I$. Write
\[
 C=G^{-1/2}RG^{-1/2},\qquad
 D=G^{1/2}(G+A_0)^{-1}G^{1/2},\qquad 0\prec D\prec I.
\]
If $C=U\operatorname{diag}(c_j)U^\top$, put $q_j=1/c_j$. The
determinant expansion gives $Z_s(I,C)=\det(C)P_s(q)$, and
Lemma~\ref{lem:mean-identity} gives the mean operator
\[
 \Phi_s(C)\widetilde X^\top,\qquad
 \Phi_s(C)=U\operatorname{diag}(\rho_j(q))U^\top.
\]
The desired operator is $D\widetilde X^\top$. Full row rank of
$\widetilde X^\top$ makes calibration equivalent to $\Phi_s(C)=D$.
Thus every solution $C$ commutes with $D$. In a simultaneous eigenbasis,
Lemma~\ref{lem:moment-map} gives exactly one positive eigenvalue vector
$q$ if and only if $\operatorname{tr}D<s$. Equal target eigenvalues
force equal $q$ coordinates, since swapping them preserves the unique
moment-map solution. There is no freedom to rotate a different penalty
inside a repeated eigenspace. This proves all-matrix uniqueness and
\eqref{eq:feasibility}. Equivalently, the exact identity
\[
 \rho_i-\rho_j=\frac{q_i-q_j}{P_s(q)}
 \sum_{J\subseteq[d]\setminus\{i,j\}}c_{|J|+1}\prod_{k\in J}q_k
\]
has a strictly positive final factor, whose empty term is $c_1>0$.
Unwhitening gives $R_s=G^{1/2}CG^{1/2}$. For scalar $A_0$, $D$ is a
spectral function of $G$, so $R_s$ commutes with $G$.

For the bounds, write $q_J=\prod_{k\in J}q_k$ and split
$P_s=B_j+q_jA_j$. Then
\[
 \frac{\rho_j}{1-\rho_j}=q_j\frac{A_j}{B_j},\qquad
 \frac{A_j}{B_j}
 =\frac{\sum_{J\not\ni j}c_{|J|}q_J (s-|J|)/(m-|J|)}
        {\sum_{J\not\ni j}c_{|J|}q_J}.
\]
Terms with zero coefficient are omitted. At $s<m$ all denominators here
are positive; this weighted average is at most $s/m$. For $s\ge d$ it lies between
$(s-d+1)/(m-d+1)$ and $s/m$. The eigenvalues of the whitened target
penalty are $(1-\rho_j)/\rho_j$, and the ratio of the calibrated to
target eigenvalues is $A_j/B_j$. Congruence proves
\eqref{eq:penalty-upper} and \eqref{eq:penalty-sandwich}. If $d>1$ and
$d\le s<m$, the positive weights at $|J|=0$ and $|J|=d-1$ have distinct
ratios, so both sandwich inequalities are strict. At $s=m$ the fit is deterministic; equality
of its operator with the target and full rank give $R_m=A_0$ directly.
\end{proof}

\subsection{Singleton and noncommuting examples}
\begin{corollary}[Two closed forms]\label{cor:calibration-closed}
If $s=1$ is feasible, then
\begin{equation}\label{eq:singleton-penalty}
 R_1=\frac{1-d_{\mathrm{eff}}(A_0)}{m}(G+A_0).
\end{equation}
For $d=1$ and $A_0=\lambda_0>0$, every budget has
$r_s=s\lambda_0/m$.
\end{corollary}
\begin{proof}
For $s=1$, $P_1(q)=m+\sum_jq_j$. Hence
$q_j=m\rho_j/(1-\sum_k\rho_k)$.
In the whitened coordinates of the theorem, this says
$C=(1-\operatorname{tr}D)D^{-1}/m$.
Congruence by $G^{1/2}$ proves \eqref{eq:singleton-penalty}.
For $d=1$, $P_s(q)=c_0+c_1q$ and $c_1/c_0=s/m$.
The odds equation in the theorem's proof gives $r_s=s\lambda_0/m$.
\end{proof}

Take $m=4$, rows $(1,0),(1,1),(0,1),(0,0)$, and
\[
 G=\begin{pmatrix}2&1\\1&2\end{pmatrix},\qquad
 A_0=\begin{pmatrix}3&1\\1&4\end{pmatrix}.
\]
Then $(G+A_0)^{-1}=\frac1{26}\begin{pmatrix}6&-2\\-2&5\end{pmatrix}$,
so $d_{\mathrm{eff}}=9/13<1$ and
$R_1=(G+A_0)/13$. Thus one sampled row can preserve this two-dimensional
target. Also $(GR_1-R_1G)_{12}=1/13\ne0$.
The penalty need not commute with $G$ when the target does not.
The simultaneous spectral reduction occurs after Gram whitening.

\section{General covariance structure and leave-one-out equality}\label{sec:general}
\subsection{Risk, response operators, and equality}
Write $\Sigma_s(y)=\operatorname{Cov}_S(w_S)$ under the calibrated law.
Fix a query metric $W\succeq0$ in the original coefficient coordinates.
We call this the physical query and hold it fixed as the response varies.
For $W=Q^\top Q$ and any fixed query response $y_Q$, calibration gives
\begin{equation}\label{eq:query-excess}
 \mathbb E\lVert Qw_S-y_Q\rVert^2-\lVert Qw_0-y_Q\rVert^2
 =\operatorname{tr}(W\Sigma_s(y)).
\end{equation}
Indeed, expand around $Qw_0-y_Q$. The cross term has expectation zero,
and the remaining quadratic term is the covariance trace.
Our sharp risk is
\begin{equation}\label{eq:sharp-risk}
 h_s(X,W)=\sup_{y\ne0}\frac{\operatorname{tr}(W\Sigma_s(y))}{J_0(X,y)}.
\end{equation}
We write $h_s(X)$ or $h_s$ when the fixed query is clear.
This supremum introduces no stochastic model for labels.

Each selected fit is linear in $y$. If $F_Sy=w_S$ and
$F_0=(G+A_0)^{-1}X^\top$, define the positive semidefinite response matrix
\begin{equation}\label{eq:response-operator}
 C_s(X,W)=\mathbb E(F_S-F_0)^\top W(F_S-F_0),\qquad
 \operatorname{tr}(W\Sigma_s(y))=y^\top C_s(X,W)y.
\end{equation}
This is an $m$-by-$m$ matrix for scalar query risk. The coefficient
covariance $\Sigma_s(y)$ is a different, $d$-by-$d$ matrix.
The normalized operator is
\begin{equation}\label{eq:normalized-operator}
 M_s(X,W)=K_X^{-1/2}C_s(X,W)K_X^{-1/2},\qquad
 h_s(X,W)=\lambda_{\max}(M_s(X,W)).
\end{equation}
For a proposed bound $\rho$, the standard quadratic-form test is
$\Delta_\rho=\rho K_X-C_s(X,W)\succeq0$.
A nonzero response attains that bound exactly when it belongs to
$\ker\Delta_\rho$. If this kernel is zero, the bound is strict.
The results below identify these operators and kernels under stated
sampling and design assumptions.

\subsection{A common energy across feasible budgets}
Before resolving equality at one budget, we ask whether zero variance can
occur on different response spaces at different feasible interior budgets.
For $d\ge2$, a common energy shows that these spaces are the same for
total covariance or a fixed positive definite query.
The comparison constants need not be sharp.

\begin{lemma}[Zero total variance]\label{lem:zero-variance}
Let $d\ge2$, $1\le s<m$, and $R\succ0$. Under any law positive on
every size-$s$ subset, the selected ridge fits satisfy
\[
 \operatorname{Cov}(w_S)=0
 \quad\Longleftrightarrow\quad x_i y_i=0\quad\text{for every }i.
\]
Thus the zero-variance responses are exactly those supported on zero rows.
This statement does not require calibration.
\end{lemma}
\begin{proof}
Zero covariance makes all fits the same vector $w$. Set
$v_i=x_i(y_i-x_i^\top w)$. The normal equations give
$\sum_{i\in S}v_i=Rw$ for every size-$s$ subset.
For any two indices, choose $s-1$ other indices and exchange the remaining
index. This shows that all $v_i$ agree. Their common value lies in every
row line. Those lines have zero intersection because the rows span at
least two dimensions. Hence $v_i=0$ and $Rw=0$, so $w=0$ and $x_iy_i=0$.
The converse follows because every selected right-hand side is then zero.
\end{proof}

Define the full-data residual and the common energy by
\begin{equation}\label{eq:common-energy}
 e_0=y-Xw_0,\qquad b_i=x_i(e_0)_i,\qquad
 \bar b=\frac{A_0w_0}{m},\qquad
 \mathcal E_0(y)=\sum_i\lVert b_i-\bar b\rVert^2.
\end{equation}
The full normal equation gives $\sum_i b_i=A_0w_0$.
For $d\ge2$, put
\[
 L_x^2=\max_i\lVert x_i\rVert^2,\qquad
 \mu_X=\frac{\operatorname{tr}G-\lambda_{\max}(G)}{L_x^2}>0.
\]
At each feasible interior budget define
\begin{align}\label{eq:energy-constants}
 a_s&=\frac{s(m-s)}{m(m-1)},&
 \kappa_s&=\frac{\det(G+R_s)}{\det R_s},&
 \Lambda_s&=\lambda_{\max}(G+R_s),\nonumber\\
 r_{\min,s}&=\lambda_{\min}(R_s),&
 \mathcal D_s&=\frac{s}{m}I-R_sA_0^{-1},\nonumber\\
 \underline c_s&=\frac{a_s}{\kappa_s\Lambda_s^2},&
 \overline c_s&=\frac{\kappa_s}{r_{\min,s}^2}
 \left(a_s+\frac{m^2\lVert \mathcal D_s\rVert_{\mathrm{op}}^2}{\mu_X}\right).
\end{align}
These symbols are scalar comparison constants, separate from $C_s$, $M_s$,
and the sharp risk $h_s$.

\begin{proposition}[Common-energy comparison]\label{prop:common-energy}
Let $d\ge2$. At every feasible $1\le s<m$,
\begin{equation}\label{eq:common-energy-bound}
 0<\underline c_s\le\overline c_s<\infty,\qquad
 \underline c_s\mathcal E_0(y)\le\operatorname{tr}\Sigma_s(y)
 \le\overline c_s\mathcal E_0(y).
\end{equation}
Consequently, for any two feasible interior budgets $s,t$,
\begin{equation}\label{eq:total-comparison}
 \frac{\underline c_s}{\overline c_t}\operatorname{tr}\Sigma_t(y)
 \le\operatorname{tr}\Sigma_s(y)
 \le\frac{\overline c_s}{\underline c_t}\operatorname{tr}\Sigma_t(y).
\end{equation}
For a fixed $W\succ0$, the same comparison holds for
$\operatorname{tr}(W\Sigma_s)$ after multiplying the lower constant by
$\lambda_{\min}(W)/\lambda_{\max}(W)$ and the upper constant by its
reciprocal. The energy $\mathcal E_0$ vanishes exactly on the zero-row
response space.
\end{proposition}
\begin{proof}
Set $\delta_S=w_S-w_0$. The selected normal equation gives
\[
 (G_S+R_s)\delta_S=\sum_{i\in S}b_i-R_sw_0=:t_S.
\]
Let $U_s$ be the uniform law on size-$s$ subsets. Write
$t_S=\sum_{i\in S}(b_i-\bar b)+[(s/m)A_0-R_s]w_0$.
Uniform inclusion probabilities are $s/m$ for one index and
$s(s-1)/[m(m-1)]$ for two distinct indices. Since
$\sum_i(b_i-\bar b)=0$, expansion yields
\begin{equation}\label{eq:uniform-energy}
 \mathbb E_{U_s}\lVert t_S\rVert^2
 =a_s\mathcal E_0+\lVert[(s/m)A_0-R_s]w_0\rVert^2.
\end{equation}
For a nonzero row, $b_i$ belongs to its row line. Projection onto the
orthogonal complement of that line therefore gives
\[
 \lVert b_i-\bar b\rVert^2\ge
 \frac{\lVert x_i\rVert^2\lVert\bar b\rVert^2-(x_i^\top\bar b)^2}{L_x^2}.
\]
For a zero row the right side is zero. Summing and using the largest
eigenvalue of $G$ gives
\[
 \mathcal E_0\ge\mu_X\lVert\bar b\rVert^2,\qquad
 \lVert A_0w_0\rVert^2\le\frac{m^2}{\mu_X}\mathcal E_0.
\]
Thus \eqref{eq:uniform-energy} lies between $a_s\mathcal E_0$ and
$(a_s+m^2\lVert \mathcal D_s\rVert_{\mathrm{op}}^2/\mu_X)\mathcal E_0$.

The order bounds $R_s\preceq G_S+R_s\preceq G+R_s$ imply
\[
 \frac{\lVert t_S\rVert}{\Lambda_s}\le\lVert\delta_S\rVert
 \le\frac{\lVert t_S\rVert}{r_{\min,s}}.
\]
Every determinant weight lies between $\det R_s$ and $\det(G+R_s)$.
Its ratio to the average weight therefore lies between
$\kappa_s^{-1}$ and $\kappa_s$, so
$\kappa_s^{-1}\le p_{s,R_s}(S)/U_s(S)\le\kappa_s$.
Calibration gives $\mathbb E_p\delta_S=0$ and hence
$\operatorname{tr}\Sigma_s=\mathbb E_p\lVert\delta_S\rVert^2$.
Combining the bounds proves \eqref{eq:common-energy-bound}.
Combining budgets proves \eqref{eq:total-comparison}; the extreme
eigenvalues of $W$ prove its query version.
Lemma~\ref{lem:zero-variance} identifies the common zero space.
\end{proof}

\paragraph{Why singular queries need a separate analysis.}
Take two copies of each coordinate row $e_1,e_2$, fix
$A_0=(54/13)I$, and put $y=(1,1,1,1)^\top$, $u=(1,1)^\top$.
The calibrated penalties at sizes two and three are
\[
 R_2=r_2I,\quad r_2=\frac{7+\sqrt{283}}{13},\qquad R_3=3I.
\]
Here the full target coefficient is $(G+A_0)^{-1}X^\top y=(13/80)X^\top y$.
To verify calibration, at size two the count patterns $(2,0),(0,2),(1,1)$
have weights $r(r+2),r(r+2),4(r+1)^2$. Symmetry within each group gives
the all-response mean coefficient
\[
 \frac{3r+2}{6r^2+12r+4}.
\]
Equating it to $13/80$ gives $39r^2-42r-54=0$, with positive root $r_2$.
At size three, each omission has equal probability, and the mean
coefficient is $[1/(r+1)+2/(r+2)]/4$, which equals $13/80$ at $r=3$.
For size three, $u^\top w_S=1/4+2/5=13/20$ deterministically.
At size two it takes the distinct values $2/(2+r_2)$ and $2/(1+r_2)$,
each with positive probability. Thus
\[
 u^\top\Sigma_3(y)u=0,\qquad u^\top\Sigma_2(y)u>0.
\]
No positive lower comparison from size two to size three can hold for
the singular query $W=uu^\top$. Calibration alone does not provide
response-uniform two-sided Loewner comparisons.

\subsection{General-design leave-one-out geometry}\label{app:loo-proof}
Let $m>d\ge1$ and $s=m-1$, with arbitrary $A_0\succ0$ and its calibrated
$R\succ0$. This budget is feasible because
$d_{\mathrm{eff}}(A_0)<d\le m-1$.
Write $B=(G+R)^{-1}$, $L=I-XBX^\top$, and abbreviate $K_X$ by $K$.
For a fixed nonzero $W\succeq0$, define
\begin{equation}\label{eq:loo-scores}
 \begin{gathered}
 \ell_i=x_i^\top Bx_i,\qquad v_i=x_i^\top BWBx_i,\qquad
 D=\sum_i(1-\ell_i),\\
 \omega_i=\frac{v_i}{D(1-\ell_i)},\qquad
 \rho=\max_i\omega_i,\qquad M=\{i:\omega_i=\rho\}.
 \end{gathered}
\end{equation}
These scores combine query energy with the actual omission probabilities.
A \emph{residual response} satisfies $X^\top y=0$; for such a response,
$w_0=0$ and $J_0(y)=\lVert y\rVert^2$.

\begin{theorem}[LOO envelope and attainment]\label{thm:loo-envelope}
Under the preceding assumptions, $\rho>0$. The risk satisfies
$h_{m-1}(X,W)\le\rho$.
The equality kernel at this envelope is
\begin{equation}\label{eq:loo-kernel}
 \ker(\rho K-C_{m-1})
 =\{y:X^\top y=0,\ y_i=0\text{ for }i\notin M\}.
\end{equation}
Thus $h_{m-1}(X,W)=\rho$ if and only if the rows $X_M$ are linearly
dependent. When this holds, \eqref{eq:loo-kernel} is the complete
maximizing space, of dimension $|M|-\operatorname{rank}(X_M)$.
If those rows are independent, $h_{m-1}(X,W)<\rho$.
\end{theorem}
\begin{proof}
Every deleted normal matrix $G+R-x_ix_i^\top$ is positive definite.
The determinant lemma and the rank-one inverse identity give
\begin{equation}\label{eq:loo-deletion}
 0\le\ell_i<1,\qquad p_i=\frac{1-\ell_i}{D},\qquad
 w_{-i}=BX^\top y-\frac{Bx_i(Ly)_i}{1-\ell_i}.
\end{equation}
Also $D=m-\operatorname{tr}(GB)>m-d\ge1$.
The mean of the deleted correction is $D^{-1}BX^\top Ly$.
Since $X^\top L=RBX^\top$, all-response calibration and full row rank
of $X^\top$ give
\begin{equation}\label{eq:loo-calibration-identity}
 (G+A_0)^{-1}=B-D^{-1}BRB,\qquad
 K-L=D^{-1}XBRBX^\top.
\end{equation}
The second matrix is positive semidefinite with kernel $\ker(X^\top)$.
No commutation is used here.

Taking the correction's second moment in \eqref{eq:loo-deletion} and
subtracting its mean square gives the centered response operator
\begin{equation}\label{eq:loo-centered-operator}
 C_{m-1}=L\operatorname{diag}(\omega_i)L
       -D^{-2}LXBWBX^\top L.
\end{equation}
Woodbury gives $L=(I+XR^{-1}X^\top)^{-1}$, so $0\prec L\preceq I$.
Spectral calculus for this matrix gives $L-L^2\succeq0$. Therefore
\begin{equation}\label{eq:loo-full-gap}
 K-L^2=(K-L)+(L-L^2)\succeq0,\qquad
 \ker(K-L^2)=\ker(X^\top).
\end{equation}
Since the vectors $Bx_i$ span coefficient space and $W\ne0$, at least
one $v_i=\lVert W^{1/2}Bx_i\rVert^2$ is positive. Hence $\rho>0$.
Combining the identities gives the exact deficit
\begin{equation}\label{eq:loo-deficit-main}
 \rho K-C_{m-1}=\rho(K-L^2)
   +L\operatorname{diag}(\rho-\omega_i)L
   +D^{-2}LXBWBX^\top L.
\end{equation}
All three terms are positive semidefinite. The first forces
$X^\top y=0$ at equality. Then $Ly=y$, the last term vanishes, and the
middle term vanishes exactly when $y$ is supported on $M$.
This proves \eqref{eq:loo-kernel}. Such responses identify with
$\ker(X_M^\top)$, giving the dimension and attainment criterion.
If the dimension is zero, the deficit is positive definite. Compactness
of $\{y:y^\top Ky=1\}$ then gives $h_{m-1}(X,W)<\rho$.
\end{proof}

The theorem includes tied scores, singular queries, and zero rows.
A zero row has score zero and cannot belong to $M$.
When the envelope is unattained, the theorem does not compute the
smaller sharp value or classify its maximizers.

\begin{corollary}[Combined balance and the coloop qualification]\label{cor:loo-balance}
If $v_i/(1-\ell_i)=a$ for every row, then $a>0$,
$h_{m-1}(X,W)=a/D$, and the complete maximizing space is $\ker(X^\top)$.
Conversely, if every nonzero residual response globally maximizes risk,
then $\omega_i=h_{m-1}(X,W)$ on each non-coloop row.
Here a \emph{coloop} is a row whose deletion lowers column rank.
If there are no coloops, the following conditions are equivalent:
constant scores; every nonzero residual maximizes; and the complete
maximizing space is $\ker(X^\top)$.
\end{corollary}
\begin{proof}
Constant scores make $M$ contain every row, and $m>d$ guarantees a
nonzero residual. The theorem proves the forward statement.
For the converse write $h=h_{m-1}(X,W)$. At any global maximizer $z$,
Rayleigh stationarity gives $C_{m-1}z=hKz$. For residual $z$, $Lz=Kz=z$,
and the centered term in \eqref{eq:loo-centered-operator} vanishes.
Hence $L\operatorname{diag}(\omega_i)z=hz$.
Multiplying by $L^{-1}$ gives $\operatorname{diag}(\omega_i)z=hz$.
Row $i$ is not a coloop exactly when some residual has $z_i\ne0$:
a dependence with that nonzero coefficient expresses row $i$ using the
others, and conversely. Thus $\omega_i=h$ on every such row.
With no coloops this proves constancy and the stated equivalence.
\end{proof}

\paragraph{Attainment on a proper residual subspace.}
Take two copies of $e_1$, followed by two copies of $e_2$, and set
\[
 R=\operatorname{diag}(1,2),\qquad
 A_0=\operatorname{diag}(7/5,20/7),\qquad W=I.
\]
Substitution in \eqref{eq:loo-calibration-identity} verifies calibration.
Here $D=17/6$ and
$(\omega_i)=(1/17,1/17,1/34,1/34)$.
Thus $h_3=1/17$ and the complete maximizing space is
$\operatorname{span}\{(1,-1,0,0)^\top\}$.
The scores need not be constant for some residuals to maximize.

\paragraph{An unattained envelope.}
For $X=(1,2,3)^\top$, $R=W=1$, and $A_0=3/2$,
\eqref{eq:loo-calibration-identity} holds and the scores are
$(1/434,4/341,3/62)$. The unique maximal row is nonzero, hence
independent. Therefore $h_2<3/62$.

\paragraph{Why the coloop condition matters.}
Take
\[
 X=\begin{pmatrix}1&0\\1&0\\0&1\end{pmatrix},\qquad
 R=W=I,\qquad A_0=\operatorname{diag}(5/3,7/4).
\]
The scores are $(1/11,1/11,3/11)$, but
\eqref{eq:loo-centered-operator} gives
\[
 \frac1{11}K-C_2=
 \begin{pmatrix}7/363&7/363&0\\7/363&7/363&0\\0&0&1/121\end{pmatrix}.
\]
This matrix is positive semidefinite with kernel
$\ker(X^\top)=\operatorname{span}\{(1,-1,0)^\top\}$.
Every residual therefore maximizes at $1/11<\rho$.
The third row is a coloop and escapes the converse's constraint.

\subsection{Coefficient-matrix contact is a different equality question}
The preceding theorem fixes $W$ and varies the response.
A coefficient-matrix ceiling instead gives a simultaneous bound for all
coefficient query directions at each response.

\begin{proposition}[Calibrated LOO matrix contact]\label{prop:loo-matrix-contact}
Under the LOO assumptions above, put $H=G+R$ and
$t_\ell=\max_i\ell_i/(1-\ell_i)$. Then
\begin{equation}\label{eq:loo-matrix-ceiling}
 \Sigma_{m-1}(y)\preceq J_0(y)\frac{t_\ell}{D}H^{-1}.
\end{equation}
For $y\ne0$, the kernel of the difference in \eqref{eq:loo-matrix-ceiling}
is nonzero exactly when $X^\top y=0$ and all rows in
$\operatorname{supp}(y)$ are signed copies of one nonzero vector $v$
with globally maximal calibrated leverage $v^\top H^{-1}v$.
In that case the complete kernel in the original coefficient query
coordinates is $\operatorname{span}(v)$; otherwise it is zero.
\end{proposition}

Appendix~\ref{app:calibration-support} gives the full positive semidefinite
decomposition and the kernel argument. The leverage is evaluated at the
calibrated penalty $R$, not at the prescribed target $A_0$.
For any $W\succeq0$,
$v_i\le\ell_i\operatorname{tr}(WB)$, so
\[
 \rho\le\frac{t_\ell}{D}\operatorname{tr}(WB).
\]
Thus tracing the matrix ceiling can lose the sharp fixed-query value.
Its contact kernel is also separate from the zero-total-variance space
of Lemma~\ref{lem:zero-variance}.

\paragraph{A noncommuting example with unequal omission probabilities.}
Consider
\begin{equation}\label{eq:loo-witness}
 \begin{gathered}
 X=\begin{pmatrix}1&1\\0&1\\3/5&7/5\\-3/5&-7/5\end{pmatrix},\qquad
 R=\begin{pmatrix}1&1\\1&2\end{pmatrix},\\
 A_0=\frac1{1125}\begin{pmatrix}1607&1583\\1583&3152\end{pmatrix},\qquad
 W=\frac15\begin{pmatrix}28&52\\52&118\end{pmatrix}.
 \end{gathered}
\end{equation}
The matrices requiring positive definiteness have positive diagonal
entries and determinants. They satisfy \eqref{eq:loo-calibration-identity}.
Direct substitution gives
\[
 \begin{gathered}
 (\ell_i)=(41/100,17/50,1/4,1/4),\qquad
 (v_i)=(59/100,33/50,3/4,3/4),\\
 D=11/4,\qquad (p_i)=(59,66,75,75)/275.
 \end{gathered}
\]
Every score is $4/11$, although neither leverage nor query energy is
constant. Hence the sharp value is $4/11$ and the complete maximizing
space is spanned by
$(-3/5,-4/5,1,0)^\top$ and $(3/5,4/5,0,1)^\top$.
There are three distinct row lines in dimension two, so no invertible
feature change makes this a coordinate-replica design.
The upper-right entries of $GR-RG$, $GA_0-A_0G$, and $RA_0-A_0R$ are,
respectively, $-38/25$, $-836/375$, and $-38/1125$.
Finally, $t_\ell=41/59$ and $\operatorname{tr}(WB)=4$.
The traced matrix ceiling is therefore $656/649$, strictly larger than
the sharp value $4/11$ under the same law, target, and query.

\section{Balanced geometry across budgets}\label{sec:geometry}
\subsection{Balanced replicas across budgets}
Repeated-coordinate designs and countwise ridge calculations have a
precedent in \citet[proof of Theorem~18]{derezinski2018}.
Here we use them to study centered covariance over the calibrated subset draw.
The key step is the strict sector ordering; the equality space follows
from the resulting positive-semidefinite deficit.
Suppose now that $m=nd$, with $n,d\ge2$, and group $j$ contains $n$
rows $x_i=\varepsilon_i e_j$, where $\varepsilon_i\in\{-1,1\}$.
Fix $A_0=\lambda_0 I$, $\lambda_0>0$. Symmetry and
Theorem~\ref{thm:calibration} give $R_s=r_s I$ for every $d\le s<m$.
Here $q_j$ denotes a signed group mean. For $i$ in group $j$, write
\begin{equation}\label{eq:contrasts}
\begin{gathered}
 \varepsilon_i y_i=q_j+z_i,\qquad \sum_{i\in j}z_i=0,\\
 E_j=\sum_{i\in j}z_i^2,\qquad b=\frac{n\lambda_0}{n+\lambda_0}.
\end{gathered}
\end{equation}
Then $w_0=nq/(n+\lambda_0)$ and $J_0=\sum_jE_j+b\lVert q\rVert^2$.

The group counts $K_j=|S\cap j|$ have the global law
\begin{equation}\label{eq:count-law}
 \mathbb P(K=k)\ \propto\ \prod_{j=1}^d\binom n{k_j}(r_s+k_j),
 \qquad \sum_jk_j=s.
\end{equation}
There are no fixed group quotas. Set $f(k)=k/(r_s+k)$ and define
\begin{align}\label{eq:geometry-coefficients}
 \alpha_s&=\mathbb E\frac{K_1(n-K_1)}{n(n-1)(r_s+K_1)^2},\\
 \beta_s&=\frac{\mathbb E(f(K_1)-f(K_2))^2}{2b},\\
 \tau_s&=-\frac{\operatorname{Cov}(f(K_1),f(K_2))}{b}.\label{eq:geometry-tau}
\end{align}
Conditional on $K$, the selected rows are independent uniform subsets
within their groups. If $T_j$ is the selected sum of signed contrasts, then
\[
 \mathbb E(T_j\mid K)=0,\qquad
 \operatorname{Var}(T_j\mid K)=\frac{K_j(n-K_j)}{n(n-1)}E_j,
 \qquad (w_S)_j=f(K_j)q_j+\frac{T_j}{r_s+K_j}.
\]
Total covariance and exchangeability give
\begin{equation}\label{eq:block-covariance}
 \Sigma_s(y)=\alpha_s\operatorname{diag}(E_j)
   +b\beta_s\operatorname{diag}(q_j^2)-b\tau_sqq^\top.
\end{equation}
The $\alpha_s$ term comes from sampling within groups. The $\beta_s$
and $\tau_s$ terms come from random group counts. The next theorem shows
that contrasts have a strictly larger risk coefficient than group means
at every stated budget.

\begin{theorem}[Sharp query risk and all maximizers]\label{thm:sharp-geometry}
For the balanced design above, every $d\le s<nd$ satisfies
$\alpha_s>\beta_s$ and $\tau_s>0$. For any fixed nonzero $W\succeq0$,
put $c=\max_jW_{jj}$ and $I_{\max}=\{j:W_{jj}=c\}$. Then
\begin{equation}\label{eq:sharp-balanced}
 h_s(X,W)=\alpha_s c.
\end{equation}
The nonzero maximizing responses are exactly the nonzero elements of
\begin{equation}\label{eq:old-space}
 \mathcal E=\{y:q=0,\ z_i=0\text{ for }i\text{ in groups outside }I_{\max}\}.
\end{equation}
This raw response space has dimension $(n-1)|I_{\max}|$ and is the same
for every stated budget at the fixed target and query.
\end{theorem}

\begin{proof}
The strict signs determine the sharp value and every maximizing response
through the exact deficit
\begin{align}\label{eq:sharp-deficit}
 \alpha_scJ_0-\operatorname{tr}(W\Sigma_s)
 ={}&\alpha_s\sum_j(c-W_{jj})E_j\nonumber\\
 &+b\sum_j(\alpha_sc-\beta_sW_{jj})q_j^2\nonumber\\
 &+b\tau_sq^\top Wq.
\end{align}
Every mean coefficient is strictly positive. Equality therefore forces
all group means to vanish and all contrasts to lie in $I_{\max}$.
Conversely those responses attain the bound. Off-diagonal entries of
$W$ affect individual risks but neither the sharp value nor this space.

The hard step is the strict sign $\alpha_s>\beta_s$, including middle
budgets. Its proof is computer-assisted and exact. Deletion identities
reduce it to a reciprocal moment. A variance bound places an auxiliary
mean above a quadratic root. Two reciprocal lower bounds cover different
parts of the root domain, and finite integer polynomial certificates
prove the needed signs and monotonicity on that entire domain.
Here are the precise reductions used by that sign proof. For
$g(x)=(1+x)^{n-1}[r+(r+n)x]$, put
\[
\begin{split}
 Z&=[x^s]g(x)^d,\quad
 Q=[x^s](1+x)^{d(n-1)+1}[r+(r+n)x]^{d-1},\\
 T&=[x^{s-1}](1+x)^{d(n-1)}[r+(r+n)x]^{d-2},\\
 A&=[x^{s-1}]g(x)^{d-1}
       \sum_{j=0}^{n-2}\binom{n-2}{j}\frac{x^j}{r+1+j}.
\end{split}
\]
Deletion of one or two determinant factors gives
$\alpha=A/Z$ and $\beta=[nT-(n-1)A]/[(n+r)Q]$.
Thus it suffices to prove $A\{(n+r)Q+(n-1)Z\}>nZT$.
For $n\ge3,d\ge4$, select $s-1$ items with product weights from
$d(n-1)$ weight-one items and $d-2$ items of weight $(r+n)/r$.
For disjoint ordinary blocks of sizes $n-2,n$, with counts $U,V$,
coefficient extraction gives $A/T=\mathbb E[(r+V)/(r+1+U)]$.
Two lower bounds are used: the pointwise integer bound
\[
 \frac1{r+1+U}\ge\frac{r+2-U}{(r+1)(r+2)}
\]
and weighted Cauchy. The variance of the number of unselected
exceptional items places its mean above the positive root of an explicit
quadratic. At that root, three rational cells cover the full domain.
On the two lower cells the integer bound has positive residual; on the
upper cell Cauchy does. Exact positive-coefficient expansions also prove
the monotonicity needed to pass from the root to the actual mean.
Appendix~\ref{app:geometry} gives every defining polynomial, the cell
maps and their inverses, both finite seams, and the minimum-budget face.
It also proves the separate $n=2$ and $d=2,3$ cases, and the direct LOO
endpoint. These cases cover exactly the theorem's domain. Newton's
coefficient inequality gives strict negative covariance of the increasing
count scores, hence $\tau>0$; its boundary-aware conditional argument is
included there. All large coefficient checks use exact integer arithmetic.
The preceding deficit then proves the theorem.
\end{proof}

For example, when $n=d=2$, $s=3$, and $\lambda_0=10/7$, calibration
gives $r_s=1$ and $\alpha_s=1/16$. For $W=I$, all responses
$(a,-a,b,-b)$ maximize the ratio in the positive-copy realization.
The sharp value is $1/16$, while the isotropic matrix ceiling
$\Sigma_s\preceq\alpha_sJ_0I$ gives the larger trace bound $1/8$.
At $s=nd$ every covariance is zero and every nonzero response maximizes.
The theorem excludes that endpoint and makes no assertion of monotone
risk across budgets or of optimality over other estimators.

\subsection{A quantitative margin at the minimum budget}
The sign above holds for every positive penalty. Near zero penalty at
$s=d$, its size also has a simple bound.
\begin{proposition}[Minimum-budget margin]\label{prop:small-penalty-gap}
In the balanced model with $n,d\ge2$ and $s=d$, if the calibrated raw
penalty satisfies $0<r_d\le1/(4d^2)$, then
\[
 \alpha_d-\beta_d\ge\frac1{4[d(n-1)+1]}.
\]
For fixed $n,d$, as the raw penalty $r\downarrow0$ and the induced target
also tends to zero,
\[
 \alpha\longrightarrow\frac1n,\qquad
 \beta\longrightarrow\frac{(n-1)(d-1)}{n[d(n-1)+1]},\qquad
 \frac{\lambda_0}{r}\longrightarrow d(n-1)+1.
\]
\end{proposition}
\begin{proof}
Put $M=d(n-1)$ and use the deletion quantities $A,Q,T,Z$ from the
preceding proof at $s=d$. In the auxiliary experiment, the number $H$
of unselected exceptional items has mass proportional to
\[
 \binom{d-2}{j}\binom M{j+1}\left(\frac r{r+n}\right)^j,
 \qquad 0\le j\le d-2.
\]
The ordinary selected count is $1+H$. Write $h=\mathbb EH$, and let
$U,V$ count disjoint ordinary blocks of sizes $n-2,n$. Define
\begin{align*}
 Q_n&=rnd(nd-1)+nd(M+1)+n^2dh,\\
 Z_n&=rQ_n+nd[r(nd-1)+n(1+h)],\\
 H_n&=(n+r)Q_n+(n-1)Z_n,\\
 B&=Mr+n(1+h),\qquad C=Mr(d-1)-(1+h)r(d-1).
\end{align*}
Adding the deleted factors and taking ordinary factorial moments gives
\[
 \frac QT=\frac{Q_n}{dM},\quad \frac ZT=\frac{Z_n}{dM},\quad
 \mathbb E(r+V)=\frac BM,\quad
 \mathbb E[(r+V)U]=\frac{(n-2)C}{M(M-1)}.
\]
These are also the specialization $u=1$ of
\eqref{eq:actual-mean-identities}; they remain valid when $n=2$ or $d=2$.
The integer reciprocal bound gives
\[
 \frac AT\ge
 \frac{(r+2)(M-1)B-(n-2)C}{M(M-1)(r+1)(r+2)}.
\]
Consequently, if
\[
 P(h)=[(r+2)(M-1)B-(n-2)C]H_n
       -n(r+1)(r+2)M(M-1)Z_n,
\]
then
\begin{equation}\label{eq:margin-reduction}
 \alpha-\beta\ge
 \frac{dP(h)}{(M-1)(r+1)(r+2)(n+r)Q_nZ_n}.
\end{equation}
The short expansions in Appendix~\ref{app:margin-algebra} show that
$P(h)=P(0)+p_1h+p_2h^2$, with $p_1,p_2>0$, and
\[
 P(0)\ge dn^4(M-1)[2-d^2(3r+4r^2+r^3)].
\]
On the stated window the subtracted expression is at most
$3/4+1/16+1/1024<1$. Hence $P(h)\ge dn^4(M-1)$.

The adjacent mass ratio for $H$ gives the drift identity
\[
 r(d-2)(M-1)=r(nd-2)h+nh+n\mathbb EH^2.
\]
Since $H^2\ge H$, we obtain
$0\le h\le r(d-2)(M-1)/(2n)\le rd^2/2\le1/8$.
For $Q_0=nd(M+1)$ and $Z_0=n^2d$, it follows that
\begin{align*}
 Q_n/Q_0
 &=1+\frac{r(nd-1)+nh}{M+1}
 \le1+2r+r(d-2)/2\le1+rd\le9/8,\\
 Z_n/Z_0
 &=1+h+\frac{r(nd-1)}n+
          \frac{r(M+1)}n\frac{Q_n}{Q_0}
 \le1+\frac18+\frac{17}8rd<\frac32,\\
 \frac{(r+1)(r+2)(n+r)}{2n}
 &\le\frac{17}{16}\left(\frac{33}{32}\right)^2
 =\frac{18513}{16384}.
\end{align*}
Here $(nd-1)/(M+1)\le2$, $(M+1)/n\le d$, $rd\le1/8$,
and $r\le1/16$ were used. Substitution into
\eqref{eq:margin-reduction} yields
\[
 \alpha-\beta\ge
 \frac1{M+1}\frac{1/2}{(18513/16384)(9/8)(3/2)}
 =\frac{131072}{499851(M+1)}>\frac1{4(M+1)}.
\]
Finally, the coefficient definitions at $r=0$ give
$Z=n^d$, $A=n^{d-1}$, $Q=n^{d-1}(M+1)$, and $T=n^{d-2}M$.
Their positive denominators give the asserted limits, including
$\alpha-\beta\to1/(M+1)$. Since $b=nrQ/Z$ and
$\lambda_0=nb/(n-b)$, also $\lambda_0/r\to M+1$.
\end{proof}
The variance normalization $b$ vanishes at this boundary. Setting all
count scores to one before normalizing would lose the nonzero limit of
$\beta$. The proposition gives a finite bound only when the calibrated
penalty meets its stated window.

\subsection{Why the budget restriction matters}
Calibration can exist below $d$. The strict sector order can then reverse.
Take two positive copies of each coordinate in dimension three, and
sample $s=2$. The count patterns are permutations of $(2,0,0)$ and
$(1,1,0)$. Their individual unnormalized masses are respectively
$r^2(r+2)$ and $4r(r+1)^2$. Thus, with $A_r=5r^2+10r+4$,
\[
 Z=3rA_r,\qquad
 \rho=\mathbb E f(K_1)=\frac{2(5r+4)}{3A_r},\qquad
 \lambda_0=\frac{2(1-\rho)}\rho
           =\frac{15r^2+20r+4}{5r+4}.
\]
For example, the numerator for $\mathbb Ef(K_1)$ is
$2r^2+8r(r+1)$: one pattern has $K_1=2$ and two have $K_1=1$.
The map $r\mapsto\lambda_0$ increases from $1$ to infinity, since its
derivative is $(75r^2+120r+60)/(5r+4)^2>0$. This agrees with the
feasibility condition $2>6/(2+\lambda_0)$.

The same six patterns give
\[
 \mathbb Ef(K_1)^2=\frac{4r^2/(r+2)+8r}{Z},\qquad
 \mathbb E[f(K_1)f(K_2)]=\frac{4r}{Z}.
\]
Only the two patterns with $K_1=1$ contribute to $\alpha$.
Using $b=2(1-\rho)$ in the coefficient definitions therefore gives
\begin{align*}
 \alpha&=\frac4{3A_r},\qquad
 \beta=\frac{4(r+1)}{(r+2)(15r^2+20r+4)},\\
 \tau&=\frac{4(5r^2+5r+2)}{3A_r(15r^2+20r+4)},\\
 \beta-\alpha
 &=-\frac{4(5r^2+2r-4)}{3(r+2)A_r(15r^2+20r+4)}.
\end{align*}
For the physical rank-one query $W=vv^\top$, $v=(1,1,1)^\top$, the
contrast coefficient is $\alpha$. The normalized mean operator is
$\beta I-\tau vv^\top$, whose largest eigenvalue is $\beta$ on
$v^\perp$. Hence the sharp risk is $\max(\alpha,\beta)$.
The transition occurs at
\[
 r_*=(\sqrt{21}-1)/5,\qquad
 \lambda_*=(8+2\sqrt{21})/5.
\]
For $1<\lambda_0<\lambda_*$, all maximizers are pure means with
$q\perp v$. At equality, they are all sums of arbitrary contrasts and
means with $q\perp v$. For $\lambda_0>\lambda_*$, precisely the pure
contrasts maximize. Indeed the remaining mean eigenvalue is
$\beta-3\tau<\beta$. This exact phase change occurs inside the feasible
calibrated domain. It does not contradict Theorem~\ref{thm:sharp-geometry},
whose budgets satisfy $s\ge d$.

\section{Frame geometry at two and three deletions}\label{sec:frames}
\subsection{Two deletions and flat physical queries}
At two deletions, we ask which queries make every nonzero residual
response worst case on an equiangular tight frame.
Assume that $X$ is an existing real unit-row equiangular tight frame
(ETF), with $d\ge2$ and $m\ge d+2$. Thus, for $g=m/d$,
\begin{equation}\label{eq:etf-definition}
 X^\top X=gI,\qquad
 (x_i^\top x_j)^2=\mu^2=\frac{g-1}{m-1}\quad(i\ne j).
\end{equation}
Fix $A_0=\lambda_0I$, $\lambda_0>0$, and use $s=m-2$.
Every such target has a unique calibrated penalty $rI$, also unique
among SPD penalties. Under this law, deleting any pair of distinct
original rows has the same probability. This uses the classical ETF pair geometry
\citep{bodmann2005}; it does not assume that ETFs exist for all $(m,d)$.

Put $b=(g+r)^{-1}$, $a=1-b$, $\delta=a^2-b^2\mu^2$, and
$N=\binom m2$. The residual coefficient is
\begin{equation}\label{eq:etf-cr}
 c_R=\frac{b^2\{(m-2)(a^2+b^2\mu^2)+2ab(g-2)\}}
 {N\delta^2}.
\end{equation}
\begin{theorem}[ETF query criterion]\label{thm:etf-query}
Under these assumptions, fix any nonzero $W\succeq0$ and put
$w=\operatorname{tr}(W)/d$. The following are equivalent:
\begin{enumerate}
\item $x_i^\top Wx_i=w$ for every original row;
\item every nonzero response in $\ker(X^\top)$ globally maximizes risk;
\item the complete maximizing space is $\ker(X^\top)$.
\end{enumerate}
When they hold, $h_{m-2}(X,W)=wc_R$, and the maximizing space has
dimension $m-d$. Singular queries are included.
\end{theorem}

\begin{proof}\label{app:etf-proof}
\par\medskip\noindent\textbf{Calibration and the centered operator.}
Assume the ETF conditions of Theorem~\ref{thm:etf-query}. Put
$T=XX^\top$, $P=T/g$, $P_\perp=I-P$, $q=\mu^2=(g-1)/(m-1)$,
and $h=g-1$. Here $q$ is the squared ETF coherence, not a group mean.
Let $k=m-d\ge2$. Then
\begin{equation}\label{eq:etf-retained-positive}
 h^2-q=\frac{k\{k(m-1)-d\}}{d^2(m-1)}>0,
 \qquad k(m-1)-d=(k-1)m>0.
\end{equation}
The Gram matrix after deleting two rows has eigenvalues $g$ with
multiplicity $d-2$, and $h+\mu,h-\mu$. All are positive.

For a positive raw penalty $r$, define
\begin{equation}\label{eq:etf-basic-coefficients}
 b=(g+r)^{-1},\quad a=1-b,\quad \ell=1-bg=br,\quad
 \delta=a^2-b^2q,\quad N=\binom m2,\quad
 L=I-bT=P_\perp+\ell P.
\end{equation}
For a deleted pair $D_0=\{i,j\}$ let $E_{D_0}$ select its two
coordinates, $X_{D_0}=E_{D_0}X$, and $t=T_{ij}$. Its deletion matrix is
\begin{equation}\label{eq:etf-pair-inverse}
 Q_{D_0}=I_2-bE_{D_0}TE_{D_0}^\top
 =\begin{pmatrix}a&-bt\\-bt&a\end{pmatrix},\qquad
 Q_{D_0}^{-1}=\delta^{-1}\begin{pmatrix}a&bt\\bt&a\end{pmatrix}.
\end{equation}
The deletion matrix $Q_{D_0}$ is positive definite. Every retained determinant is
$(g+r)^d\delta$, so the actual law is uniform on the $N$ original
pairs. This is the pair geometry associated with ETFs
\citep{bodmann2005}; it does not use a rotation of the response coordinates.
The fit can be written as
\begin{equation}\label{eq:etf-correction}
 w_S=bX^\top y-z_{D_0}(y),\qquad
 z_{D_0}(y)=bX_{D_0}^\top Q_{D_0}^{-1}E_{D_0}Ly.
\end{equation}
Counting diagonal and off-diagonal pair occurrences gives
\begin{align}
 \sum_{D_0}E_{D_0}^\top Q_{D_0}^{-1}E_{D_0}
  &=\delta^{-1}\{[(m-1)a-b]I+bT\},\nonumber\\
 e&=(m-1)a+b(g-1),\qquad
 \eta=\frac{b\ell e}{N\delta}>0,\nonumber\\
 \mathbb E z_{D_0}(y)&=\eta X^\top y,\qquad
 \mathbb E w_S=f(r)X^\top y,\quad f(r)=b-\eta.\label{eq:etf-mean}
\end{align}
To identify $f$, apply this mean identity to $y=Xv$ and take traces
of the mean fit map. The retained spectrum above gives
\begin{equation}\label{eq:etf-calibration}
 f(r)=\frac1m\left\{
 \frac{(d-2)g}{g+r}+
 \frac{h+\mu}{h+\mu+r}+\frac{h-\mu}{h-\mu+r}\right\}.
\end{equation}
It decreases continuously and strictly from $1/g$ to zero. Hence every
prescribed $\lambda_0>0$ has exactly one positive scalar solution of
$f(r)=1/(g+\lambda_0)$. Theorem~\ref{thm:calibration} gives uniqueness
among all SPD penalties because $s=m-2\ge d$ is feasible. At this
calibration,
\begin{equation}\label{eq:etf-full-loss}
 k_F=1-gf=\frac{\lambda_0}{g+\lambda_0}=\ell+g\eta>0,
 \qquad K=P_\perp+k_FP.
\end{equation}
Thus this construction reaches a prescribed target, not just a target
chosen after the raw penalty.

\par\medskip\noindent\textbf{The operator for any symmetric query.}
Let $H=XWX^\top$, $h_i=H_{ii}$, and $D_h=\operatorname{diag}(h_i)$.
Define
\begin{equation}\label{eq:etf-operator-coefficients}
 \alpha=a^2(m-2)+2ab(g-2)-2b^2q,\quad
 \beta=b^2q,\quad \gamma=ab,\quad \theta=a^2+b^2q.
\end{equation}
The coefficient covariance is
\begin{equation}\label{eq:etf-actual-covariance}
 \operatorname{Cov}(w_S(y))=
 \frac1N\sum_{D_0}z_{D_0}(y)z_{D_0}(y)^\top
   -\eta^2X^\top yy^\top X.
\end{equation}
For a pair $D_0=\{i,j\}$, signed multiplication in
\eqref{eq:etf-pair-inverse} gives
\begin{align}
 (Q_{D_0}^{-1}H_{D_0D_0}Q_{D_0}^{-1})_{11}
  &=\frac{a^2h_i+2ab\,tH_{ij}+b^2q h_j}{\delta^2},\nonumber\\
 (Q_{D_0}^{-1}H_{D_0D_0}Q_{D_0}^{-1})_{12}
  &=\frac{ab\,t(h_i+h_j)+\theta H_{ij}}{\delta^2}.
 \label{eq:etf-signed-product}
\end{align}
The second diagonal entry swaps $i,j$. Since $TH=HT=gH$, summing the
first line over $j\ne i$ has numerator
$[a^2(m-1)+2ab(g-1)-b^2q]h_i+b^2q\operatorname{tr}(H)$.
Each off-diagonal entry occurs in one pair. Therefore
\begin{align}
 M_W&=\frac{\alpha D_h+\beta\operatorname{tr}(H)I
                 +\gamma(TD_h+D_hT)+\theta H}{N\delta^2},\nonumber\\
 C&=b^2LM_WL-\eta^2H.\label{eq:etf-arbitrary-operator}
\end{align}
This is the response operator, not the coefficient covariance matrix.
The last subtraction centers at the actual mean in \eqref{eq:etf-mean}.
Before imposing any flatness, its cross block is
\begin{equation}\label{eq:etf-cross-block}
 P_\perp CP=\frac{b^2\ell(\alpha+g\gamma)}{N\delta^2}P_\perp D_hP.
\end{equation}
Indeed $H$ and $T$ kill residuals, while $L$ acts as the identity on
residuals and as $\ell$ on the fitted space.

\par\medskip\noindent\textbf{Flat-query deficit, strict signs, and converse.}
If $h_i=w=\operatorname{tr}(W)/d$ for every row, then
$D_h=wI$ and $\operatorname{tr}(H)=mw$. Define
\begin{align}
 A&=\alpha+m\beta=(m-2)(a^2+b^2q)+2ab(g-2),\nonumber\\
 c_R&=\frac{b^2A}{N\delta^2},\qquad
 u=\frac{b^2\ell^2(A+2gab)}{N\delta^2},\nonumber\\
 \tau&=\frac{b^2\ell^2(e^2-N\theta)}{N^2\delta^2},\qquad
 \xi=c_Rk_F-u.\label{eq:etf-flat-coefficients}
\end{align}
Substitute $D_h=wI$ in \eqref{eq:etf-arbitrary-operator} and use
$LH=HL=\ell H$. The result is
\begin{equation}\label{eq:etf-flat-deficit}
 C=wc_RP_\perp+wuP-\tau H,\qquad
 wc_RK-C=w\xi P+\tau H.
\end{equation}
The coefficient $c_R$ is \eqref{eq:etf-cr}. It is also the isotropic
residual coefficient: its numerator is
$(m-1)[a^2+b(2a+b)q]-[a^2+2ab+b^2q]=A$,
using $(m-1)q=g-1$.

We now prove all strict signs needed for sufficiency and necessity.
The assumptions imply
\begin{equation}\label{eq:etf-sign-domain}
 m\ge4,\qquad h=g-1\ge\frac2{m-2}>0,\qquad
 mh>2,\qquad a=b(h+r)>bh,\qquad 0<\ell<1.
\end{equation}
First,
\begin{equation}
 A-2a\ell=a[(m-4)a+2b(2g-3)]+(m-2)b^2q>0.
\end{equation}
To check the sign even if $2g-3<0$, the expression in square brackets
is at least $b[(m-4)h+2(2g-3)]=b(mh-2)>0$.
Since $k_F=\ell+g\eta>\ell$, it follows that
\begin{equation}\label{eq:etf-xi-positive}
 \xi=\frac{b^2\{A(k_F-\ell^2)-2gab\ell^2\}}{N\delta^2}
 >\frac{b^3g\ell(A-2a\ell)}{N\delta^2}>0.
\end{equation}
In particular $A>0$, so $c_R>0$.

For the second sign put $t_0=a/b=h+r>h$. Then
\begin{equation}
 \frac{e^2-N\theta}{b^2}=F(t_0),\qquad
 F(t)=\frac{(m-1)(m-2)}2t^2+2(m-1)th+h^2-\frac{mh}2.
\end{equation}
This polynomial increases strictly for $t\ge h>0$, and
$F(h)=mh[(m+1)h-1]/2>0$. Hence $\tau>0$.
Finally,
\begin{equation}
 \frac{\alpha+g\gamma}{b^2}=G(t_0),\qquad
 G(t)=(m-2)t^2+(3h-1)t-\frac{2h}{m-1}.
\end{equation}
For $t\ge h$,
$G'(t)\ge2(m-2)h+3h-1\ge3+3h>0$.
Also $G(h)=h[(m+1)h-1-2/(m-1)]>0$, because
$(m+1)h\ge2(m+1)/(m-2)>2$ and $1+2/(m-1)\le5/3$.
Thus $\alpha+g\gamma>0$ and the scalar in
\eqref{eq:etf-cross-block} is strictly positive. These signs hold for
all positive raw penalties in the stated domain.

For $W\succeq0$, $W\ne0$, we have $H\succeq0$ and $w>0$.
The first term in \eqref{eq:etf-flat-deficit} has kernel
$\ker(X^\top)$ and the second term vanishes on that space. Consequently
\begin{equation}
 \ker(wc_RK-C)=\ker(X^\top).
\end{equation}
There are $m-d>0$ residual dimensions. They attain $wc_R$, while every
response with a fitted component has strictly smaller normalized risk.
The argument does not invert $W$, so singular queries are included.
For $W=wI$, writing $c_F=u-g\tau$ gives the isotropic sector gap
$c_Rk_F-c_F=\xi+g\tau>0$. Positivity of covariance also gives $c_F\ge0$.
Thus no separate sign certificate for that special case is needed here.

For necessity, suppose every nonzero residual is a global maximizer,
with sharp value $h_\star$. Since $K\succ0$, full Rayleigh stationarity
gives $Cz=h_\star Kz=h_\star z$ on residuals. Hence $PCP_\perp=0$,
and by symmetry $P_\perp CP=0$. The positive coefficient in
\eqref{eq:etf-cross-block} implies $P_\perp D_hP=PD_hP_\perp=0$.
Thus $D_h$ commutes with $P$. Entrywise,
$(h_i-h_j)T_{ij}/g=0$. Every off-diagonal $T_{ij}$ is nonzero since
$q>0$, so all $h_i$ agree. Their sum is $g\operatorname{tr}(W)=mw$;
therefore each is $w$. This proves the equivalence in
Theorem~\ref{thm:etf-query}. Constancy of a residual compression alone
would not prove this converse; the cross block is essential.

\end{proof}

The flat symmetric query class has $d(d+1)/2-m$ traceless directions.
Indeed the row projectors have positive definite Frobenius Gram matrix
$(1-q)I+q\mathbf1\mathbf1^\top$, so they are independent; their sum is
$gI$ \citep[Section~5.1]{fickus2018}. Appendix~\ref{app:etf-witness}
gives the full dimension argument and an actual $(16,6)$ frame with a
rank-three nonisotropic flat query. Failure of flatness excludes the
claim that every residual maximizes. It does not classify individual
residual maximizers or the nonflat sharp value.
For the same flat query and prescribed target, single deletion also has
constant row scores, so Theorem~\ref{thm:loo-envelope} gives the same
complete maximizing space. The penalties are separately calibrated.

\subsection{Three deletions under unequal determinant weights}
We now keep isotropic queries and allow one more deletion. Pair deletion
was uniform because every pair has the same squared inner product.
At three deletions the signed triangle product changes the determinant.
The two triangle classes and their incidence counts are classical ETF
geometry \citep[Section~5.2, Lemma~5.19 and Proposition~5.20]{bodmann2005}.
Their determinant masses follow from the triple spectra. We retain those
unequal masses in the centered covariance calculation. The key step is
the strict normalized sector gap that identifies every maximizing response.

The formulas below summarize the first two resolvent moments under this
weighted law. For isotropic queries, they give one covariance coefficient
for residual responses and one for fitted responses. Under the assumptions
below, the theorem proves a strict gap after normalization by the Full loss.
Thus the residual space is exactly the maximizing space.

Assume an existing real unit-row ETF with $d\ge3$ and $k=m-d\ge3$.
Keep $g=m/d$, $T=XX^\top$, $P=T/g$, and
$q=k/[d(m-1)]$ and $\mu=\sqrt q>0$. Fix $A_0=\lambda_0I$, $\lambda_0>0$, and
$W=wI$, $w>0$. At $s=m-3$ the same positive scalar penalty $rI$ is used
in the determinant law and the selected fit. Define
\begin{align}
 z&=g+r,\quad h=z-1,\quad
 t=\frac{2(g-1)(g-2)}{(m-1)(m-2)},\quad D=h^3-3hq-t,\nonumber\\
 H_1&=\frac{d-3}{z}+\frac{3(h^2-q)}D,\qquad
 H_2=\frac{d-3}{z^2}+
       \frac{3[h(h^2-q)+3t]}{(h^2-4q)D},\label{eq:triple-moments}\\
 M&=d-rH_1,\quad B=d-2rH_1+r^2H_2,\quad
 k_F=1-\frac Md,\nonumber\\
 c_R&=\frac{H_1-rH_2-B/g}{k},\qquad
 c_F=\frac Bm-\frac{gM^2}{m^2}.\label{eq:triple-sectors}
\end{align}
The symbol $h$ in these formulas is a scalar resolvent parameter, not a
sharp risk. All quantities are evaluated at the unique calibrated root
$M(r)/m=1/(g+\lambda_0)$.

\begin{theorem}[Isotropic ETF geometry at three deletions]
\label{thm:etf-triple}
Under the preceding assumptions, every prescribed $\lambda_0>0$ has
exactly one calibrated positive scalar penalty, also unique among all SPD
penalties. The centered isotropic response operator is
\[
 C_{m-3}(X,I)=c_R(I-P)+c_FP,\qquad
 K_X=(I-P)+k_FP,\quad k_F=\frac{\lambda_0}{g+\lambda_0}.
\]
For every positive $r$, $c_Rk_F-c_F>0$ and $c_F\ge0$. Consequently
\begin{equation}\label{eq:triple-deficit}
 h_{m-3}(X,wI)=wc_R,\qquad
 wc_RK_X-C_{m-3}(X,wI)=w(c_Rk_F-c_F)P.
\end{equation}
The complete maximizing response space is $\ker(X^\top)$, of dimension
$m-d$. Every response with a nonzero fitted component has strictly
smaller normalized risk.
\end{theorem}
\begin{proof}
\par\medskip\noindent\textbf{Triangle classes and positivity.}
For a deleted triple $D_0=\{i,j,l\}$ put
$e=\operatorname{sign}(T_{ij}T_{jl}T_{li})\in\{-1,1\}$.
A diagonal sign switch makes the three off-diagonal entries of its local
Gram $J=E_{D_0}TE_{D_0}^\top$ all equal to $e\mu$.
Thus $J$ has eigenvalues $1+2e\mu,1-e\mu,1-e\mu$.
The retained Gram has eigenvalues $g$ repeated $d-3$ times,
$g-1-2e\mu$ once, and $g-1+e\mu$ twice. Therefore
\begin{equation}\label{eq:triple-det}
 \det(G_S+rI)=z^{d-3}\delta_e,\qquad
 \delta_e=(h-2e\mu)(h+e\mu)^2=h^3-3hq-2e\mu^3.
\end{equation}
The $(i,j)$ entry of $T^2=gT$, multiplied by $T_{ij}$, gives
\[
 \sum_{l\ne i,j}T_{ij}T_{il}T_{jl}=(g-2)q.
\]
Hence the number of type-$e$ triples through every pair, through every
vertex, and in total is respectively
\begin{equation}\label{eq:triple-incidence}
 n_e=\frac{m-2+e(g-2)/\mu}{2},\quad
 \nu_e=\frac{(m-1)n_e}{2},\quad N_e=\frac{m(m-1)n_e}{6}.
\end{equation}
Both classes occur, since
\[
 (m-2)^2q-(g-2)^2=\frac{m^2(d-1)(k-1)}{d^2(m-1)}>0.
\]
With $N=\binom m3$, summing \eqref{eq:triple-det} yields
\begin{equation}\label{eq:triple-law}
 Z_s(r)=Nz^{d-3}D,\qquad
 \mathbb P(S=D_0^c)=\frac{\delta_e}{ND}.
\end{equation}
All triples in one class have the same individual probability.
This probability differs between the two classes.

Every retained Gram is positive definite even at $r=0$. For $k\ge4$,
\[
 (g-1)^2-4q=\frac{k[(k-4)d+k(k-1)]}{d^2(m-1)}>0.
\]
For $k=3$, complete the columns of $X/\sqrt g$ to an orthogonal matrix
and multiply its remaining columns by $\sqrt{m/k}$. This gives a unit-row
ETF in dimension three. Its row projectors have positive definite
Frobenius Gram matrix $(1-q')I+q'\mathbf1\mathbf1^\top$, with $q'<1$,
so they are linearly independent in the six-dimensional space of
symmetric $3$ by $3$ matrices. Therefore $m\le6$. Since $d\ge3$, the
only possible pair is $d=k=3$, where $(g-1)^2-4q=1/5$.
Thus $g-1>2\mu$ in every actual case. In particular $z$,
$h\pm2\mu$, $\delta_e$, $D$, and $h^2-4q$ are positive.
This is the usual Naimark-complement and projector-dimension argument;
it does not assert existence at other formal parameter pairs.

\par\medskip\noindent\textbf{Calibration with the target fixed first.}
Expanding the selected determinants into original-row minors and counting
the supersets of each minor gives
\begin{equation}\label{eq:triple-partition}
 Z_s(X,r)=\sum_{j=0}^d\binom{m-j}{s-j}r^{d-j}e_j(X^\top X).
\end{equation}
At $X^\top X=gI$, the matrix derivative of $\log Z_s$ with respect to
the Gram is scalar. Differentiating instead with respect to each original
row identifies the entire mean map as $f(r)X^\top$. Taking the trace of
its product with $X$ gives
\[
 mf(r)=\mathbb E\operatorname{tr}[G_S(G_S+rI)^{-1}]
       =d-rZ_s'(r)/Z_s(r)=M(r).
\]
Every coefficient in \eqref{eq:triple-partition} is positive, because
$s=m-3\ge d$. Under the distribution on $j$ proportional to those
terms, $M=\mathbb E j$ and
\[
 \frac{dM}{d\log r}=-\operatorname{Var}(j)<0,\qquad
 M(0+)=d,\quad M(+\infty)=0.
\]
Thus $M/m=1/(g+\lambda_0)$ has a unique positive solution for every
fixed positive target. Theorem~\ref{thm:calibration} gives uniqueness
among SPD penalties. In particular $k_F=1-M/d=rH_1/d>0$.

\par\medskip\noindent\textbf{The centered operator has two sectors.}
Put $L=I-T/z$ and $Q_{D_0}=I_3-J/z$. Solving the retained normal equation
gives
\[
 w_S=z^{-1}X^\top y-z_{D_0}(y),\qquad
 z_{D_0}(y)=z^{-1}X_{D_0}^\top Q_{D_0}^{-1}E_{D_0}Ly.
\]
The mean correction is $(z^{-1}-f)X^\top y$. To compute its second moment,
put $a=e\mu$ and define
\begin{align*}
 A_e&=\frac{z^2(1+a-2q)(a-2h)}{(h-2a)^2(h+a)^2},\\
 B_e&=\frac{z^2(h^2+2h-a+2q)}{(h-2a)^2(h+a)^2}.
\end{align*}
Interpolation of $z^2x/(z-x)^2$ at $1+2a,1-a$ gives
$Q_{D_0}^{-1}JQ_{D_0}^{-1}=A_eI_3+B_eJ$.
Its diagonal entries are $A_e+B_e$, and its $(i,j)$ entry is $B_eT_{ij}$.
Using the exact probabilities and the incidences in
\eqref{eq:triple-incidence}, its embedded average is $A_*I+B_*T$, where
\begin{align*}
 A_*&=\frac1{ND}\sum_{e=\pm1}n_e\delta_e
       \left[\frac{m-1}{2}A_e+\frac{m-3}{2}B_e\right],\\
 B_*&=\frac1{ND}\sum_{e=\pm1}n_e\delta_e B_e.
\end{align*}
Subtracting the actual correction mean square yields
\begin{equation}\label{eq:triple-local-operator}
 C=z^{-2}L(A_*I+B_*T)L-(z^{-1}-f)^2T.
\end{equation}
Since $T$ is zero on residuals and equals $gI$ on fitted responses,
this proves the asserted two-sector form. No transitivity of a frame
symmetry group is needed.

The retained spectra and actual class weights give
$H_j=\mathbb E\operatorname{tr}(G_S+rI)^{-j}$, $j=1,2$.
For example,
\begin{align*}
 \delta_e\left[\frac1{h-2a}+\frac2{h+a}\right]&=3(h^2-q),\\
 \delta_e\left[\frac1{(h-2a)^2}+\frac2{(h+a)^2}\right]
 &=\frac{3(h^2-2ha+3q)}{h-2a}.
\end{align*}
Averaging the second line gives
$3[h(h^2-q)+3t]/[(h^2-4q)D]$, proving
\eqref{eq:triple-moments}. In particular $H_2$ is not simply $-H_1'$;
the law changes with $r$.

For the original-coordinate fit map
$F_S=(G_S+rI)^{-1}X_S^\top E_S$, let $O=\mathbb EF_S^\top F_S$.
Then
\[
 \operatorname{tr}O=H_1-rH_2,\qquad
 \operatorname{tr}(PO)=B/g,\qquad C=O-f^2T.
\]
The first identity follows from $F_SF_S^\top=(G_S+rI)^{-1}G_S(G_S+rI)^{-1}$;
the second uses $F_SX=(G_S+rI)^{-1}G_S$.
Dividing the residual and fitted traces by $k$ and $d$ proves
\eqref{eq:triple-sectors}. Actual covariance gives $c_F\ge0$.
The term $gM^2/m^2$ is the necessary fitted mean-square subtraction.

\par\medskip\noindent\textbf{Strict gap and complete equality.}
Appendix~\ref{app:triple-certificate} constructs an integer polynomial
$\mathcal H(r,d,k)$ directly from these moments and proves the identity
\[
 c_Rk_F-c_F=\frac{3d^3r\mathcal H}{m\mathcal Z^2\mathcal E\mathcal F^2},
\]
with all denominator factors positive. After $d=u+3$, $k=v+4$, the
polynomial has 581 positive integer coefficients, including a positive
constant. The appendix gives its defining expression, the exact
division and the common-denominator derivation. The ancillary checker
reconstructs the polynomial before comparing every coefficient.
For the only remaining actual case $d=k=3$, the gap is
\[
 \frac{r(5r^2+10r+4)(5r^3+25r^2+30r+6)}
 {2(r+1)^2(5r^2+10r+1)(5r^2+10r+2)^2}>0.
\]
This is an exact algebraic proof on the whole stated domain.
Finally, for $y=y_R+y_F$ with residual and fitted components, the risk
ratio is
\[
 \frac{w[c_R\lVert y_R\rVert^2+c_F\lVert y_F\rVert^2]}
 {\lVert y_R\rVert^2+k_F\lVert y_F\rVert^2}.
\]
Its deficit from $wc_R$ times the denominator is
$w(c_Rk_F-c_F)\lVert y_F\rVert^2$. The coefficient is strictly positive,
and $k>0$ ensures nonzero residuals. This proves both the sharp value
and every equality case.
\end{proof}

For a concrete unequal-weight example, use the $(16,6)$ frame in
Appendix~\ref{app:etf-witness}. At $r=1/3$ there are 320 positive and
240 negative triangles. Each positive triangle has probability
$49/27680$, while each negative triangle has probability $5/2768$.
Equations~\eqref{eq:triple-moments}--\eqref{eq:triple-sectors} give
\[
 \lambda_0=\frac{28712}{63969},\quad
 c_R=\frac{46863}{885760},\quad
 c_F=\frac{1291585}{1654953984},\quad
 c_Rk_F-c_F=\frac{452910521}{66198159360}>0.
\]
In this example, the displayed $\lambda_0$ is the prescribed target
for every subset. The theorem separately proves calibration for every prescribed
positive target.

Combining the pair, triple, and LOO results gives the same complete
space $\ker(X^\top)$ at $s=m-1,m-2,m-3$, for the same existing frame,
scalar target, and isotropic query in the triple theorem's domain.
Each budget uses its own calibrated penalty. The sharp values need not
agree or be monotone. The three-deletion theorem makes no assertion for
nonisotropic queries or for other budgets.

\section{A same-sample ridge--HT mixture}\label{sec:mixture}
At a balanced design, the strict gap between contrast and mean risk lets us
change the estimator while keeping the same sampled rows and target.
We use this freedom to lower the worst-response risk for every nonzero
positive semidefinite query. Calibration uniqueness fixes the penalty within
the matched ridge family. It does not rule out a different unbiased
estimator on the same subset.
In the balanced model, let $a=n/(n+\lambda_0)$, $\mu=s/d$, and $c_H=a/\mu$.
Each original row has inclusion probability $s/m$. Thus the estimate
\[
 (w_H)_j=c_H\sum_{i\in S\cap j}\varepsilon_i y_i
\]
and the matched ridge estimate $w_R$ both have mean $w_0$ under the
same determinant law. Consider $w(t)=(1-t)w_R+tw_H$.

Put $u_t(k)=(1-t)/(r_s+k)+tc_H$ and $g_t(k)=ku_t(k)$.
Define $\alpha_s(t),\beta_s(t),\tau_s(t)$ by
\eqref{eq:geometry-coefficients}--\eqref{eq:geometry-tau},
replacing $1/(r_s+k)$ by $u_t(k)$
and $f(k)$ by $g_t(k)$. They are the contrast and mean coefficients
in the same covariance decomposition \eqref{eq:block-covariance} for
$w(t)$, and equal the ridge coefficients at $t=0$.
Their formulas are given in \eqref{eq:mixture-coeff} below.

\begin{proposition}[A finite improving mixture]\label{prop:mixture}
Under the assumptions of Theorem~\ref{thm:sharp-geometry}, an explicit
$T_s>0$, depending only on the features, target, and budget, satisfies
the following. For every $0<t\le T_s$ and every nonzero $W\succeq0$,
\begin{align*}
 \sup_{y\ne0}\frac{\operatorname{tr}(W\operatorname{Cov}(w(t)))}{J_0(y)}
 &=\alpha_s(t)\max_jW_{jj}\\
 &<\alpha_s(0)\max_jW_{jj}.
\end{align*}
The complete maximizing raw response space remains $\mathcal E$.
\end{proposition}

\begin{proof}
Fix a balanced design, $d\le s<nd$, and its exactly calibrated
$r=r_s>0$. All expectations below use the actual count law
\eqref{eq:count-law}. Put $a=n/(n+\lambda_0)$,
$b=n(1-a)$, $\mu=s/d$, and $c_H=a/\mu$. Define
\[
 \begin{gathered}
 h(k)=\frac1{r+k},\qquad f(k)=kh(k),\\
 \delta(k)=c_H-h(k),\qquad v(k)=k\delta(k),\\
 h_t(k)=h(k)+t\delta(k),\qquad g_t(k)=f(k)+tv(k).
 \end{gathered}
\]
The mixture is
$(w(t))_j=h_t(K_j)\sum_{i\in S\cap j}\varepsilon_i y_i$.
By symmetry every original row has inclusion probability $s/(nd)$.
Hence $\mathbb E w_H=aq=w_0$, and linearity proves
$\mathbb E w(t)=w_0$ for every real $t$.

\subsection{The covariance quadratics}
Conditional on all counts, the selected contrast sums are independent
across groups, have mean zero, and variance
$K_j(n-K_j)E_j/[n(n-1)]$. Thus, with
$\omega(k)=k(n-k)/[n(n-1)]$,
\begin{align}
 \alpha(t)&=\mathbb E[\omega(K_1)h_t(K_1)^2],\nonumber\\
 \beta(t)&=\frac{\mathbb E[(g_t(K_1)-g_t(K_2))^2]}{2b},\qquad
 \tau(t)=-\frac{\operatorname{Cov}(g_t(K_1),g_t(K_2))}{b},\label{eq:mixture-coeff}\\
 \operatorname{Cov}(w(t))
 &=\alpha(t)\operatorname{diag}(E_j)
      +b\beta(t)\operatorname{diag}(q_j^2)-b\tau(t)qq^\top.\nonumber
\end{align}
This calculation includes the covariance between the ridge and
inclusion-weighted estimates on their shared subset. Their cross
contrast coefficient is $c_H\mathbb E[\omega(K_1)h(K_1)]$;
their mean cross covariance is
$\operatorname{Cov}(f(K_i),c_HK_j)$.
Exchangeability makes that cross covariance symmetric. No independence
between the two estimates is assumed.

For $0\le t\le1$, the score is strictly increasing:
\[
 \frac d{dk}g_t(k)=\frac{(1-t)r}{(r+k)^2}+tc_H>0.
\]
The actual-law conditional ordering proved at the start of
Appendix~\ref{app:geometry} therefore gives $\tau(t)>0$ throughout
this interval.

The positive marginal support of $K_1$ is
$\max(1,s-(d-1)n),\ldots,\min(n,s)$.
It contains two consecutive positive integers: if its upper endpoint is
$n$, then $n-1$ is possible since $s<nd$; otherwise the upper endpoint
is $s\ge2$ and the lower positive endpoint is one.
Consider the size-biased probability $\nu(k)=k\mathbb P(K_1=k)/\mu$.
Calibration says
\[
 \mathbb E_\nu h(K_1)=\frac{\mathbb E f(K_1)}\mu=c_H.
\]
Differentiating the finite sum defining $\alpha$ gives
\begin{equation}\label{eq:negative-derivative}
 \alpha'(0)=-\frac{2\mu}{n(n-1)}
       \operatorname{Cov}_\nu((n-K_1)h(K_1),h(K_1))<0.
\end{equation}
Both functions inside this covariance are strictly decreasing on
positive counts, with derivatives $-(n+r)/(r+k)^2$ and
$-1/(r+k)^2$. The independent-copy covariance identity and the two
distinct positive support points make the covariance strictly positive.

Put
\[
 A=\alpha(0),\quad D=-\alpha'(0)>0,\quad
 Q=\mathbb E[\omega(K_1)\delta(K_1)^2].
\]
Then $\alpha(t)=A-Dt+Qt^2$. Also $Q>0$: if it vanished,
$\delta$ would vanish on every positive-$\omega$ atom, making the
derivative zero in contradiction to \eqref{eq:negative-derivative}.
Similarly,
\begin{align*}
 B_0&=\beta(0),\\
 B_1&=\mathbb E[(f(K_1)-f(K_2))(v(K_1)-v(K_2))]/b,\\
 B_2&=\mathbb E[(v(K_1)-v(K_2))^2]/(2b)\ge0,\\
 \beta(t)&=B_0+B_1t+B_2t^2.
\end{align*}
Theorem~\ref{thm:sharp-geometry} supplies the strict gap $g=A-B_0>0$.
Define the finite step
\begin{equation}\label{eq:mixture-step}
 H_*=D+|B_1|+Q+B_2,\qquad
 T_s=\min\left\{\frac12,\frac{D}{2Q},\frac{g}{2H_*}\right\}>0.
\end{equation}
For $0<t\le T_s$,
\[
 \alpha(t)\le A-Dt/2<A,\qquad
 \alpha(t)-\beta(t)
 \ge g-[D+|B_1|+Q+B_2]t\ge g/2>0,
\]
where $t\le1$ bounds the quadratic terms in the second inequality.
The gap, rather than a limiting continuity argument alone, gives this
explicit interval.

For $c=\max_jW_{jj}>0$, the exact deficit is
\begin{align*}
 \alpha(t)cJ_0-\operatorname{tr}(W\operatorname{Cov}(w(t)))
 ={}&\alpha(t)\sum_j(c-W_{jj})E_j\\
 &+b\sum_j[\alpha(t)c-\beta(t)W_{jj}]q_j^2
    +b\tau(t)q^\top Wq .
\end{align*}
Every mean coefficient is at least $bc[\alpha(t)-\beta(t)]>0$.
Also $\alpha(t)>0$, since an intermediate count has positive probability
and $h_t$ is positive. Equality therefore holds exactly on the same
nonzero contrast space $\mathcal E$. This proves
Proposition~\ref{prop:mixture}, including arbitrary singular and
nondiagonal PSD queries and all ties. For a finite grid of budgets,
$\min_sT_s>0$ supplies one common weight.

\end{proof}
We call $0<t\le T_s$ the protected interval: the contrast risk falls and
the contrast--mean gap stays positive there.
The mixture uses the same sample, labels, target and law. The weight uses
features only. It lowers each query's worst-response risk, but can raise
an individual response's risk. The inclusion-weighted endpoint uses the
Horvitz--Thompson principle \citep{horvitz1952}. Combining dependent
unbiased estimates is a standard construction \citep{lavancier2016}.
Appendix~\ref{app:mixture} gives count formulas and an adverse response.

\subsection{An exact boundary under a mean-share constraint}\label{app:gain-share}
Unrestricted maximizing responses have $q=0$, and hence $w_0=0$.
To examine responses with a nonzero Full mean, define the deterministic
loss share
\[
 \eta(y)=\frac{b\|q\|^2}{J_0(y)},\qquad
 \mathcal C_\delta=\{y\ne0:\eta(y)\ge\delta\},\quad 0<\delta\le1.
\]
This share measures the fraction of Full loss carried by the group means.
It is a response restriction, not a population signal-to-noise ratio.
For each fixed query, let $H_\delta(t,W)$ be the supremum of normalized
mixture risk over $\mathcal C_\delta$.

\begin{proposition}[Mean-share risk and all equality cases]\label{prop:share-risk}
For $0\le t\le1$ and $0\ne W\succeq0$, put
\[
 c=\max_jW_{jj},\quad p_t=\alpha(t)c,\quad
 B_t(W)=\beta(t)\operatorname{diag}(W_{jj})-\tau(t)W,\quad
 \ell_t=\lambda_{\max}(B_t(W)).
\]
The sharp value at exact share $\eta(y)=\delta$ is
$(1-\delta)p_t+\delta\ell_t$, and
\begin{equation}\label{eq:gain-share-value}
 H_\delta(t,W)=\max\{\ell_t,(1-\delta)p_t+\delta\ell_t\}.
\end{equation}
At exact share $0<\delta<1$, all maximizers have $q$ in the complete
top eigenspace of $B_t(W)$, contrasts supported only in groups with
$W_{jj}=c$, and
$\sum_j E_j=(1-\delta)b\|q\|^2/\delta$.
At $\delta=1$ they are nonzero pure means in that eigenspace.
For the lower-bound-share class, $p_t>\ell_t$ forces share $\delta$;
$p_t<\ell_t$ forces share one; and $p_t=\ell_t$ allows every share
in $[\delta,1]$, with these same conditions on each nonzero sector.
For $W=0$, every admissible response has zero risk.
\end{proposition}
\begin{proof}
At fixed share $\eta$, the contrast and mean Rayleigh bounds give
$(1-\eta)p_t+\eta\ell_t$. Both bounds are attained independently:
$n\ge2$ supplies contrasts in every maximal group, and the mean
block has a top eigenvector. Equality in their sum requires equality
in each sector of positive energy. The function is affine in $\eta$.
Its maximum on $[\delta,1]$ and all its equality cases follow by its
slope. These sets are energy-constrained cones, generally not linear spaces.
\end{proof}

\begin{theorem}[Exact all-query improvement phase]\label{thm:share-phase}
Use the feature coefficients and protected interval $0<t\le T=T_s$
in \eqref{eq:mixture-step}. Set
\[
 K_\delta=(1-\delta)D-\delta B_1,\qquad
 L_\delta=(1-\delta)Q+\delta B_2\ge0.
\]
For a fixed protected weight,
\[
 H_\delta(t,W)<H_\delta(0,W)\quad\hbox{for every }0\ne W\succeq0
 \quad\Longleftrightarrow\quad K_\delta>L_\delta t.
\]
Some protected positive weight works exactly when $K_\delta>0$.
In that case one may take $t_\delta=\min\{T,K_\delta/(2L_\delta)\}$
if $L_\delta>0$, and $t_\delta=T$ otherwise.
Every protected weight works exactly when $K_\delta>L_\delta T$.
If $K_\delta\le0$, no positive scalar weight in this mixture family
improves every query, even outside the protected interval.
\end{theorem}
\begin{proof}
On $[0,T]$, $\alpha(t)>\beta(t)$ and $\tau(t)>0$, so
$p_t>\ell_t$ and \eqref{eq:gain-share-value} uses share $\delta$.
Put $\gamma(t)=\beta(t)-d\tau(t)$. Exchangeability gives
\[
 \gamma(t)=\frac{\operatorname{Var}(\sum_jg_t(K_j))}{db}
           =(1-t)^2\gamma(0)\ge0,
\]
since $\sum_jK_j=s$ is fixed. Define
$P_W=\operatorname{diag}(W_{jj})-W/d$ and $R_W=W/d$.
Both are PSD: the Gram representation of $W$ and Cauchy--Schwarz give
$W\preceq d\operatorname{diag}(W_{jj})$.
Thus $B_t(W)=\beta(t)P_W+\gamma(t)R_W$ and $P_W+R_W\preceq cI$.

Write $x=A-\alpha(t)>0$ and $y=\beta(t)-B_0$.
If $y\ge0$, the preceding identity gives
$\ell_t-\ell_0\le yc$ and hence
\[
 H_\delta(0,W)-H_\delta(t,W)
 \ge c\{(1-\delta)x-\delta y\}=ct(K_\delta-L_\delta t).
\]
A flat rank-one query $W=vv^\top$, $v_j\in\{-1,1\}$, attains this
bound: its mean top eigenspace is $v^\perp$ and $\ell_t=\beta(t)$.
This proves necessity and sufficiency when $y\ge0$.

If $y<0$, set $\rho=\max\{\beta(t)/B_0,(1-t)^2\}<1$ when
$\gamma(0)>0$, and $\rho=\beta(t)/B_0<1$ otherwise.
Here $B_0>0$ because unequal counts have positive probability and
$f$ is strictly increasing. We have $B_t(W)\preceq\rho B_0(W)$.
Also $\ell_0\ge(B_0-\tau(0))c>0$, since
$B_0-\tau(0)=((d-1)B_0+\gamma(0))/d>0$.
Consequently $\ell_t<\ell_0$ for every nonzero query. Both sectors
strictly improve, including the pure-mean class $\delta=1$.
The scalar condition holds as well because $(1-\delta)x-\delta y>0$.

The weight conclusions follow from $L_\delta\ge0$.
For the final obstruction use the flat query and one baseline attaining
response of share $\delta$, with mean orthogonal to $v$.
For every real $t>0$ its risk change equals
$-K_\delta t+L_\delta t^2\ge0$. The endpoint supremum is at least
this response's risk. Thus it cannot fall below the baseline supremum.
\end{proof}

A useful uniform window follows directly. Put
$L_+=\max\{0,B_1+B_2T\}$ and $\delta_{\rm safe}=D/(D+2L_+)>0$.
Since $A-\alpha(t)\ge Dt/2$ and
$B_t(W)-B_0(W)\preceq tL_+cI$, every
$0<\delta\le\delta_{\rm safe}/2$ and $0<t\le T$ obeys
\begin{equation}\label{eq:gain-share-window}
 H_\delta(0,W)-H_\delta(t,W)\ge Dtc/4.
\end{equation}
For finitely many budgets, minima give a common positive share window
and weight. This does not imply monotonicity with budget.

For the exact example $n=d=s=2$, $r=1$, $\lambda_0=12/5$,
Appendix~\ref{app:mixture} gives
\[
 \alpha(t)=\frac{(11-t)^2}{1331},\quad
 \beta(t)=\frac{(11+4t)^2}{2178},\quad T=\frac{847}{3116}.
\]
Some positive weight improves every query exactly when $\delta<9/31$.
At $t=1/4$, the exact threshold is $\delta<783/2807$.
At share $1/2$, the flat-query sharp risk instead increases by
$1241/383328$. These follow by substitution in
$-K_\delta t+L_\delta t^2$. They delimit this scalar mixture family;
they do not classify other estimators or perturbed mean-share classes.

\section{Stability under full recalibration}\label{sec:transfer}
We now return to the unrestricted response class $y\ne0$. The mean-share restriction in Section~\ref{sec:mixture} is not imposed here. We first treat arbitrary full-column-rank designs, then return to the balanced designs of Theorem~\ref{thm:sharp-geometry}. Their full equality space has a use when
the exact row structure is broken.
Keep the raw response coordinates, target $A_0$, and physical query $W$
fixed, and replace $X$ by $X+H$. At each design recompute the full fit,
calibrated penalty, sampling law, and selected fit.
The question is whether searching only the old maximizing responses still
finds nearly the worst response for the perturbed design. The answer is
yes to second order, provided we search the whole space and use the new
Full loss.

Use the response matrix and normalized operator from
\eqref{eq:response-operator}--\eqref{eq:normalized-operator} at each design.
All operator norms below are spectral norms.
We first bound the recalibrated operator, then study response-space
restriction, and finally compare the two estimator risks.

\subsection{General-design sensitivity}
\begin{theorem}[Sensitivity of the fully recalibrated procedure]\label{thm:general-sensitivity}
Fix a full-column-rank $X\in\mathbb R^{m\times d}$, an SPD target $A_0$,
a fixed PSD physical query $W$, and $d\le s<m$.
Put $\epsilon=\|HA_0^{-1/2}\|_F$ and recompute the Full fit, calibrated
penalty, determinant law and selected fit at $X+H$.
The explicit constants in \eqref{eq:radius-constants}--\eqref{eq:explicit-eta}
give, for $\epsilon\le r_0$,
\[
 \|M_s(X+H,W)-M_s(X,W)\|\le\eta_s(\epsilon)\le\mathcal M_s\epsilon,
 \qquad |h_s(X+H,W)-h_s(X,W)|\le\eta_s(\epsilon).
\]
For every fixed nonzero raw response,
$e^{-2\epsilon}\le J_0(X+H,y)/J_0(X,y)\le e^{2\epsilon}$.
These constants require no enumeration of row subsets.
\end{theorem}
\begin{proof}
Put $Z=XA_0^{-1/2}$, $V=A_0^{-1/2}WA_0^{-1/2}$ and
$P=A_0^{-1/2}R_sA_0^{-1/2}$. Along the feature segment, use the explicit
Gram box $[\ell,u]$ and constants of Appendix~\ref{app:transfer}.
Calibration has marked-set activation means
$\rho_j=\lambda_j/(1+\lambda_j)$ and odds $q_j=\lambda_j/p_j$.
Its log-coordinate Hessian is $\operatorname{Cov}(\mathbf1_J)$.
Select $s$ slots with weights $1+q_j$ on feature slots and one on the
others, then independently mark a selected feature slot with probability
$q_j/(1+q_j)$. Conditional variance gives
\[
 \operatorname{Cov}(\mathbf1_J)\succeq
 \operatorname{diag}\{\rho_j/(1+q_j)\}\succeq\kappa_sI.
\]
The explicit $\kappa_s>0$ in \eqref{eq:calibration-modulus} follows
from the penalty sandwich and slot inclusion bounds.
Integrating this Hessian bound along a coordinate swap bounds the
spectral divided differences as well. Thus it includes all symmetric
matrix directions and repeated eigenvalues. Inverting the moment
map gives $\|\dot P\|_F\le L_X\epsilon$.

For $F_S=(Z^\top D_SZ+P)^{-1}Z^\top D_S-(I+Z^\top Z)^{-1}Z^\top$
and $T=(I+ZZ^\top)^{1/2}$, ordered differentiation of the ridge fits
and the equation $T\dot T+\dot TT=EZ^\top+ZE^\top$ give
\[
 \|F_ST\|\le B_Ft,\qquad
 \|(F_ST)^{\boldsymbol\cdot}\|\le C_F\epsilon.
\]
The log derivative of each unnormalized determinant weight is at most
$L_p\epsilon$. The normalized endpoint laws therefore have total
variation at most $\tanh(L_p\epsilon/2)$, as proved in
\eqref{eq:law-tv}. Since the risk atom is PSD and at most
$\|V\|B_F^2t^2I$, integrating its derivative and changing its law yield
\[
 \|M_s(X+H)-M_s(X)\|
 \le\|V\|\{2B_FtC_F\epsilon+
           B_F^2t^2\tanh(L_p\epsilon/2)\}=\eta_s(\epsilon).
\]
The Rayleigh principle gives the sharp-value bound.
Finally $K=(I+ZZ^\top)^{-1}$ obeys
$|\partial_x\log(y^\top Ky)|\le2\epsilon$ along the segment,
using $\|K^{1/2}Z\|\le1$ and $\|K^{1/2}\|\le1$.
Integration gives the loss ratio. Appendix~\ref{app:transfer} supplies
the stated constants and each derivative and likelihood bound.
\end{proof}

\subsection{Restriction to the old complete maximizing space}
\begin{theorem}[Recalibrated transfer]\label{thm:transfer}
Let $X_0$ be a balanced design as in Theorem~\ref{thm:sharp-geometry},
and fix nonzero $W\succeq0$. Let $\mathcal E$ be its complete maximizing
raw response space. For each $d\le s<nd$, there are explicit positive
constants $r,L_s,v,\gamma_s$ such that, for
$\epsilon=\lVert HA_0^{-1/2}\rVert_F\le r$, all recalibrated objects
exist and
\begin{equation}\label{eq:operator-lipschitz}
 \lVert M_s(X_0+H)-M_s(X_0)\rVert\le\eta_s(\epsilon)
 \le L_s\epsilon.
\end{equation}
Here $\gamma_s$ is the gap below the full top eigenspace at $X_0$.
If $\eta_s<\gamma_s/2$ and $\omega=v\epsilon<1$, put
\[
 \theta_s=\frac{\eta_s}{\gamma_s-\eta_s}+\frac{\omega}{1-\omega}.
\]
Then $|h_s(X_0+H)-h_s(X_0)|\le\eta_s$ and
\begin{align}\label{eq:raw-regret}
 0\le{}&h_s(X_0+H)
 -\sup_{0\ne y\in\mathcal E}
 \frac{y^\top C_s(X_0+H)y}{y^\top K_{X_0+H}y}\nonumber\\
 \le{}&(h_s(X_0)+\eta_s)\min\{1,\theta_s^2\}.
\end{align}
One positive radius makes these conditions hold at every stated budget.
\end{theorem}

The normalized coordinate is $u=K_X^{1/2}y$, so keeping the raw space
$\mathcal E$ fixed means searching $K_X^{1/2}\mathcal E$ at the new design.
The first term in $\theta_s$ controls movement of the top eigenspace;
the second controls this change of normalization.
The right side of \eqref{eq:raw-regret} is $O(\epsilon^2)$ at a fixed
base. The optimum itself can move by $O(\epsilon)$. The response must
be reoptimized within the entire old space, including all ties, using
the perturbed denominator. A selected old maximizing vector need not
have quadratic regret.

\subsection{The balanced gap and the unchanged raw space}
At the balanced base, contrasts in group $j$ are normalized eigenvectors
with eigenvalue $\alpha_sW_{jj}$, while the normalized mean block is
\[
 B_{\mathrm{mean},s}=\beta_s\operatorname{diag}(W_{jj})-\tau_sW.
\]
To check this normalization, use the orthonormal signed group-mean
coordinates $\sqrt n\,q$. The loss matrix equals $b/n$ on that sector,
so division by the loss changes its risk block into the matrix above.
On all contrast directions, $K_{X_0}=I$. Thus the full top eigenspace
is precisely the raw space $\mathcal E$ in \eqref{eq:old-space}.
Its gap is explicitly
\begin{equation}\label{eq:base-gap}
 \gamma_s=\alpha_sc-
 \max\left\{\alpha_s\max_{j\notin I_{\max}}W_{jj},
             \lambda_{\max}(B_{\mathrm{mean},s})\right\}>0,
\end{equation}
where an absent contrast term is omitted. Positivity follows from
$B_{\mathrm{mean},s}\preceq\beta_scI$ and $\alpha_s>\beta_s$.
In particular a lower bound is the minimum of
$c(\alpha_s-\beta_s)$ and any present positive contrast gap.

Let $M=M_s(X_0)$, $M'=M_s(X_0+H)$, $\eta=\eta_s$, and
$h'=\lambda_{\max}(M')$. The variational characterization gives
$|h'-h_s(X_0)|\le\eta$. If $u'$ is a unit top eigenvector of $M'$
and $P$ is the orthogonal projector onto $\mathcal E$, projecting its
eigenvalue equation onto $\mathcal E^\perp$ gives
\[
 \lVert(I-P)u'\rVert\le\frac{\eta}{\gamma_s-\eta}.
\]
Indeed the spectrum of $M$ on that complement is at most
$h_s(X_0)-\gamma_s$, while $h'\ge h_s(X_0)-\eta$.
This estimate applies to every new top eigenvector, regardless of ties.

The normalized response coordinate is $u=K_X^{1/2}y=T^{-1}y$.
Thus the unchanged raw space $\mathcal E$ has normalized image
$(T')^{-1}\mathcal E$ at the endpoint. The inverse identity and $T,T'\succeq I$
give
\[
 \lVert(T')^{-1}-T^{-1}\rVert
 \le\lVert T'-T\rVert\le v\epsilon=\omega.
\]
Since $T^{-1}=I$ on $\mathcal E$, the map
$(T')^{-1}|_{\mathcal E}$ differs from the identity by at most $\omega$.
Its least singular value is at least $1-\omega$. Resolving its parallel
and perpendicular components shows that the projector onto
$(T')^{-1}\mathcal E$ differs from $P$ by at most
$\omega/(1-\omega)$ (the two spaces have the same dimension).
Thus a unit new maximizer has distance at most
\[
 \theta_s=\frac{\eta}{\gamma_s-\eta}+\frac{\omega}{1-\omega}
\]
from that endpoint normalized space.

If its projection is nonzero, normalize this projection to obtain $v'$.
For a top eigenvector $u'$ of a symmetric positive semidefinite $M'$,
\[
 h'-(v')^\top M'v'
 \le(h'-\lambda_{\min}(M'))
          [1-|\langle u',v'\rangle|^2]
 \le h'\theta_s^2.
\]
If the projection is zero use the trivial upper bound $h'$. Optimizing
over the endpoint normalized image is exactly optimizing over unchanged
raw $y\in\mathcal E$ with the endpoint loss denominator.
Since $h'\le h_s(X_0)+\eta$, this proves \eqref{eq:raw-regret}.

At the balanced base $\sigma=z=\sqrt{n/\lambda_0}$, so the constants
above are explicit. Over the finite grid $d\le s<nd$, the radius
\[
 r_*=\min_{d\le s<nd}
 \left\{r_0,\frac1{2v},\frac{\gamma_s}{4\mathcal M_s}\right\}>0
\]
gives $\omega\le1/2$ and $\eta_s\le\gamma_s/4$ at every size.
It proves the asserted common local order. No uniformity is asserted
as dimension grows, the target penalty vanishes, or a query gap closes.

\subsection{A larger space with a query-uniform radius}
The old space can have a small external gap when two query diagonals are
nearly tied. Retaining more contrast groups removes this dependence.
For $I\subseteq\{1,\ldots,d\}$ containing every maximal query group,
let $\mathcal E_I$ be the unchanged raw contrasts supported on $I$.
It has dimension $(n-1)|I|$. The base top-to-complement gap is at least
\[
 \gamma_I=\min\{\alpha_s(c-\max_{j\notin I}W_{jj}),
                       (\alpha_s-\beta_s)c\}>0,
\]
with the first term omitted if $I$ contains every group.
This follows from the exact contrast eigenvalues and
$\beta_s\operatorname{diag}(W_{jj})-\tau_sW\preceq\beta_scI$.
For $\mathcal E_{\rm all}=\ker X_0^\top$ it gives
$\gamma_{\rm all}=(\alpha_s-\beta_s)c$.
For $I_\vartheta=\{j:W_{jj}\ge(1-\vartheta)c\}$, $0<\vartheta\le1$,
it gives $\gamma_I\ge c\min\{\alpha_s\vartheta,\alpha_s-\beta_s\}$.

\begin{theorem}[Enlarged contrast-space transfer]\label{thm:enlarged}
Keep the balanced base, target and fixed physical query, and recalibrate
exactly at $X_1=X_0+H$. Write $K_1=K_{X_1}$,
$M_1=K_1^{-1/2}C_s(X_1,W)K_1^{-1/2}$ and $h_1=\lambda_{\max}(M_1)$.
Use the finite surrogate $\widetilde M$ and error $e=e_{\rm end}$ from
\eqref{eq:common-surrogate}--\eqref{eq:endpoint-error-bound}, or any valid
symmetric surrogate with $\|M_1-\widetilde M\|\le e$.
Let $U$ and $Q$ be orthonormal bases of $K_1^{1/2}\mathcal E_I$ and its
orthogonal complement. Define
\[
 a=\lambda_{\max}(U^\top\widetilde M U),\quad
 d_0=\lambda_{\max}(Q^\top\widetilde M Q),\quad
 b=\|Q^\top\widetilde M U\|,\quad g=a-d_0-2e,\quad B=b+e.
\]
If $h_{1,I}$ is the supremum over nonzero raw $y\in\mathcal E_I$ with
the endpoint loss denominator, then, when $g>0$,
\begin{equation}\label{eq:enlarged-regret}
 0\le h_1-h_{1,I}\le
 \frac{2B^2}{\sqrt{g^2+4B^2}+g}\le\frac{B^2}{g}.
\end{equation}
At every endpoint an alternative bound is
$h_1-h_{1,I}\le\lambda_{\max}(\widetilde M)-a+2e$.
Every response of true sharp deficit $h_1-r_1(y)\le\rho$, where
$r_1(y)=y^\top C_s(X_1,W)y/(y^\top K_1y)$, obeys, if $g>0$,
\begin{equation}\label{eq:enlarged-leakage}
 \min_{v\in\mathcal E_I}
 \sqrt{\frac{(y-v)^\top K_1(y-v)}{y^\top K_1y}}
 \le\min\left\{1,\frac{B+\sqrt{B^2+g\rho}}{g}\right\}.
\end{equation}
For an exact maximizer the sharper bound is $B/\sqrt{g^2+B^2}$.
For the full contrast space, one explicit positive feature radius works
for all nonzero PSD queries after normalization $\max_jW_{jj}=1$, and
all budgets $d\le s<m$. On that ball the regret is $O(\epsilon^2)$ and
the leakage is $O(\epsilon+\sqrt\rho)$, uniformly over those queries and
budgets. The same holds for each fixed threshold $\vartheta>0$ with
its query-dependent space $\mathcal E_{I_\vartheta}$.
\end{theorem}
\begin{proof}
Let $A=\lambda_{\max}(U^\top M_1U)$,
$D=\lambda_{\max}(Q^\top M_1Q)$ and $b_1=\|Q^\top M_1U\|$.
Compression and the variational principle give $A-D\ge g$ and $b_1\le B$.
For a unit vector $x=Uu+Qv$,
\[
 x^\top M_1x\le A\|u\|^2+D\|v\|^2+2b_1\|u\|\|v\|.
\]
The largest eigenvalue of the corresponding two-by-two matrix, minus
$A$, is $2b_1^2/(\sqrt{(A-D)^2+4b_1^2}+A-D)$.
It decreases with the gap and increases with the cross norm. This proves
\eqref{eq:enlarged-regret}, the classical block Rayleigh bound
\citep{lili2005}. Weyl's inequality proves the alternative bound.
Here $A=h_{1,I}$ because the trial space is $K_1^{1/2}\mathcal E_I$.

For leakage put $t=\|v\|$. Since $h_1\ge A$,
\[
 h_1-x^\top M_1x\ge(A-D)t^2-2b_1t\sqrt{1-t^2}\ge gt^2-2Bt.
\]
Solving this quadratic proves \eqref{eq:enlarged-leakage}.
For a top eigenvector, its complementary eigen-equation gives
$(h_1-D)t\le b_1\sqrt{1-t^2}$, proving the sharper bound, including ties.
A surrogate deficit at most $r$ implies true deficit at most $r+2e$.
Thus the leakage statement also has an evaluable surrogate test.

For the uniform claim, Appendix~\ref{app:enlarged} proves
$\|M_1-M_0\|\le Lh$ and $e\le\bar e(h)=O(h^2)$, with
$h=\epsilon+\|P_1-P_0\|$ and constants independent of normalized $W$.
It bounds the moving-space projector by $f(h)=O(h)$ and proves
\[
 g\ge\gamma-(\alpha_s+\gamma)f(h)^2-2Lh-4\bar e(h),\quad
 B\le Lh+2\alpha_sf(h)+2\bar e(h).
\]
The explicit dyadic choice \eqref{eq:enlarged-dyadic} makes
$g\ge\gamma/2>0$ on a whole ball. The full-space gap is
$\gamma=\alpha_s-\beta_s$, or use the threshold gap above.
Exact calibration sensitivity gives $h\le(1+L_{X,s})\epsilon$.
The finite minimum \eqref{eq:enlarged-feature-radius} then proves the
common feature radius, quadratic regret and leakage orders.
\end{proof}

This argument uses the top value in the retained space, not its lowest
value. Some retained contrast eigenvalues may lie below mean eigenvalues.
There is no assertion that the entire retained space is a separated
spectral cluster. The full contrast space has dimension $m-d$ and need
not be small. Its common radius removes dependence on internal query
gaps; it does not make the sharp risk itself vary quadratically.

\subsection{The mixture gain under the same perturbation}
The preceding results approximate a response search. We now bound the
change of the estimator gain itself. This is a separate argument.
At $X=X_0+H$, use the exactly calibrated ridge fit and the same-subset
inclusion-weighted estimate
\[
 w_H=L_0(X)\operatorname{diag}(\mathbf1_S/\pi(X))y,\qquad
 L_0(X)=(X^\top X+\lambda_0I)^{-1}X^\top.
\]
Here $\pi_i(X)$ is the actual inclusion probability under the recalibrated
law. Both means equal the new Full fit. Let $h_t(X,W)$ be the mixture's
sharp risk over unrestricted $y\ne0$, using the new Full loss; put
$h_R=h_0$. Keep the base feature-selected weight interval $0<t\le T_s$.

\begin{theorem}[A local bound on the mixture gain]\label{thm:gain}
Fix a balanced base, target and budget from Proposition~\ref{prop:mixture}.
Write $\alpha(t)=A-Dt+Qt^2$ and $T=T_s$.
There are explicit $r_1,C>0$, independent of query and weight, such that
for $\epsilon=\|H/\sqrt{\lambda_0}\|_F\le r_1$, every nonzero PSD $W$
and every $0<t\le T$ satisfy, with $c=\max_jW_{jj}$,
\begin{equation}\label{eq:gain-modulus}
 |h_R(X,W)-h_t(X,W)-t(D-Qt)c|\le Ct\epsilon c.
\end{equation}
Thus $\epsilon\le\min\{r_1,D/(4C)\}$ gives gain at least $Dtc/4>0$.
\end{theorem}
The factor $t$ comes from the coupled same-subset operators and their
base deficits. Two separate Lipschitz bounds on the sharp risks would
lose this factor. The following proof includes both comparisons.
\subsection{The coupled operator bound}
Use the physical coefficients $A,D,Q,B_0,B_1,B_2$ and $T=T_s$ from
Section~\ref{sec:mixture}; set $g=A-B_0>0$ and
$\tau(t)=\tau_0+\tau_1t+\tau_2t^2$. The coefficients of $\tau$ are
obtained by expanding \eqref{eq:mixture-coeff}. Explicitly, with $f,v,b$
from Section~\ref{sec:mixture},
\begin{align*}
 \tau_0&=-\operatorname{Cov}(f(K_1),f(K_2))/b,\qquad
 \tau_2=-\operatorname{Cov}(v(K_1),v(K_2))/b,\\
 \tau_1&=-\{\operatorname{Cov}(f(K_1),v(K_2))+
                  \operatorname{Cov}(v(K_1),f(K_2))\}/b.
\end{align*}
These are physical coefficients; they are not divided by $\lambda_0$.

Set $Z_0=X_0/\sqrt{\lambda_0}$, $E=H/\sqrt{\lambda_0}$, and
$Z_x=Z_0+xE$. In \eqref{eq:radius-constants}--\eqref{eq:composite-constants}
take $A_0=\lambda_0I$ and evaluate all constants at $X_0$. To avoid
confusing the metric bound there with the mixture weight, denote that
bound by $\theta=\sqrt{1+v^2}$, and set
\[
 q_* = L_p,\qquad r_1=\min\{r_0,(2q_*)^{-1}\},\qquad
 \pi_*=(s/m)e^{-1}.
\]
Equation~\eqref{eq:penalty-derivative} gives
$P_x=R_s(X_0+xH)/\lambda_0\succeq pI$ and
$\|\dot P_x\|_F\le L_X\epsilon$ throughout this ball.
Let $p_x(S)$ be the normalized subset probability. The log-determinant
bound in Appendix~\ref{app:transfer} implies
\[
 |\partial_x\log p_x(S)|\le2q_*\epsilon,\qquad
 |\dot\pi_i|=|\mathbb E_x[\mathbf1_{i\in S}\partial_x\log p_x(S)]|
       \le2q_*\epsilon\pi_i.
\]
All inclusions are positive. Their base values are $s/m$, so integration
and differentiation of the reciprocal give
\begin{equation}\label{eq:gain-inclusions}
 \pi_i(x)\ge(s/m)e^{-2q_*\epsilon}\ge\pi_*,\qquad
 |\partial_x\pi_i^{-1}|\le2q_*\epsilon/\pi_*.
\end{equation}
This step uses the changing inclusions, rather than their base values.

Write $D_S=\operatorname{diag}(\mathbf1_S)$ and
\begin{gather*}
 A_S=Z_x^\top D_SZ_x+P_x,\quad L_S=A_S^{-1}Z_x^\top D_S,\quad
 L=(I+Z_x^\top Z_x)^{-1}Z_x^\top,\\
 I_S=\operatorname{diag}(\mathbf1_S/\pi),\quad
 F_R=L_S-L,\quad F_H=L(I_S-I),\quad G=F_H-F_R.
\end{gather*}
The physical output maps are these maps divided by $\sqrt{\lambda_0}$.
The fit bounds from Appendix~\ref{app:transfer} and
\eqref{eq:gain-inclusions} yield the following atom and derivative bounds:
\begin{align*}
 B_R&=(2\sqrt p)^{-1}+1/2,&
 D_R&=5/(4p)+5/4+L_X/(2p^{3/2}),\\
 B_H&=(2\pi_*)^{-1},&D_H&=(5/4+q_*)/\pi_*,\\
 B_G&=B_R+B_H,&D_G&=D_R+D_H.
\end{align*}
Here $\|F_R\|\le B_R$, $\|\dot F_R\|\le D_R\epsilon$, and likewise
for $F_H,G$. For example, $\|I_S-I\|\le1/\pi_*$,
$\|L\|\le1/2$, and $\|\dot L\|\le5\epsilon/4$ give the $H$ bounds.
With $T_x=(I+Z_xZ_x^\top)^{1/2}$, the Sylvester calculation in
Appendix~\ref{app:transfer} gives $\|T_x\|\le\theta$ and
$\|\dot T_x\|\le v\epsilon$. Hence, for $U_R=F_RT_x$, $U_G=GT_x$, put
\[
 a_R=\theta B_R,\quad d_R=\theta D_R+vB_R,\qquad
 a_G=\theta B_G,\quad d_G=\theta D_G+vB_G.
\]
These bound $U_R,U_G$ and their derivatives, with a factor $\epsilon$
on each derivative bound.

Exact calibration and inclusion weighting give the same mean $Ly$
in target coordinates. Their shared-subset covariance therefore gives
\begin{equation}\label{eq:gain-coupling}
 M_t=M_R+t\mathcal J_t,\qquad
 \mathcal J_t=\mathbb E_x[U_R^\top VU_G+U_G^\top VU_R+tU_G^\top VU_G],
 \quad V=W/\lambda_0.
\end{equation}
This polynomial identity also defines $\mathcal J_0$ without division by zero.
Since $\|V\|\le dc/\lambda_0$, define
\begin{align}
 L_R&=\frac d{\lambda_0}(2a_Rd_R+2q_*a_R^2),\nonumber\\
 L_D&=\frac d{\lambda_0}\{2(d_Ra_G+a_Rd_G)+2Ta_Gd_G
                         +2q_*(2a_Ra_G+Ta_G^2)\},\qquad
 L_M=L_R+TL_D.\label{eq:gain-basic-constants}
\end{align}
Differentiating the finite expectation in \eqref{eq:gain-coupling}
differentiates both its atom and its probability. The first terms in
\eqref{eq:gain-basic-constants} bound atom derivatives; the terms with
$q_*$ bound the normalized score contribution. The cross atom may be
indefinite, so we use its full norm. Integration gives, simultaneously
for $0\le t\le T$,
\begin{equation}\label{eq:gain-three-bounds}
 \|E_A\|\le L_R\epsilon c,\quad \|E_B\|\le L_M\epsilon c,\quad
 \|E_B-E_A\|\le tL_D\epsilon c,
\end{equation}
where $E_A=M_R(X,W)-M_R(X_0,W)$ and
$E_B=M_t(X,W)-M_t(X_0,W)$. In particular the last bound retains $t$.

\subsection{Comparing two moving sharp risks}
We give the elementary spectral step explicitly. Suppose symmetric
operators $\mathsf A,\mathsf B$ have largest eigenvalues $a,b$, and write
$H_A=aI-\mathsf A$, $H_B=bI-\mathsf B$, and $N=H_A-H_B$. If
\[
 -ktH_A\preceq N\preceq ktH_A,\qquad
 -ktH_B\preceq N\preceq ktH_B,
\]
then arbitrary symmetric changes $E_A,E_B$ satisfy
\begin{equation}\label{eq:gain-deficit-lemma}
 |\lambda_{\max}(\mathsf A+E_A)-\lambda_{\max}(\mathsf B+E_B)-(a-b)|
 \le\|E_B-E_A\|+2kt\max\{\|E_A\|,\|E_B\|\}.
\end{equation}
Indeed, any unit top eigenvector $u$ of $\mathsf A+E_A$ obeys
$u^\top H_Au\le2\|E_A\|$ by the Rayleigh principle. Evaluate
$\mathsf B+E_B=\mathsf A+E_A-(a-b)I+N+E_B-E_A$ at $u$ for one bound.
Evaluate the reversed equality at a unit top eigenvector of
$\mathsf B+E_B$, using $v^\top H_Bv\le2\|E_B\|$, for the other.
This argument allows multiple top eigenvalues.

Apply this to the two base operators. Their top values are $Ac$ and
$\alpha(t)c$, where $\alpha(t)=A-Dt+Qt^2$. On contrast group $j$,
\[
 H_A=A(c-W_{jj})I,\quad H_B=\alpha(t)(c-W_{jj})I,\quad
 N=t(D-Qt)(c-W_{jj})I.
\]
Moreover $\alpha(t)\ge(1-T)^2A$: its defining square contains the
nonnegative summands $(1-t)h(k)$ and $tc_H$. On the mean sector,
$H_A\succeq gcI$ and $H_B\succeq(g/2)cI$ by
Section~\ref{sec:mixture}. The mean part of $N$ has norm at most
$t(D+m_*)c$, where
\[
 m_*=|B_1|+TB_2+d(|\tau_1|+T|\tau_2|),\qquad
 k=\max\left\{\frac D{(1-T)^2A},\frac{2(D+m_*)}g\right\}.
\]
Thus both comparisons required for \eqref{eq:gain-deficit-lemma} hold.
A tied contrast diagonal makes all three corresponding deficits zero;
no lower bound on a query diagonal gap is used. Equations
\eqref{eq:gain-three-bounds} and \eqref{eq:gain-deficit-lemma} prove
Theorem~\ref{thm:gain} with
\[
 C=L_D+2kL_M,\qquad r_{\rm gain}=\min\{r_1,D/(4C)\}.
\]
Since $D-Qt\ge D/2$ on $[0,T]$, this radius gives the stated gain.
All constants depend only on the fixed base, target, and budget.

\subsection{Exact finite neighborhoods}
The explicit constants above prove a positive local radius.
Appendix~\ref{app:gain} sharpens them by transporting the calibration
Jacobian and the probability-weighted derivative energies from the base.
This certifies a radius of $1/1000$ in $\|H/\sqrt{\lambda_0}\|_F$ for
three specified matched bases:
\[
 (n,d,s;r,\lambda_0)=(2,2,2;2/3,5/3),\quad
 (2,2,3;5/4,55/31),\quad (2,3,3;3/2,363/101).
\]
For every point in each closed ball, every nonzero PSD query and every
$0<t\le T_s$, the two-sided modulus \eqref{eq:gain-modulus} and the
positive gain $Dtc/4$ hold. Table~\ref{tab:gain-certificates} states the
exact weight endpoints and strict rational gate bounds. Its proof gives
all transport formulas and rational enclosure rules. The finite coefficient
and energy records are reconstructible by the supplied ancillary scripts.
These statements assume exact endpoint calibration. They certify these
three neighborhoods, not approximate calibration software or prediction
performance on a data distribution.

\section{Scope and open questions}\label{sec:conclusion}
Established mean identities and duality give exact target calibration
at a fixed row budget. The feasibility threshold and unique matched penalty
provide the basis for covariance comparisons.
The main analysis proves strict sector inequalities and resolves response
equality in the stated design regimes. Complete response geometry then
answers a stronger question than a matrix ceiling alone.
In the stated exact regimes, it identifies the sharp risk and all responses
that attain it.

The exact regimes have distinct boundaries.
General LOO geometry gives an attainment test, but does not compute the
smaller optimum when its envelope is strict.
The balanced theorem covers all budgets from the dimension upward.
Its lower-budget phase example prevents a blanket extension below that range.
The ETF pair theorem characterizes flat queries.
The triple theorem currently requires an isotropic query.
Neither implies an all-budget ETF classification or optimality over designs.

These results leave both conceptual and extension questions.
Can structural inequalities replace the longer coefficient certificates?
At three ETF deletions, does flat query energy suffice beyond isotropic queries?
The constructive conclusions suggest further questions.
Can one classify general all-budget sharp risk without replication or frame
structure? Which nonflat ETF queries admit residual maximizers?
Can the exact mean-share phase be transported through full recalibration?
Can useful query-uniform radii be obtained without retaining all base contrasts?
The present results isolate the operators, gaps and equality conditions
that these extensions must address.

\section*{AI use statement}
We used generative AI tools to assist with mathematical exploration,
formulating and developing proofs, literature search, experiment design,
implementation, and interpretation of results. We also used these tools
to draft and edit text. The workflow included separate proof checks and
independent numerical checking implementations. The author is responsible
for the final claims, proofs, references, code, and presentation.

\clearpage
\appendix
\section{The calibrated LOO coefficient-matrix ceiling}\label{app:calibration-support}
We prove Proposition~\ref{prop:loo-matrix-contact}.
Keep its calibrated $R$, put $H=G+R$, and define
\[
 w=H^{-1}X^\top y,\qquad e=y-Xw,\qquad
 J_R=\lVert e\rVert^2+w^\top Rw,\qquad
 a_i=H^{-1/2}x_i,\qquad c_i=1-\lVert a_i\rVert^2=1-\ell_i.
\]
Here $w$ is the full-data fit at the raw calibrated penalty $R$.
It is distinct from the prescribed target $w_0$.
Also put $U=XH^{-1/2}$ and
$V=H^{1/2}\Sigma_{m-1}(y)H^{1/2}$.
Since $R\prec A_0$ by \eqref{eq:penalty-upper}, inverse order gives
\begin{equation}\label{eq:matrix-target-slack}
 J_0-J_R=y^\top X\{H^{-1}-(G+A_0)^{-1}\}X^\top y\ge0,
 \qquad J_0=J_R\ \Longleftrightarrow\ X^\top y=0.
\end{equation}
The inverse difference is positive definite. The deletion identities
\eqref{eq:loo-deletion}, followed by subtraction of the mean square,
give
\[
 V=\frac1D\sum_i\frac{e_i^2}{c_i}a_ia_i^\top
   -\frac1{D^2}(U^\top e)(U^\top e)^\top.
\]
Thus the exact whitened deficit is
\begin{align}\label{eq:matrix-contact-decomposition}
 J_0\frac{t_\ell}{D}I-V
 ={}&\frac{t_\ell}{D}(J_0-J_R+w^\top Rw)I\nonumber\\
 &+\frac1D\sum_i e_i^2
       \left(t_\ell I-\frac{a_ia_i^\top}{c_i}\right)
   +\frac1{D^2}(U^\top e)(U^\top e)^\top.
\end{align}
Each term is positive semidefinite: $t_\ell>0$ by full column rank,
$c_i>0$, and the sole nonzero eigenvalue of $a_ia_i^\top/c_i$ is
$\ell_i/(1-\ell_i)\le t_\ell$.
Congruence proves the ceiling.

If $w\ne0$, then $w^\top Rw>0$, so the first term of
\eqref{eq:matrix-contact-decomposition} is positive definite and its
kernel is zero. If $w=0$, then $X^\top y=0$, $e=y$, $J_0=J_R$,
and the first and last terms vanish. For each active index $y_i\ne0$,
the bracket has a nonzero kernel only if $\ell_i/(1-\ell_i)=t_\ell$.
Its kernel is then $\operatorname{span}(a_i)$.
A common nonzero kernel therefore requires every active row to have
maximal leverage and all active $a_i$ to lie on one line.
Equal leverage on that line means equal nonzero lengths, so these
vectors are signed copies of one vector $a$.
Multiplication by $H^{1/2}$ makes the original active rows signed copies
of $v=H^{1/2}a$. Conversely, these conditions make every active bracket
vanish on $a$, and their common kernel is exactly its span.
Finally, if $\Delta=J_0(t_\ell/D)H^{-1}-\Sigma_{m-1}(y)$, then
\[
 \ker\Delta=H^{1/2}\ker(H^{1/2}\Delta H^{1/2}).
\]
The raw coefficient query kernel is therefore precisely
$\operatorname{span}(v)$, as claimed.

\section{Exact balanced sign algebra}\label{app:geometry}
This appendix proves the strict sign for every positive raw penalty $r$.
Calibration then chooses $r=r_s$, with the same fixed target at each
budget. In this appendix only, put
$\rho=\mathbb E[K_1/(r+K_1)]$ and $b=n(1-\rho)$.
The induced target is $\lambda_0=n(1-\rho)/\rho>0$.

We first prove $\tau>0$ and reduce $\alpha>\beta$ to a reciprocal-moment
inequality. The strict comparison has three cases: $n=2$ in any dimension
(Appendix~\ref{app:two-replicas}); $d=2,3$ for the remaining replica counts
(Appendix~\ref{app:small-dimensions}); and $n\ge3,d\ge4$
(Appendix~\ref{app:general-sign}). In the last case, two reciprocal lower
bounds cover the root domain. Polynomial certificates prove the root signs
and the monotonicity needed to pass from the root to the actual mean.
The LOO endpoint has a separate direct formula. The final deficit identifies
all equality responses.

\subsection{Covariance and the sign of the cross term}
Conditional on the counts, subsets in different groups are independent
and uniform. If $T_j$ is the sum of selected signed contrasts in group
$j$, then
\[
 \mathbb E[T_j\mid K]=0,\qquad
 \operatorname{Var}(T_j\mid K)=\frac{K_j(n-K_j)}{n(n-1)}E_j,\qquad
 (w_S)_j=f(K_j)q_j+\frac{T_j}{r+K_j}.
\]
Total covariance and exchangeability prove
\eqref{eq:block-covariance}, with the coefficients
\eqref{eq:geometry-coefficients}. In particular, $\beta$ is the
transverse eigenvalue of the count-score covariance divided by $b$;
it is not its diagonal variance divided by $b$.

We give a pairwise dependence argument that will also serve the mixture.
Let $a_k=\binom nk(r+k)$ and
\[
 g(x)=\sum_{k=0}^na_kx^k=(1+x)^{n-1}[r+(r+n)x].
\]
For $d\ge3$, put $H_\ell=[x^\ell]g(x)^{d-2}$. This polynomial has
only negative real roots and positive coefficients on its full support
$0,\ldots,(d-2)n$. Newton's coefficient inequality implies ordinary
log-concavity of $(H_\ell)$. The conditional mass of $K_2=\ell$ given
$K_1=k$ is proportional to $a_\ell H_{s-k-\ell}$. For adjacent possible
$k,k+1$, the ratio of the new to old conditional masses is proportional to
\[
 \frac{H_{s-k-1-\ell}}{H_{s-k-\ell}}.
\]
Log-concavity makes this ratio nonincreasing in $\ell$. Both support
endpoints $\max(0,s-k-(d-2)n)$ and $\min(n,s-k)$ move down as $k$
increases. At a new lower point the ratio is infinite; at a lost upper
point it is zero. Thus the ordering includes the support boundaries:
the conditional law of $K_2$ decreases stochastically with $k$.
For $d=2$, $K_2=s-K_1$ gives the same conclusion directly.

These conditional laws are distinct because their means are
$(s-k)/(d-1)$. Consequently, for a strictly increasing function $v$,
$\mathbb E[v(K_2)\mid K_1=k]$ strictly decreases on the possible $k$.
For any strictly increasing $u$, the independent-copy identity
\[
 2\operatorname{Cov}(u(K_1),v(K_2))
 =\mathbb E[(u(K_1)-u(K_1'))
  (\mathbb E[v(K_2)\mid K_1]-\mathbb E[v(K_2)\mid K_1'])]
\]
is strictly negative when $1\le s<nd$, since $K_1$ is nonconstant.
Here the prime denotes an independent copy. Taking $u=v=f$ proves
$\tau>0$. This argument uses the actual determinant count law.

\subsection{Deletion identities and a reciprocal reduction}\label{app:deletion}
Write $M=d(n-1)$ and define
\begin{align*}
 Z&=[x^s]g(x)^d,\\
 Q&=[x^s](1+x)^{M+1}[r+(r+n)x]^{d-1},\\
 T&=[x^{s-1}](1+x)^M[r+(r+n)x]^{d-2},\\
 A&=[x^{s-1}]g(x)^{d-1}
       \sum_{j=0}^{n-2}\binom{n-2}{j}\frac{x^j}{r+1+j}.
\end{align*}
All four are positive for $d\le s<nd$. Direct deletion of determinant
factors gives
\begin{equation}\label{eq:deletion-coefficients}
 \alpha=\frac AZ,\qquad
 \beta=\frac{nT-(n-1)A}{(n+r)Q}.
\end{equation}
For details, put $R_1=Z\mathbb E f(K_1)$,
$C_{12}=Z\mathbb E[f(K_1)f(K_2)]$, and
$D_1=Z\mathbb E[K_1/(r+K_1)^2]$. The binomial identities give
\[
 Z-R_1=rQ,\qquad (n+r)D_1-R_1=n(n-1)A,\qquad
 nR_1-(n+r)C_{12}=n^2rT.
\]
The transverse covariance numerator is $R_1-rD_1-C_{12}$, while
$b=nrQ/Z$. These prove \eqref{eq:deletion-coefficients}.

For an auxiliary proof experiment, select $s-1$ items with product
weights from $nd-2$ items: $M$ ordinary items of weight one and $d-2$
exceptional items of weight $(r+n)/r$. Let $U$ and $V$ count disjoint
ordinary blocks of sizes $n-2$ and $n$. Expanding the coefficient
generating functions, including the two distinguished blocks, gives
\begin{equation}\label{eq:auxiliary-ratio}
 \frac AT=\mathbb E\frac{r+V}{r+1+U}.
\end{equation}
For example the factor for the $U$ block after division by $r+1+U$
is $\sum_{j=0}^{n-2}\binom{n-2}{j}x^j/(r+1+j)$; multiplication by
$r+V$ on the other block gives $g(x)$, leaving $g(x)^{d-2}$ from the
remaining ordinary and exceptional blocks. The normalizing polynomial
is the one defining $T$. Hence
\begin{equation}\label{eq:sign-reduction}
 \alpha>\beta
 \quad\Longleftrightarrow\quad
 A\{(n+r)Q+(n-1)Z\}>nZT.
\end{equation}

\subsection{Two replicas in any dimension}\label{app:two-replicas}
Here $n=2$. We first record the coefficient comparison needed for this
case. For $a\ge1$, $D\ge1$, define
\[
 b_j=[x^j](1+x)^{D-1}(1+ax)^D,\qquad
 c_j=[x^j](1+x)^D(1+ax)^{D+1}.
\]
Then, for $0\le k\le D$,
\begin{equation}\label{eq:reciprocal-coeff}
 b_kc_k\ge b_{k-1}c_{k+1}.
\end{equation}
Here and below coefficients outside support are zero. We include the
argument and its Schur-function foundation explicitly. At a positive
alphabet, the two sequences
\[
 A_j=e_j^2-e_{j-1}e_{j+1}=s_{(2^j)},\qquad
 B_j=e_je_{j-1}-e_{j-2}e_{j+1}=s_{(2^{j-1},1)}
\]
are log-concave in $j$. This is the specialization of the
rounded-partition Schur log-concavity theorem
\citep[Theorem~11]{lam2007schur}, applied to
$(j-1,j-1),(j+1,j+1)$ and $(j-1,j-2),(j+1,j)$, followed by the Schur
involution and the dual Jacobi--Trudi identity. The second assertion is
used for $j\ge2$, with $B_0=0$.

Put $q=\sqrt a$, $L=D-1$, and let $e_j$ be the coefficients of
$(1+qx)^L(1+q^{-1}x)^L$. This reciprocal alphabet gives
$e_j=e_{2L-j}$, $A_j=A_{2L-j}$, and $B_j=B_{2L+1-j}$.
The positive parts of both sequences are therefore nondecreasing up
to their midpoint, by log-concavity and symmetry. Set
$v_j=e_j+qe_{j-1}$ and $\Delta_j(v)=v_j^2-v_{j-1}v_{j+1}$.
Expansion gives
\[
 \Delta_j(v)=A_j+qB_j+q^2A_{j-1}.
\]
Thus $\Delta_j(v)\ge\Delta_{j-1}(v)$ for $1\le j\le L$; at $j=L+1$
the difference is $(q^2-1)(A_L-A_{L-1})\ge0$.
The case $L=0$ uses $A_0=1,A_{-1}=B_0=B_1=0$.
Since $b_j=a^{j/2}v_j$, this gives
$\Delta_k(b)\ge a\Delta_{k-1}(b)$. Finally,
$c_j=b_j+(1+a)b_{j-1}+ab_{j-2}$ implies
$b_kc_k-b_{k-1}c_{k+1}=\Delta_k(b)-a\Delta_{k-1}(b)$.
The case $k=0$ is immediate, completing \eqref{eq:reciprocal-coeff}.

For the application define
$E_j(M,D)=[x^j](1+x)^M(1+ax)^D$, and put
\[
 P_*=E_{s-2}(d-1,d-2),\quad T_*=E_{s-1}(d,d-2),\quad
 U_*=E_{s-1}(d,d-1),\quad Q_*=E_s(d+1,d-1).
\]
These numbers are positive. Coefficient complementation says
$E_j(M,D)=a^{j-M}E_{M+D-j}(D,M)$. With
$t=2d-s$, $k=t-1$, $D=d-1$, and $\ell=d-1-t$, it gives
\[
 P_*=a^\ell b_k,\quad T_*=a^\ell(b_k+ab_{k-1}),\quad
 U_*=a^\ell c_{k+1},\quad Q_*=a^\ell(c_{k+1}+ac_k).
\]
Consequently $P_*Q_*\ge U_*T_*$ by \eqref{eq:reciprocal-coeff}.
Now take $a=(r+2)/r>1$. The deletion quantities obey
\[
 Z=r^d[Q_*+(a-1)U_*],\quad Q=r^{d-1}Q_*,
 \quad T=r^{d-2}T_*,\quad
 A=\frac{r^{d-1}[T_*+(a-1)P_*]}{r+1}.
\]
The difference in \eqref{eq:sign-reduction} is
\[
 \frac{2r^{2d-2}(a-1)}{a+1}
 \{a(P_*Q_*-U_*T_*)+P_*Q_*+(a-1)P_*U_*\}>0.
\]
This proves the strict sign for every $d\le s<2d$.

\subsection{Two and three dimensions}\label{app:small-dimensions}
We use a simple reciprocal bound. If $0\le U\le L$,
$\mu=\mathbb EU<L$, $v=\operatorname{Var}U$, and $a>0$, put
$u_0=\mu-v/(L-\mu)$. Tilting the law by
$(L-U)/(L-\mu)$ makes its mean $u_0\in[0,L]$.
The identity
\[
 \frac1{a+U}=\frac1{a+L}
       +\frac{L-U}{(a+L)(a+U)}
\]
and Jensen's inequality under this tilted law imply
\begin{equation}\label{eq:tilted-reciprocal}
 \mathbb E\frac1{a+U}\ge
 \frac1{a+L}+\frac{L-\mu}{(a+L)(a+u_0)}.
\end{equation}

For $d=2$, divide the deletion normalizers by $\binom{2n}{s}$ and write
\[
 z=r^2+rs+\frac{ns(s-1)}{2(2n-1)},\quad q=r+s/2,\quad
 t=\frac{s(2n-s)}{2n(2n-1)},\quad h=(n+r)q+(n-1)z.
\]
The auxiliary $U$ is hypergeometric with population $2n-2$, marked
count $n-2$, and sample size $s-1$, and $V=s-1-U$. Its mean and
variance are
\[
 \mu=\frac{(n-2)(s-1)}{2n-2},\qquad
 v=\frac{(s-1)(n-2)n(2n-1-s)}{(2n-2)^2(2n-3)}.
\]
Use $L=s-1$ in \eqref{eq:tilted-reciprocal}. Then
$u_0=(n-2)(s-2)/(2n-3)$ and
\[
 \mathbb E\frac1{r+1+U}\ge
 L_0:=\frac1{r+s}+
 \frac{n(s-1)}{2(n-1)(r+s)(r+1+u_0)}.
\]
Put $R_0=(2r+s)L_0-1$. By \eqref{eq:auxiliary-ratio},
$A/T\ge R_0$ and
$\alpha-\beta=t\{(A/T)h-nz\}/[z(n+r)q]$.
Set $u=n-2$, $w=s-2$, and
$\mathcal D=4(n-1)(2n-1)(r+s)[(2n-3)r+(n-2)s+1]$.
Direct expansion gives $\mathcal D(R_0h-nz)=\sum_{i=0}^3C_ir^i$, where
\begin{align*}
 C_0&=n^2s^2(s-1)[2u^2+5u+2+(u+1)w],\\
 C_1&=ns[4u^3+30u^2+50u+18+
             (8u^3+50u^2+82u+33)w+(6u^2+16u+9)w^2],\\
 C_2&=2[12u^3+58u^2+78u+26+
             (4u^4+40u^3+129u^2+152u+51)w+
             (6u^3+29u^2+41u+16)w^2],\\
 C_3&=2(2n-3)(2n-1)[2+(u+3)w].
\end{align*}
All are positive for $n,s\ge2$. All cleared factors are positive and
$t>0$ for $s<2n$, proving the sign, including the deterministic
auxiliary endpoint.

For $d=3$, put $M=3(n-1)$. Conditional on $k$ selected ordinary items
in the auxiliary experiment, $U$ is hypergeometric on the ordinary
population. Its tilted mean is $(n-2)(k-1)/(M-1)$. Thus for $k\ge1$,
\begin{align*}
 L_k&=\frac1{r+1+k}+
 \frac{k(M-n+2)}{M(r+1+k)[r+1+(n-2)(k-1)/(M-1)]},\\
 \ell_k&=\frac{[r(M+2)+n(k+1)]L_k-n}{M-n+2}
 \le\mathbb E\left[\frac{r+V}{r+1+U}\,\middle|\,k\right].
\end{align*}
The second formula follows from
$\mathbb E[V\mid U,k]=n(k-U)/(M-n+2)$.
Divide $Z,Q,T$ by $\binom{3n}{s}$ and write
\begin{align*}
 z&=r^3+sr^2+\frac{ns(s-1)r}{3n-1}
             +\frac{n^2s(s-1)(s-2)}{3(3n-1)(3n-2)},\\
 q&=r^2+2sr/3+\frac{ns(s-1)}{3(3n-1)},\\
 t_0&=\frac{s(3n-s)}{3n(3n-1)},\qquad
 t_1=\frac{s(s-1)(3n-s)}{3n(3n-1)(3n-2)},\\
 t&=rt_0+nt_1,\qquad h=(n+r)q+(n-1)z,\\
 a&=t_1\left[(r+n)\ell_{s-2}
       +\frac{r(3n-s-1)}{s-1}\ell_{s-1}\right].
\end{align*}
There is one exceptional item. Its two possible selection states have
unnormalized masses $(r+n)\binom M{s-2}$ and $r\binom M{s-1}$.
The conditional bounds therefore give $A/\binom{3n}{s}\ge a$.
At $s=3n-1$, the second mass is zero and its impossible conditional
component is omitted.

Define
\begin{align*}
 \mathcal D={}&27n(n-1)(3n-2)^2(3n-1)^2(r+s)(r+s-1)\\
 &\quad\cdot[(3n-4)r+(n-2)s+n][(3n-4)r+(n-2)s+2],\\
 P(n,s,r)&=\frac{\mathcal D(ah-nzt)}{s(3n-s)}.
\end{align*}
Exact rational cancellation makes $P$ an integer polynomial.
Expansion of $P(2+u,3+w,r)$ has 342 nonzero coefficients, all positive,
and constant 9216. This is the small-dimension polynomial certificate.
It is reconstructed from the displayed rational formulas by clearing
denominators, exact division, substitution, and coefficient collection;
no sampled values of $n,s,r$ enter the check. The public
\path{anc/balanced/three_coordinate_certificate.py} checker verifies the full
formal identity and every coefficient.
For additional fingerprints,
\begin{align*}
 [r^0]P&=n^4s^2(s-2)(s-1)^3(ns+n-2s)
                   (3n^2+2ns-10n-2s+6),\\
 [r^7]P&=3(n-1)(3n-4)^2(3n-2)(3n-1)(3ns-6n+2s).
\end{align*}
These two coefficients do not replace the complete coefficient check.
All factors in $\mathcal D$ and $s(3n-s)$ are positive in
$3\le s<3n$. Hence $ah-nzt>0$, and
\[
 \alpha-\beta\ge\frac{ah-nzt}{z(n+r)q}>0.
\]

\subsection{The general domain: a variance-root split}\label{app:general-sign}
It remains to prove the sign for $n\ge3,d\ge4$.
First let $d\le s\le nd-2$ and put
\[
 M=d(n-1),\quad D=d-2,\quad N=M+D,\quad
 u=s-d+1,\quad L=M-u.
\]
In the auxiliary experiment, let $H$ count unselected exceptional
items. Its mass is proportional to
\[
 \binom Dk\binom M{u+k}r^k(r+n)^{D-k}.
\]
The selected ordinary count is $u+H$. Put $h=\mathbb EH$.
Detailed balance between adjacent masses gives
\[
 rDL=rNh+nuh+n\mathbb E H^2.
\]
Distinct exceptional inclusion indicators have nonpositive covariance.
Indeed, deleting the two items leaves three elementary symmetric
coefficients $a,b,c$; their covariance is a positive multiple of
$ac-b^2\le0$, by Newton's inequality. Complements have the same
covariance. Therefore
$\operatorname{Var}H\le h(1-h/D)$, and
\begin{equation}\label{eq:variance-root}
 q(h):=n(D-1)h^2+D[rN+n(u+1)]h-rD^2L\ge0.
\end{equation}
Here $0<h<D$ and $L>0$. The polynomial $q$ is strictly increasing on
the nonnegative line and has a unique positive root $\eta$, with
$h\ge\eta>0$. Since
$q(D)=D^2[n(D+u)+r(N-L)]>0$, also $\eta<D$.

The following polynomials use $h$ as a formal variable:
\begin{align*}
 Q_n&=rm(m-1)+ns(M+1)+nmh,\qquad m=nd,\quad s=u+d-1,\\
 Z_n&=rQ_n+ns[r(m-1)+n(u+h)],\\
 H_n&=(n+r)Q_n+(n-1)Z_n,\\
 B&=Mr+n(u+h),\\
 C&=Mr(s-1)+(u+h)[n(u-1)-r(d-1)].
\end{align*}
At the actual mean they satisfy
\begin{equation}\label{eq:actual-mean-identities}
 \frac QT=\frac{Q_n}{s(L+1)},\quad
 \frac ZT=\frac{Z_n}{s(L+1)},\quad
 \mathbb E(r+V)=\frac BM,\quad
 \mathbb E[(r+V)U]=\frac{(n-2)C}{M(M-1)}.
\end{equation}
Here is a derivation that keeps the second moment explicit. If $e_j$
are the elementary symmetric coefficients of the auxiliary weights,
counting selected inverse weights and unselected weights gives
\[
 \frac{e_s}{e_{s-1}}=\frac{rL+nh}{rs},\qquad
 \frac{e_{s-2}}{e_{s-1}}=
       \frac{r(s-1)+n(u+h)}{(r+n)(L+1)}.
\]
Multiplication by the two missing linear factors gives $Q/T$ and $Z/T$.
Conditional ordinary factorial moments give
\[
 \mathbb E[(r+V)U]
 =\frac{n-2}{M(M-1)}
 \{r(M-1)(u+h)+n\mathbb E[(u+H)(u+H-1)]\}.
\]
The drift identity above reduces the braces to $C$.
This proves \eqref{eq:actual-mean-identities} without replacing a
random variable by its mean inside a reciprocal.

Let $\theta=nZ_n/H_n$. The sign condition
\eqref{eq:sign-reduction} is now $A/T>\theta$.
For integer $U\ge0$,
\[
 \frac1{r+1+U}-\frac{r+2-U}{(r+1)(r+2)}
 =\frac{U(U-1)}{(r+1+U)(r+1)(r+2)}\ge0.
\]
Multiplication by $r+V$ gives the lattice bound below.
Weighted Cauchy gives the other bound:
\begin{align*}
 \frac AT&\ge L_1=
 \frac{(r+2)(M-1)B-(n-2)C}{M(M-1)(r+1)(r+2)},\\
 \frac AT&\ge L_C=
 \frac{(M-1)B^2}{M[(r+1)(M-1)B+(n-2)C]}.
\end{align*}
At the actual mean the latter denominator is positive because it is a
positive multiple of $\mathbb E[(r+V)(r+1+U)]$.
Define polynomial residuals
\begin{align}
 P(h)&=[(r+2)(M-1)B-(n-2)C]H_n
       -n(r+1)(r+2)M(M-1)Z_n,\label{eq:lattice-polynomial}\\
 S(h)&=(M-1)B^2H_n
       -nMZ_n[(r+1)(M-1)B+(n-2)C].\label{eq:cauchy-polynomial}
\end{align}
At the actual mean, $P(h)>0$ or $S(h)>0$ proves the desired sign.

Put $t=\eta/D$, $z=u+\eta$, and $\delta=M(1-t)-z$. Substitution into
$q(\eta)=0$ gives
\begin{equation}\label{eq:root-map}
 r=\frac{nt(z+1-t)}{\delta},\qquad
 1<z<M,\quad 0<t\le\frac{z-1}{D},\quad t<1-\frac zM.
\end{equation}
Indeed $0<t<1$, $z>1$, and the numerator is positive, so $\delta>0$.
The two upper bounds meet at $z_c=M(d-1)/N$, and
$1<z_c<d<M$ because
$z_c-1=D(M-1)/N$ and $d-z_c=d(n+d-3)/N$.

\subsection{Finite polynomial certificate and full domain coverage}
The exact certificate proves four statements: $P$ increases strictly
on $h\ge0$ when $1\le u\le d$; $S$ increases strictly at fixed $u$
on $0\le h<D,\ u+h\ge d$; $P(\eta)>0$ when $z<d$; and
$S(\eta)>0$ when $z\ge d$. The following formulas specify the entire
certificate and make it independently reconstructible.

First substitute $h=Dt$, $u=z-Dt$, and the value of $r$ in
\eqref{eq:root-map}, and define
\[
 P^*(n,d,z,t)=\delta^3P(Dt)/n^3,\qquad
 S^*(n,d,z,t)=\delta^3S(Dt)/n^4.
\]
Exact expansion verifies both divisions and gives integer polynomials
of total degree seven in $(z,t)$. Their individual $z,t$ degrees are
also seven. For the following three rational cells write
$z=z_n/z_d$, $t=t_n/t_d$, with $v,w\ge0$:
\[
\resizebox{\textwidth}{!}{$
\setlength{\arraycolsep}{3pt}
\begin{array}{c|c|c|c|c|c}
 &z_n&z_d&t_n&t_d&\text{kernel}\\ \hline
 \text{Low}&N(1+w)+D(M-1)w&N(1+w)&(M-1)wv&
 N(1+w)(1+v)&P^*\\
 \text{Upper}&M(d-1)+dNw&N(1+w)&
 [MN(1+w)-M(d-1)-dNw]v&MN(1+w)(1+v)&P^*\\
 \text{High}&d[1+(n-1)w]&1+w&(n-2)v&
 (n-1)(1+w)(1+v)&S^*
\end{array}
$}
\]
Every denominator is positive. In each cell expand
$t_d^7\operatorname{kernel}(n,d,z_n/z_d,t_n/t_d)$ and then put
$n=3+x,d=4+y$. To do this without rational arithmetic, write
$t_d=z_d\ell$. Each kernel monomial $a_{ij}(n,d)z^it^j$ contributes
\[
 a_{ij}(n,d)(z_n\ell)^it_n^jt_d^{7-i-j};
\]
all exponents are nonnegative by the total-degree check.
The complete coefficient checks are summarized here.
\[
\resizebox{\textwidth}{!}{$
\setlength{\arraycolsep}{2.5pt}
\begin{array}{l|l|r|r}
 \text{polynomial}&\text{nonnegative variables}&
 \text{nonzero coefficients}&\text{constant}\\ \hline
 {[h]P},\ d=u+v&(n-3,u-1,v,r)&234&432\\
 {[h^2]P},\ d=u+v&(n-3,u-1,v,r)&108&270\\
 (1+v)^2\partial_h S,\ u=d+z-h,\ h=Dv/(1+v)
       &(n-3,d-4,r,z,v)&840&412776\\
 \text{Low cell}&(n-3,d-4,v,w)&12144&576240000000\\
 \text{Upper cell}&(n-3,d-4,v,w)&27907&2953583794675777536\\
 \text{High cell}&(n-3,d-4,v,w)&5256&35782656
\end{array}
$}
\]
The first two rows prove monotonicity of $P$, and the third proves
monotonicity of $S$. The Low and Upper cells prove the root sign of $P$;
the High cell proves the root sign of $S$.
Every nonzero coefficient is a positive integer. For the third row,
differentiate at fixed $u,r$ before substituting. The first two rows
use $s=u+d-1$ throughout. The checker expands the defining polynomials,
compares the entire coefficient dictionaries (including absent
monomials), and verifies exact division by the stated factors.
Thus the 46,489 entries certify polynomial identities and positivity
on unbounded nonnegative orthants. Testing the signs of supplied
entries without reconstructing the polynomials would not suffice.

For clarity, the rational cells cover the whole root domain.
For $1<z<z_c$ with $t<(z-1)/D$, their inverse coordinates are
\[
 w=\frac{z-1}{z_c-z},\qquad
 v=\frac{t}{(z-1)/D-t}.
\]
For $z_c\le z<d$, use
$w=(z-z_c)/(d-z)$ and $v=t/(1-z/M-t)$.
For $d\le z<M$, use
$w=(z-d)/(M-z)$ and the same expression for $v$.
The seam $z=z_c$ belongs to Upper at $w=0$, and $z=d$ belongs to
High at $w=0$. No infinite-$w$ argument is needed.

The remaining face is $t=(z-1)/D$, or $u=1$. It necessarily has
$z<z_c$ and is the Low cell's $v\to\infty$ face.
The Low polynomial has degree seven in $v$. Its coefficient of $v^7$
has 1,504 positive coefficients in $(n-3,d-4,w)$ and constant
576240000000. Dividing by $t_d^7$ and taking this limit preserves a
strictly positive value. The other upper face $t=1-z/M$ is impossible
for finite $r>0$ because $\delta>0$.

We can now transport the residual signs. If $u+\eta<d$, the root
certificate gives $P(\eta)>0$, and $u\le d$ makes $P$ increasing on
$[\eta,h]$. Hence $P(h)>0$. If $u+\eta\ge d$, every $x\in[\eta,h]$
has $0\le x<D$ and $u+x\ge d$. The third coefficient row makes $S$
increasing there, so $S(h)\ge S(\eta)>0$. Only now do we divide by the
positive actual-mean Cauchy denominator. Both cases prove
$\alpha>\beta$.

At the remaining endpoint $s=nd-1$, the omitted row is uniform, so
$K_1=n-1$ with probability $1/d$ and otherwise $K_1=n$.
Direct substitution into the coefficient definitions gives
\[
 \alpha=\frac1{dn(r+n-1)^2},\qquad
 \frac{\beta}{\alpha}
 =\frac{dr(r+n-1)}{(r+n)[d(r+n-1)+1]}<1.
\]
This deals with the endpoint without a limiting root argument.

\subsection{Certificate status and geometric conclusion}
The public files \path{anc/balanced/sign_certificate.py} and
\path{anc/balanced/coefficient_certificate.json} reconstruct and check all six
general-domain polynomial identities using exact integer arithmetic.
The separate \path{anc/balanced/three_coordinate_certificate.py} checker reconstructs the 342-coefficient
identity above. Both checks are finite algebraic proofs of parameter
identities; original-row or count enumerations in the checkers are
additional finite corroboration. The formulas in this appendix specify
what those programs verify, including all boundary cases and signs of
removed factors.

The union of the two-replica, two- and three-dimensional, and general
arguments is exactly $n,d\ge2$, $d\le s<nd$, $r>0$.
At each calibrated $r_s$, the same target gives the common $b$ and
$J_0$ of \eqref{eq:contrasts}. The deficit
\eqref{eq:sharp-deficit} now proves the value and complete equality
space in Theorem~\ref{thm:sharp-geometry}. Indeed $\beta\ge0$ and
$\alpha c-\beta W_{jj}\ge c(\alpha-\beta)>0$.
Nonzero contrasts in a maximizing group exist because $n\ge2$.
This also shows that no further equality vectors appear for singular
queries, nonzero off-diagonal entries, or tied maximal diagonals.

\subsection{Short algebra for the quantitative margin}\label{app:margin-algebra}
Use the notation of Proposition~\ref{prop:small-penalty-gap}. Put
$E=d(n^2-2)-2n+2>0$; indeed $E\ge2(n^2-n-1)>0$.
Direct expansion of its displayed quadratic $P(h)$ gives
\begin{align*}
 p_2&=dn^2(nr+2n-1)[2n(M-1)+rE],\\
 p_1&=dn[2n^2(3n-1)(M-1)+nrB_1+r^2A_1+nr^3(nd-1)E].
\end{align*}
For $a=n-2\ge0$ and $e=d-2\ge0$, the two remaining polynomials are
\begin{align*}
 B_1={}&2e^2(a+1)(2a^2+8a+7)
       +e(23a^3+103a^2+142a+58)\\
      &+2(15a^3+59a^2+71a+22),\\
 A_1={}&e^2n(2a+3)(2a^2+8a+5)
       +e(18a^4+122a^3+290a^2+281a+88)\\
      &+2(10a^4+61a^3+131a^2+113a+29).
\end{align*}
Every displayed factor is nonnegative and both constant terms are
positive, so $p_1,p_2>0$. The constant term is
\[
 P(0)=dn(M-1)[2n^3-nC_1r-C_2r^2-(nd-1)C_3r^3],
\]
where
\begin{align*}
 C_1={}&3d^2n^2-5d^2n+2d^2-6dn^2+5dn-2d-2n^2+4n-2,\\
 C_2={}&4d^2n^3-7d^2n^2+2d^2n-8dn^3+9dn^2-3dn+4n^2-4n+2,\\
 C_3={}&dn^2-dn-d-2n^2+2n.
\end{align*}
We have $C_1\le3d^2n^2$, since the difference is
$d^2(5n-2)+d(6n^2-5n+2)+2(n-1)^2>0$.
Also $C_2\le4d^2n^3$: the difference is
\[
 d^2(7n^2-2n)+d(8n^3-9n^2+3n)-4n^2+4n-2>0.
\]
Its first and last terms together are at least $24n^2-4n-2>0$,
and its middle term is positive. Finally $C_3\le dn^2$ implies
$(nd-1)C_3\le d^2n^3$. No lower bound on the $C_i$ is needed.
These estimates prove the bound on $P(0)$ used in the proposition.
All identities are short polynomial expansions; this margin proof does
not use the large all-parameter sign certificate.

\section{Frame examples and exact coefficient algebra}\label{app:frames}
\subsection{Query class, a singular example, and the second budget}\label{app:etf-witness}
The row projectors $x_ix_i^\top$ have Frobenius Gram matrix
$(1-q)I+q\mathbf1\mathbf1^\top\succ0$. Their linear independence is a
classical ETF fact \citep[Section~5.1]{fickus2018}. Let $\mathcal Q_0$
be the orthogonal complement of their span in the symmetric matrices.
It has dimension $d(d+1)/2-m$. Since
$\sum_i x_ix_i^\top=gI$, every $V\in\mathcal Q_0$ has trace zero.
The symmetric flat-energy query class is exactly
\begin{equation}
 \{wI+V:w\in\mathbb R,\ V\in\mathcal Q_0\}.
\end{equation}
If $m<d(d+1)/2$, small nonzero multiples of any nonzero $V$ give
nonisotropic positive definite queries $I+\epsilon V$. When $m=d(d+1)/2$,
only scalar queries have flat row energy. Thus the $(6,3)$ ETF gives no larger query
class, whereas an existing $(16,6)$ ETF has five traceless directions.
This dimension count is a consequence of the classical projector fact,
not a new independence theorem.

Here is an explicit $(16,6)$ example. Index the features by the edges
$12,13,14,23,24,34$ of the complete graph on four vertices. At each
vertex place the four sign triples
\[
 (1,1,1),\quad(1,-1,-1),\quad(-1,1,-1),\quad(-1,-1,1)
\]
on its incident edges in the listed order, with zeros elsewhere. Stack the rows to form $Z$, and
put $X=Z/\sqrt3$. Then $Z^\top Z=8I$, each row has squared norm three,
and distinct row inner products have squared value one. This verifies
that $X$ is a unit-row ETF with $g=8/3$ and $q=1/9$. Fix
\begin{equation}\label{eq:etf-singular-query}
 W=\begin{pmatrix}
 1&0&0&0&0&1\\0&2&0&0&2&0\\0&0&3&3&0&0\\
 0&0&3&3&0&0\\0&2&0&0&2&0\\1&0&0&0&0&1
 \end{pmatrix}.
\end{equation}
Its eigenvalues are $0,0,0,2,4,6$. Off-diagonal couplings join disjoint
edges, which never share a row support; the diagonal weights sum to
six on every vertex star. Thus $x_i^\top Wx_i=2=w$ for every row.
This is a singular nonisotropic query in the physical coordinates.

For the prescribed target $\lambda_0=1$, the calibrated raw penalty is
the unique positive root $r_\star\in(4/5,9/10)$ of
\begin{equation}\label{eq:etf-witness-root}
 216r^3+603r^2+32r-576=0.
\end{equation}
Indeed \eqref{eq:etf-calibration} yields
\begin{equation}
 f(r)-\frac3{11}=
 -\frac{216r^3+603r^2+32r-576}{88(r+2)(3r+4)(3r+8)}.
\end{equation}
The root and interval are exact; no rounded penalty defines this model.
The sharp value is $2c_R(r_\star)$, where
\begin{equation}
 c_R(r)=\frac{21r^2+72r+64}{20(r+2)^2(3r+4)^2},
\end{equation}
and all ten residual dimensions maximize.

More generally, keep the ETF, flat physical query, and prescribed target
fixed, and consider LOO with its separately calibrated raw penalty
$r_1I$. Single deletion has equal determinant masses and scalar mean
coefficient
\begin{equation}\label{eq:etf-loo-calibration}
 f_1(r_1)=\frac1m\left\{\frac{(d-1)g}{g+r_1}
                          +\frac{g-1}{g-1+r_1}\right\}.
\end{equation}
It decreases strictly from $1/g$ to zero, so it reaches the same target
uniquely. Put $b_1=(g+r_1)^{-1}$ and $a_1=1-b_1$. The LOO leverage
is $b_1$ and $x_i^\top(b_1I)W(b_1I)x_i=b_1^2w$. Therefore all LOO
scores equal $wb_1^2/(ma_1^2)$, and
Theorem~\ref{thm:loo-envelope} gives
\begin{equation}
 h_{m-1}(X,W)=\frac{w}{m(g+r_1-1)^2},\qquad
 \text{complete maximizing space}=\ker(X^\top).
\end{equation}
This proves persistence of the full space at these two budgets only.
In the explicit example $r_1$ solves $48r_1^2+43r_1-80=0$, so it differs
from $r_\star$. No monotonicity of values, other-budget conclusion,
ETF design optimality, or classification of nonflat queries is asserted.

\subsection{Reconstruction of the three-deletion gap}\label{app:triple-certificate}
This appendix proves the strict sign used in Theorem~\ref{thm:etf-triple}.
Its coefficient certificate is finite, but it covers the unbounded
parameter domain. No grid of penalties or dimensions is used.

Put $m=d+k$ and introduce the following integer polynomials:
\begin{align*}
 v&=dr+k,\qquad \mathcal Z=dr+m,\\
 \mathcal Q&=(m-1)v^2-dk,\qquad
 \mathcal E=(m-1)v^2-4dk,\\
 \mathcal F&=(m-1)(m-2)v^3-3dk(m-2)v-2dk(k-d),\\
 \mathcal U&=(d-3)\mathcal F+3(m-2)\mathcal Z\mathcal Q,\\
 \mathcal R&=(m-2)v\mathcal Q+6dk(k-d),\\
 \mathcal V&=(d-3)\mathcal E\mathcal F
              +3(m-1)\mathcal Z^2\mathcal R.
\end{align*}
Direct substitution in \eqref{eq:triple-moments} gives
\begin{equation}\label{eq:triple-integer-moments}
 \begin{split}
 \mathcal E&=d^2(m-1)(h^2-4q)>0,\qquad
 \mathcal F=d^3(m-1)(m-2)D>0,\\
 H_1&=\frac{d\mathcal U}{\mathcal Z\mathcal F},\qquad
 H_2=\frac{d^2\mathcal V}{\mathcal Z^2\mathcal E\mathcal F}.
 \end{split}
\end{equation}
Define $\mathcal H$ by exact polynomial division:
\begin{equation}\label{eq:triple-gap-polynomial}
 3d^2k\mathcal H
 =\mathcal E(\mathcal Z+mr)\mathcal U^2
  -d\mathcal Z\mathcal E\mathcal F\mathcal U
  -rd\mathcal V(\mathcal U+k\mathcal F).
\end{equation}
Here $k\mathcal F$ is an ordinary product of the codimension and a
polynomial; it is different from the fitted-loss coefficient $k_F$.
The right side is divisible by $3d^2k$ in the integer polynomial ring,
and the quotient has degree seven in $r$.

To verify its statistical role, set
\[
 \mathcal L=\mathcal Z\mathcal F-r\mathcal U,\qquad
 \mathcal B=\mathcal Z^2\mathcal E\mathcal F
       -2r\mathcal Z\mathcal E\mathcal U+r^2d\mathcal V.
\]
Equations~\eqref{eq:triple-sectors} and
\eqref{eq:triple-integer-moments} imply
\begin{align*}
 M&=\frac{d\mathcal L}{\mathcal Z\mathcal F},\quad
 B=\frac{d\mathcal B}{\mathcal Z^2\mathcal E\mathcal F},\quad
 k_F=\frac{r\mathcal U}{\mathcal Z\mathcal F},\\
 c_R&=\frac{d[m\mathcal Z\mathcal E\mathcal U-mrd\mathcal V-d\mathcal B]}
              {mk\mathcal Z^2\mathcal E\mathcal F},\\
 c_F&=\frac{d[\mathcal B\mathcal F-\mathcal E\mathcal L^2]}
              {m\mathcal Z^2\mathcal E\mathcal F^2}.
\end{align*}
On the common denominator $mk\mathcal Z^3\mathcal E\mathcal F^2$,
the gap numerator is $d$ times
\[
 r\mathcal U[m\mathcal Z\mathcal E\mathcal U-mrd\mathcal V-d\mathcal B]
 -k\mathcal Z[\mathcal B\mathcal F-\mathcal E\mathcal L^2].
\]
Using $\mathcal Z=dr+m$ and $m=d+k$, this bracket is exactly $r\mathcal Z$
times the right side of \eqref{eq:triple-gap-polynomial}. Therefore
\begin{equation}\label{eq:triple-gap-cleared}
 c_Rk_F-c_F=\frac{3d^3r\mathcal H}
                 {m\mathcal Z^2\mathcal E\mathcal F^2}.
\end{equation}
All factors removed from the denominator are strictly positive on the
actual-frame domain proved in the main text.

For $d\ge3,k\ge4$, substitute $d=u+3$, $k=v_0+4$. The complete expansion
of $\mathcal H(r,u+3,v_0+4)$ has 581 nonzero coefficients, all positive
integers. The exact coefficient data are in
\path{anc/etf_triple/gap_coefficients.json}. Each entry is
$[i,j,l,c_{ijl}]$ for $c_{ijl}r^iu^jv_0^l$. The following table provides
checks on the expansion; the full artifact supplies every coefficient.
\begin{center}
\begin{tabular}{rrrr}
\toprule
Power of $r$ & Nonzero coefficients & Minimum & Constant in $(u,v_0)$\\
\midrule
0&69&1&145212480\\
1&75&6&678003984\\
2&78&15&1273128192\\
3&79&20&1253473704\\
4&78&15&702853200\\
5&75&6&225144360\\
6&70&1&38141280\\
7&57&1&2624400\\
\bottomrule
\end{tabular}
\end{center}
This proves $\mathcal H>0$ for $r>0,u,v_0\ge0$, including the shifted
boundary. The exact reconstruction has four steps: form the displayed
integer polynomials; perform the division in
\eqref{eq:triple-gap-polynomial}; substitute $d=u+3,k=v_0+4$; and collect
all monomials. The public program
\path{anc/etf_triple/check.py} performs these steps, verifies
\eqref{eq:triple-integer-moments} and \eqref{eq:triple-gap-cleared}, and
compares the whole coefficient list, including its support, with the
artifact. Merely testing signs of stored numbers would not prove the
identity. The program also derives the local energy and weighted moments
and enumerates all 560 original-row triples of an actual $(16,6)$ ETF.
Those finite model checks support the derivation; the polynomial identity
proves its uniform sign.

The only remaining actual case is $d=k=3$. Substitution in the compact
moment formulas gives
\[
 c_Rk_F-c_F=
 \frac{r(5r^2+10r+4)(5r^3+25r^2+30r+6)}
 {2(r+1)^2(5r^2+10r+1)(5r^2+10r+2)^2}>0.
\]
The checker verifies this identity separately. The domain argument in the
main proof excludes every other actual $k=3$ pair. Thus these two
calculations cover every existing ETF in Theorem~\ref{thm:etf-triple}.

\section{Mixture coefficient evaluation and examples}\label{app:mixture}
\subsection{Evaluating the coefficients}
Four scalar count moments suffice given exact calibration:
\[
 U=\mathbb EK_1^2,\quad C=\mathbb E[K_1f(K_1)],\quad
 F=\mathbb Ef(K_1)^2,\quad P=\mathbb E[f(K_1)f(K_2)].
\]
The fixed total count gives
$C_\times=(sa-C)/(d-1)=\mathbb E[f(K_1)K_2]$ and
$U_\times=(s\mu-U)/(d-1)=\mathbb E[K_1K_2]$.
Set
\[
 L=\frac{na-C}{n(n-1)},\qquad N=\frac{n\mu-U}{n(n-1)}.
\]
Then all required coefficients are
\begin{align*}
 A&=\frac{na-(n+r)F}{rn(n-1)},\qquad
 D=2(A-c_HL),\qquad Q=A-2c_HL+c_H^2N,\\
 B_0&=(F-P)/b,\\
 B_1&=2[c_H(C-C_\times)-(F-P)]/b,\\
 B_2&=[c_H^2(U-U_\times)-2c_H(C-C_\times)+(F-P)]/b,\\
 \tau(t)&=\{a^2-[(1-t)^2P+2t(1-t)c_HC_\times+
                                    t^2c_H^2U_\times]\}/b.
\end{align*}
These follow by expanding \eqref{eq:mixture-coeff}; for the first identity
use $nf(k)-(n+r)f(k)^2=rk(n-k)/(r+k)^2$.
For numerical evaluation, the direct nonnegative sums can avoid
subtractive cancellation.

The generating polynomial $g(x)$ in Appendix~\ref{app:geometry} gives
marginal weights $a_k[x^{s-k}]g(x)^{d-1}/Z$ and pair weights
$a_ka_\ell[x^{s-k-\ell}]g(x)^{d-2}/Z$.
Repeated convolution truncated at degree $s$, followed by these sums,
takes $O(dns+n^2)$ arithmetic operations and $O(s+n)$ working memory.
This count excludes calibration, subset generation, bit complexity,
and numerical stability. An approximately supplied penalty does not
inherit the exact mean statement proved here.

\subsection{A concrete improvement and its adverse response}
Take $n=d=s=2$, $r=1$, so $a=5/11$,
$\lambda_0=12/5$, and $b=12/11$. Each mixed-group original subset
has probability $2/11$, and each full-single-group subset has
probability $3/22$. Direct evaluation gives
\[
 \alpha(t)=\frac{(11-t)^2}{1331},\qquad
 \beta(t)=\frac{(11+4t)^2}{2178}.
\]
Here
$D=2/121$, $Q=1/1331$, $B_0=1/18$, $B_1=4/99$, $B_2=8/1089$,
and \eqref{eq:mixture-step} gives $T_s=847/3116>1/4$.
At $t=1/4$, every query's sharp risk drops by the factor
$1849/1936$. Yet for the positive-copy response
$y=(1,1,-1,-1)$ and $W=(1,1)(1,1)^\top$, the normalized risk
increases from $1/18$ to $8/121$. This is a pure transverse mean
response. Thus a smaller supremum does not imply pointwise risk
dominance.

More generally, for $0\le t\le1$,
\[
 \sup_{0\ne W\succeq0}\sup_{y\ne0}
 \frac{\operatorname{tr}(W\operatorname{Cov}(w(t)))}
      {J_0(y)\max_jW_{jj}}
 =\max\{\alpha(t),\beta(t)\}.
\]
The upper bound follows from $\tau(t)\ge0$ and the two response blocks.
Contrasts attain $\alpha(t)$; a flat rank-one query
$W=vv^\top$, $v_j\in\{-1,1\}$, and nonzero means $q\perp v$
attain $\beta(t)$. Hence optimizing this scalar mixture family amounts
to minimizing the maximum of two convex quadratics. Its best weight
can differ from the conservative step above, and can introduce mean
maximizers. This does not optimize over all estimators or laws.

\section{Explicit recalibration and transfer bounds}\label{app:transfer}
This appendix first gives a perturbation bound for general full-rank
designs at every $d\le s<m$. It then uses the balanced base's complete
equality space. All constants concern a fixed target and query.

\subsection{Fixed target coordinates and explicit constants}
Put
\[
 Z=XA_0^{-1/2},\quad E=HA_0^{-1/2},\quad B=Z^\top Z,\quad
 \mathsf P_s=A_0^{-1/2}R_sA_0^{-1/2},\quad V=A_0^{-1/2}WA_0^{-1/2}.
\]
For a diagonal row selector $D_S$, define
\begin{align*}
 A_S&=Z^\top D_SZ+\mathsf P_s,&
 L_S&=A_S^{-1}Z^\top D_S,&L_0&=(I+B)^{-1}Z^\top,\\
 F_S&=L_S-L_0,&
 C_s&=\mathbb E[F_S^\top VF_S],&
 T&=(I+ZZ^\top)^{1/2}=K_X^{-1/2}.
\end{align*}
Thus the physical risk is $y^\top C_sy$ and $M_s=TC_sT$.
A prime on a whole object denotes its value at $Z+E$.

Let $z=\lVert Z\rVert$, $\sigma=\sigma_{\min}(Z)>0$, and choose
\begin{equation}\label{eq:radius-constants}
 r_0=\sigma/2,\quad \ell=\sigma^2/4,\quad
 v=z+r_0,\quad u=v^2,\quad t=\sqrt{1+v^2}.
\end{equation}
If $\epsilon=\lVert E\rVert_F\le r_0$, every $Z+xE$, $0\le x\le1$,
has Gram spectrum in $[\ell,u]$. For each $d\le s<m$, set
\begin{align}
 a_s&=\frac{s-d+1}{m-d+1},\qquad b_s=\frac sm,\qquad
 q_-=\frac{\ell}{b_s},\quad q_+=\frac{u}{a_s},\nonumber\\
 \kappa_s&=\frac{b_sq_-}{(1+q_-)(1+q_+)},\qquad
 \mathcal L_s=b_s\left[\frac1\ell+
                  \frac1{\kappa_s(1+\ell)^2}\right],\label{eq:calibration-modulus}\\
 p&=a_s,\quad L_X=2v\mathcal L_s,\quad
 B_F=\frac1{2\sqrt p}+\frac12,\nonumber\\
 C_D&=\frac5{4p}+\frac54+\frac{L_X}{2p^{3/2}},\qquad
 C_F=tC_D+B_Fv,\qquad
 L_p=\sqrt d\left(\frac2{\sqrt p}+\frac{L_X}p\right).\label{eq:composite-constants}
\end{align}
These are positive finite constants determined by $X,A_0,m,s$.
A convenient explicit operator bound is
\begin{equation}\label{eq:explicit-eta}
 \eta_s(\epsilon)=\lVert V\rVert
 \left[2B_FtC_F\epsilon+B_F^2t^2\tanh(L_p\epsilon/2)\right]
 \le \mathcal M_s\epsilon,
 \quad
 \mathcal M_s=\lVert V\rVert
 \left[2B_FtC_F+\tfrac12B_F^2t^2L_p\right].
\end{equation}
In Theorem~\ref{thm:transfer} take $r=r_0$ and $L_s=\mathcal M_s$.
These constants require no sum over original-row subsets.

\subsection{Curvature of calibration, including matrix directions}
Recall the spectral partition $P_s(q)$ from
\eqref{eq:cal-partition}. Its log-coordinate Hessian is
$\mathcal C=\operatorname{Cov}(\mathbf 1_J)$. There is a useful
conditional experiment for this distribution. Select $s$ of $m$ slots
with product weights $1+q_j$ on the first $d$ slots and one on the
others. Conditional on those slots, mark each selected feature slot
$j$ independently with probability $q_j/(1+q_j)$.
Summing over the unmarked selected companions shows that the marked
set has mass proportional to $c_{|J|}q_J$, as required.
Total conditional variance therefore gives
\begin{equation}\label{eq:marked-curvature}
 \mathcal C\succeq\operatorname{diag}\left(\frac{\rho_j}{1+q_j}\right).
\end{equation}
For weights bounded above by $w_{\max}$, the selection probability of
a slot of weight $w$ is
$we_{s-1}/(e_s+we_{s-1})$. Counting extensions of $(s-1)$-subsets
gives $se_s\le(m-s)w_{\max}e_{s-1}$, so this probability is at least
$(s/m)w/w_{\max}$. On the box $q_-\le q_j\le q_+$,
\eqref{eq:marked-curvature} consequently implies
\begin{equation}\label{eq:uniform-curvature}
 \mathcal C\succeq
 \frac{(s/m)q_-}{(1+q_-)(1+q_+)}I=\kappa_s I.
\end{equation}
The slightly conservative factor $(1+q_-)$ in this bound is harmless.

For a symmetric matrix $Q$ with eigenvalues $z_j$, define
$\mathcal A_s(Q)=\log P_s(e^{z_1},\ldots,e^{z_d})$. The diagonal
part of its Hessian is $\mathcal C$. Its off-diagonal coefficients
are $(\rho_i-\rho_j)/(z_i-z_j)$, with continuous limits at equality.
To bound them, join the vector $z$ to the vector obtained by swapping
coordinates $i,j$. The segment stays in the same log box. Integrating
\eqref{eq:uniform-curvature} along it gives
\[
 2(\rho_i-\rho_j)(z_i-z_j)\ge
 2\kappa_s(z_i-z_j)^2.
\]
The usual second derivative formula for a spectral function now proves
$\nabla^2\mathcal A_s\succeq\kappa_s I$ on all symmetric directions
in Frobenius norm, including repeated eigenvalues.

In target coordinates the calibrated $\mathsf P_s$ commutes with $B$.
For $B=U\operatorname{diag}(\lambda_j)U^\top$, its eigenvalues $p_j$
satisfy
\[
 \rho_j=\frac{\lambda_j}{1+\lambda_j},\qquad
 q_j=\frac{\lambda_j}{p_j},\qquad
 \left.\nabla_z\log P_s(e^z)\right|_{z=\log q}=\rho.
\]
The sandwich \eqref{eq:penalty-sandwich} gives
$a_sI\preceq\mathsf P_s\preceq b_sI$, placing $q_j$ in the box used above.
The matrix target moment is $D=B(I+B)^{-1}$, with
$\dot D=(I+B)^{-1}\dot B(I+B)^{-1}$.
The inverse Hessian bound gives
$\lVert\dot Q\rVert_F\le
\lVert\dot B\rVert_F/[\kappa_s(1+\ell)^2]$ for the log-parameter
matrix $Q$. On diagonal directions,
$\dot p_j=(p_j/\lambda_j)\dot\lambda_j-p_j\dot z_j$.
This yields the bound $\mathcal L_s$ in \eqref{eq:calibration-modulus}.

For completeness the off-diagonal bound does not assume a fixed
eigenbasis. If $\lambda_i>\lambda_j$, the exchange identity in
Section~\ref{sec:calibration} gives $q_i>q_j$, and
\[
 p_i-p_j=\frac{\lambda_i-\lambda_j}{q_i}
                   -p_j\frac{q_i-q_j}{q_i},\qquad
 \frac{q_i-q_j}{q_i}\le\log q_i-\log q_j.
\]
The same swap argument bounds the inverse moment divided difference
by $1/\kappa_s$, while
$(\rho_i-\rho_j)/(\lambda_i-\lambda_j)
=1/[(1+\lambda_i)(1+\lambda_j)]$.
Thus the absolute penalty divided difference is at most $\mathcal L_s$.
Diagonal and off-diagonal directions are orthogonal in Frobenius
geometry, and continuity covers ties. We have proved
\begin{equation}\label{eq:penalty-derivative}
 \lVert D\mathsf P_s(B)[\dot B]\rVert_F\le
       \mathcal L_s\lVert\dot B\rVert_F.
\end{equation}
Existence of the derivative follows locally from the positive matrix
Hessian and the inverse-function theorem. Uniqueness of calibration
patches these local maps along the full segment.

\subsection{Tracking the fit, law, and loss}
Along $Z+xE$, $\lVert\dot B\rVert_F\le2v\epsilon$, so
\eqref{eq:penalty-derivative} gives
$\lVert\dot{\mathsf P}_s\rVert_F\le L_X\epsilon$.
Each $A_S\succeq pI$. Factoring through the penalized ridge operator
and using $\sup_{x\ge0}x/(1+x^2)=1/2$ yields
\[
 \lVert L_S\rVert\le(2\sqrt p)^{-1},\qquad
 \lVert L_0\rVert\le1/2,\qquad \lVert F_S\rVert\le B_F.
\]
For the selected rows, differentiation gives
\[
 \dot L_S=A_S^{-1}E_S^\top(I-Z_SL_S)
                -L_SE_SL_S-A_S^{-1}\dot{\mathsf P}_sL_S,
\]
with the outside row selector restored when $L_S$ is viewed as a
$d$ by $m$ operator. The selected smoother is positive semidefinite
and at most $I$. Thus the first two terms are bounded by
$5\epsilon/(4p)$ and the last by
$L_X\epsilon/(2p^{3/2})$. The full-fit derivative is at most
$5\epsilon/4$. Hence $\lVert\dot F_S\rVert\le C_D\epsilon$.

The equation $T\dot T+\dot TT=EZ^\top+ZE^\top$, with $T\succeq I$,
has its integral solution
$\dot T=\int_0^\infty e^{-aT}(EZ^\top+ZE^\top)e^{-aT}\,da$.
Consequently $\lVert\dot T\rVert\le v\epsilon$, and
\[
 \lVert F_ST\rVert\le B_Ft,\qquad
 \lVert(F_ST)^{\boldsymbol{\cdot}}\rVert\le C_F\epsilon.
\]
The derivative of an unnormalized log weight obeys
\begin{align*}
 \left|\frac{d}{dx}\log\det A_S\right|
 &\le2\lVert Z_SA_S^{-1/2}\rVert_F
          \lVert E_SA_S^{-1/2}\rVert_F
       +\lVert A_S^{-1}\rVert_F\lVert\dot{\mathsf P}_s\rVert_F\\
 &\le L_p\epsilon.
\end{align*}
Thus its endpoint ratio lies in $[e^{-L_p\epsilon},e^{L_p\epsilon}]$,
and the normalized likelihood ratio lies in
$[e^{-2L_p\epsilon},e^{2L_p\epsilon}]$.
In addition,
\begin{equation}\label{eq:law-tv}
 \operatorname{TV}(p_s',p_s)\le\tanh(L_p\epsilon/2).
\end{equation}
To verify this sharper statement, if a positive unnormalized density
ratio lies in $[A,B]$ and has mean $\mu$, convexity of
$(x-\mu)_+$ bounds its normalized total variation by
$(B-\mu)(\mu-A)/[\mu(B-A)]$. Maximizing over $\mu$ gives
$(\sqrt B-\sqrt A)/(\sqrt B+\sqrt A)$, which proves
\eqref{eq:law-tv}.

Each normalized risk atom
$(F_ST)^\top V(F_ST)$ lies between zero and
$\lVert V\rVert B_F^2t^2 I$. Its derivative norm is at most
$2\lVert V\rVert B_FtC_F\epsilon$.
The change in the expectation of an atom in $[0,BI]$ is at most
$B$ times total variation: apply the scalar bound to each unit quadratic
form. Integrating the atom derivative and then changing the law proves
\[
 \lVert M_s'-M_s\rVert\le\eta_s(\epsilon)
\]
with \eqref{eq:explicit-eta}. This accounts for both centering at the
new full fit and changing normalization.

For the same nonzero raw response at both designs there is also the
useful loss comparison
\begin{equation}\label{eq:loss-comparison}
 e^{-2\epsilon}\le\frac{y^\top K_{X+H}y}{y^\top K_Xy}
                 \le e^{2\epsilon}.
\end{equation}
Indeed, along the path,
\[
 \left|\frac d{dx}\log(y^\top Ky)\right|
 =\frac{|y^\top K(EZ^\top+ZE^\top)Ky|}{y^\top Ky}
 \le2\epsilon,
\]
using $\lVert K^{1/2}Z\rVert\le1$ and $\lVert K^{1/2}\rVert\le1$.
Integrating proves \eqref{eq:loss-comparison}.

\subsection{Why reoptimization and the linear risk term matter}
The distinction is visible in an exact four-row model. Take $n=d=2$,
$s=3$, $\lambda_0=10/7$, and $W=e_1e_1^\top$. Scale every row by
$\sqrt t$. The calibrated penalty solves
$4r(t)^2+(6t-3\lambda_0)r(t)-4t\lambda_0=0$.
At $t=1$, $r=1$, $r'(1)=-1/34$, and
$h(t)=t/[4(r(t)+t)^2]$ has derivative $h'(1)=1/544$.
The strict base gap keeps this contrast branch maximal locally.
Thus the optimum can have nonzero first derivative.
At the same base with $W=I$, scale only the first coordinate group
by $\sqrt t$. Write the calibrated penalty as
$\operatorname{diag}(r_1(t),r_2(t))$. Define
\[
 A=(r_1+t)(r_2+2),\qquad B=(r_1+2t)(r_2+1),\qquad Z=A+B.
\]
The two count types $(1,2)$ and $(2,1)$ have probabilities $A/Z$ and $B/Z$.
With $\lambda=10/7$, the calibration equations are
\[
 \frac{A/(r_1+t)+2B/(r_1+2t)}Z=\frac2{2t+\lambda},\qquad
 \frac{2A/(r_2+2)+B/(r_2+1)}Z=\frac2{2+\lambda}.
\]
The two contrast risk values are
\[
 \alpha_1(t)=\frac{tA}{2Z(r_1+t)^2},\qquad
 \alpha_2(t)=\frac{B}{2Z(r_2+1)^2}.
\]
Differentiating the two calibration equations at $t=r_1=r_2=1$ gives
\[
 (r_1',r_2')=\frac1{1224}(-1,-35),\qquad
 (\alpha_1',\alpha_2')=\frac1{117504}(-589,805).
\]
Thus $\alpha_1'-\alpha_2'=-41/3456$.
For $t>1$ sufficiently close to one, an old maximizer from the first
group has linear regret. Reoptimization over the complete tied space is
essential to the quadratic statement.

\section{A finite remainder and a common contrast-space radius}\label{app:enlarged}
This appendix supplies the finite approximation used by the enlarged-space
theorem. It retains the linear fit change, the centered determinant score,
and the exact change in the Full fit. Its error is second order locally. The construction
also explains which exact mean correction would be needed for a supplied
penalty that is not calibrated. The main theorem uses exact calibration.

\subsection{The common finite surrogate}
Fix the balanced base and a budget $d\le s<m$. Work in target coordinates:
\[
 Z=X_0/\sqrt{\lambda_0},\quad E=H/\sqrt{\lambda_0},\quad Z_1=Z+E,
 \quad V=W/\lambda_0,\quad P_0=pI,\quad\Delta P=P_1-pI.
\]
Here $p=r_s/\lambda_0>0$. The endpoint is full column rank and $P_1$ is
SPD. In the exactly calibrated case $P_1=R_s(X_0+H)/\lambda_0$.
Use this same $P_1$ in its law and fit. For each original-row subset define
\begin{align*}
 A_S&=Z^\top D_SZ+pI,& A_{1,S}&=Z_1^\top D_SZ_1+P_1,\\
 L_S&=A_S^{-1}Z^\top D_S,&L_{1,S}&=A_{1,S}^{-1}Z_1^\top D_S,\\
 G_S&=Z^\top D_SE+E^\top D_SZ+\Delta P,&
 D^{\rm fit}_S&=A_S^{-1}(E^\top D_S-G_SL_S),\\
 L_{*,i}&=(I+Z_i^\top Z_i)^{-1}Z_i^\top,&
 J_S&=L_S-L_{*,0}+D^{\rm fit}_S,\\
 Q_S&=J_S^\top VJ_S,& t_S&=\operatorname{tr}(A_S^{-1}G_S).
\end{align*}
In these formulas $Z_0=Z$. Write $\mathbb E_i$ for the base and endpoint
laws, $\overline Q=\mathbb E_0Q_S$, $\bar t=\mathbb E_0t_S$, and
\[
 N(Z,P)=\sum_{|S|=s}\det(P+Z^\top D_SZ),\quad
 \mu=\frac{N(Z_1,P_1)}{N(Z,pI)},\quad
 \bar w=\mu-1-\bar t,\quad \Delta=L_{*,1}-L_{*,0}.
\]
The complete surrogate is
\begin{equation}\label{eq:common-surrogate}
 \widetilde C=\overline Q+
 \frac{\mathbb E_0[t_SQ_S]-\bar t\,\overline Q}{\mu}
 -\Delta^\top V\Delta.
\end{equation}
The endpoint covariance is
$C_1=\mathbb E_1[(L_{1,S}-L_{*,1})^\top V(L_{1,S}-L_{*,1})]$
when the endpoint is exactly calibrated. Put
$T_1=(I+Z_1Z_1^\top)^{1/2}$ and $\widetilde M=T_1\widetilde C T_1$.
The surrogate itself need not be PSD.

Here is the exact error identity. Define
\[
 R_S=L_{1,S}-L_S-D^{\rm fit}_S,\qquad
 w_S=\frac{\det A_{1,S}}{\det A_S}-1-t_S.
\]
The endpoint density is
\[
 \frac{p_1(S)}{p_0(S)}
 =1+\frac{t_S-\bar t}{\mu}+\frac{w_S-\bar w}{\mu}.
\]
Since exact calibration gives $\mathbb E_1L_{1,S}=L_{*,1}$,
expansion around $L_{*,0}$ yields
\begin{equation}\label{eq:common-error-identity}
 C_1-\widetilde C
 =\frac{\mathbb E_0[(w_S-\bar w)Q_S]}{\mu}
  +\mathbb E_1[J_S^\top VR_S+R_S^\top VJ_S]
  +\mathbb E_1[R_S^\top VR_S].
\end{equation}
Thus the finite Full shift cancels exactly. Omitting it from the surrogate
would alter the error identity.

\subsection{Ordered resolvent bounds at an SPD endpoint}
All norms below are operator norms. Let $e_s^{\rm row}$ denote the sum of the $s$
largest squared row norms of $E$. Put $c_0=1/\lambda_0$ and
\begin{align*}
 k_-&=\max\{0,s-n(d-1)\},&k_+&=\min\{n,s\},\\
 a_-&=p+c_0k_-,&a_+&=p+c_0k_+,\\
 b_*&=a_-^{-1/2},&u_*&=\sqrt{c_0k_+/(p+c_0k_+)},\\
 u_-&=\sqrt{c_0k_-/a_-},&a_*^2&=\min\{e_s^{\rm row},\|E\|^2\},\\
 v_*^2&=a_*^2/a_-,&\tau_*&=e_s^{\rm row}/a_-,\\
 g_*&=2u_*v_*+\|\Delta P\|/a_-,&\delta_*&=g_*+v_*^2.
\end{align*}
Choose a valid $l>0$ with $P_1\succeq lI$. Define
\begin{align*}
 \ell_{\rm pen}&=l/a_+,\qquad
 \ell_{\rm gram}=\ell_{\rm pen}+(u_--v_*)_+^2,\\
 \ell_{\rm ray}&=\min_{x\in[c_0k_-,c_0k_+]}
                 \frac{l+(\sqrt x-a_*)_+^2}{p+x},\\
 \ell_*&=\max\{\ell_{\rm gram},\ell_{\rm ray},1-g_*\}>0,
 \qquad \kappa_*=1/\ell_*,\\
 \beta_*&=\min\left\{\delta_*\kappa_*,
        \max\left(\left|\ell_*^{-1}-1\right|,
                  \left|(1+\delta_*)^{-1}-1\right|\right)\right\},\\
 r_{\rm pt}&=b_*\{\beta_*(v_*+g_*u_*)+\kappa_*v_*^2u_*\}.
\end{align*}
This scalar minimum is explicit. If $a_*>0$, clip
\[
 t_*=\frac{-(p-l-a_*^2)+
             \sqrt{(p-l-a_*^2)^2+4a_*^2p}}{2a_*}
\]
to $[\sqrt{c_0k_-},\sqrt{c_0k_+}]$ and evaluate at $x=t_*^2$.
If $a_*=0$, use the lower endpoint when $p\ge l$ and the upper endpoint
otherwise. To verify the rule, write $t=\sqrt x$. Below $a_*$ the
ratio decreases. Above $a_*$ its derivative has the sign of
$a_*t^2+(p-l-a_*^2)t-a_*p$, whose unique positive root exceeds $a_*$.

For each subset suppress the subscript and set
\[
 \begin{gathered}
 U=A^{-1/2}Z^\top D_S,\qquad Y=A^{-1/2}E^\top D_S,\\
 B_1=UY^\top+YU^\top+A^{-1/2}\Delta P A^{-1/2},\qquad B_2=YY^\top,\\
 B=B_1+B_2,\qquad H_B=(I+B)^{-1}.
 \end{gathered}
\]
Then $\|U\|\le u_*$, $\|Y\|\le v_*$,
$\operatorname{tr}B_2\le\tau_*$, $\|B_1\|\le g_*$, and
$\|B\|\le\delta_*$. The trace and operator bounds on $B_2$ are different;
$\tau_*$ cannot be replaced by $v_*^2$.

The penalty and Gram parts give
$I+B\succeq\ell_{\rm gram}I$. For the Rayleigh bound, any nonzero
$h$ has $x=\|D_SZh\|^2/\|h\|^2\in[c_0k_-,c_0k_+]$, and
\[
 h^\top A_1h\ge\{l+(\sqrt x-a_*)_+^2\}\|h\|^2,
 \qquad h^\top Ah=(p+x)\|h\|^2.
\]
Also $I+B\succeq(1-g_*)I$. Consequently
\[
 \ell_*I\preceq I+B\preceq(1+\delta_*)I,\quad
 \|H_B\|\le\kappa_*,\quad \|H_B-I\|\le\beta_*.
\]
The last bound follows both from $H_B-I=-H_BB$ and from its eigenvalues.
There is no assumption $\delta_*<1$ here.
The ordered resolvent identity is
\begin{equation}\label{eq:endpoint-resolvent}
 R=A^{-1/2}\{(H_B-I)(Y-B_1U)-H_BB_2U\}.
\end{equation}
Indeed $L_1=A^{-1/2}H_B(U+Y)$ and
$L+D^{\rm fit}=A^{-1/2}(U+Y-B_1U)$; subtract and use
$H_B-I=-H_B(B_1+B_2)$. This proves $\|R\|\le r_{\rm pt}$ without
commuting any matrices.

\subsection{Residual moments and the finite normalized error}
Retain the two PSD residual moments
\[
 T_S^{\rm res}=Y-B_1U=A^{1/2}D^{\rm fit},\qquad
 W_S^{\rm res}=Y^\top U=D_SE L_S.
\]
For each of the three PSD atoms
$F_S\in\{Q_S,(T_S^{\rm res})^\top T_S^{\rm res},
(W_S^{\rm res})^\top W_S^{\rm res}\}$, define
$\bar F=\mathbb E_0F_S$ and $F^{\rm sc}=\mathbb E_0[t_SF_S]$.
Put
\[
 \rho_*=\tau_*+\sum_{j=2}^d\binom dj\delta_*^j,\quad
 a_{\rm den}=\frac{\rho_*+|\bar w|}{\mu},\quad
 q_{\rm den}=\frac{(1+\delta_*)^d}{\mu},
\]
and define the scalar upper bound
\[
 \mathcal U(\bar F,F^{\rm sc})=
 \min\left\{q_{\rm den}\|\bar F\|,
 \lambda_{\max}\left(\bar F+
       \frac{F^{\rm sc}-\bar t\bar F}{\mu}+a_{\rm den}\bar F\right)\right\}.
\]
Both entries are nonnegative. Expansion of $\det(I+B)$ gives
$|w_S|\le\rho_*$ and $0<\det(I+B)\le(1+\delta_*)^d$.
The exact density formula therefore proves the PSD inequalities
\[
 \mathbb E_1F_S\preceq q_{\rm den}\bar F,\qquad
 \mathbb E_1F_S\preceq\bar F+
       (F^{\rm sc}-\bar t\bar F)/\mu+a_{\rm den}\bar F.
\]
Take their scalar norm bounds; no matrix minimum is used.
Let $q_J,q_T,q_W$ be these three values and put
\begin{align}
 r_{\rm ms}&=b_*\{\beta_*\sqrt{q_T}+\kappa_*v_*\sqrt{q_W}\},
 &e_R&=\|V\|\min\{r_{\rm ms}^2,r_{\rm pt}^2\},\nonumber\\
 D_{\rm end}&=a_{\rm den}\|\bar Q\|+2\sqrt{q_Je_R}+e_R,
 &e_{\rm end}&=(1+\|Z_1\|^2)D_{\rm end}.
 \label{eq:endpoint-error}
\end{align}
Then
\begin{equation}\label{eq:endpoint-error-bound}
 \|C_1-\widetilde C\|\le D_{\rm end},\qquad
 \|M_1-\widetilde M\|\le e_{\rm end}.
\end{equation}
To prove this, rewrite \eqref{eq:endpoint-resolvent} as
$R=A^{-1/2}\{(H_B-I)T^{\rm res}-H_BYW^{\rm res}\}$.
For every unit vector $x$, Minkowski's inequality in endpoint probability
space gives
\[
 \sqrt{\mathbb E_1\|Rx\|^2}
 \le b_*\{\beta_*\sqrt{\mathbb E_1\|T^{\rm res}x\|^2}
       +\kappa_*v_*\sqrt{\mathbb E_1\|W^{\rm res}x\|^2}\}
 \le r_{\rm ms}.
\]
The pointwise alternative gives $r_{\rm pt}$ under the same probability
law, with no extra density factor. Thus
$0\preceq\mathbb E_1R^\top VR\preceq e_RI$.
Quadratic-form Cauchy--Schwarz bounds the cross term in
\eqref{eq:common-error-identity} by $2\sqrt{q_Je_R}$.
Its density term lies between $-a_{\rm den}\bar Q$ and
$a_{\rm den}\bar Q$. Adding proves the raw bound. Congruence by $T_1$
proves the normalized bound since $\|T_1\|^2=1+\|Z_1\|^2$.
It also bounds every raw response's error divided by its endpoint Full loss.

All moments above are explicit finite sums. They can also be reconstructed
by coefficient extraction without enumerating count configurations.
For $r$ distinct specified row labels in a group and $q$ inverse factors,
use the polynomial
\[
 F_{r,q}(z)=\sum_{k=0}^n\binom nk(p+c_0k)^{1-q}
                 \frac{(k)_r}{(n)_r}z^k.
\]
An impossible row-label request is zero before any factorial division.
The expectation of a product with group profiles $(r_j,q_j)$ is
$[z^s]\prod_jF_{r_j,q_j}(z)$ divided by
$[z^s]F_{0,0}(z)^d$. Repeated row selectors collapse to one selector.
Expand $J_S$, $D^{\rm fit}_S$ and $t_S$ from their displayed formulas,
then apply this rule term by term. This reconstructs $\bar Q$ and its
score moment. For the residual moment, the factor $A$ in
$(D^{\rm fit})^\top A D^{\rm fit}$ cancels one inverse factor.
Its selector and inverse degrees are at most $(4,3)$, or $(5,4)$ after
score insertion. If column $a$ is in group $h$ with signed whitened
amplitude $\alpha_a$, then
\[
 (W_S^{\rm res})_{ia}=\mathbf1_i\mathbf1_a E_{ih}\alpha_a/(p+c_0K_h),
\]
which gives degrees at most $(3,2)$, or $(4,3)$ with the score.
These formulas specify every residual moment used in
\eqref{eq:endpoint-error}; an endpoint covariance is not an input.

For completeness, the supplied-penalty correction is also exact.
If $\bar L_1=\mathbb E_1L_{1,S}$ and $B_{\rm mean}=\bar L_1-L_{*,1}$,
subtract $\Delta^\top VB_{\rm mean}+B_{\rm mean}^\top V\Delta$ from
\eqref{eq:common-surrogate} for Full-centered MSE, and subtract
$B_{\rm mean}^\top VB_{\rm mean}$ as well for centered covariance.
Their true risks receive exactly the same corrections, so the remainder
\eqref{eq:endpoint-error} is unchanged. The mean can be computed exactly:
with $G=Z_1^\top Z_1$ and $c_j=\binom{m-j}{s-j}$,
\[
 N_1=\sum_jc_j[t^j]\det(P_1+tG),\qquad
 \bar L_1=\frac{\sum_jc_j[t^j]\{t\operatorname{adj}(P_1+tG)\}}{N_1}Z_1^\top.
\]
The determinant identity follows by counting how many $s$-sets contain
each row minor. Differentiating it with respect to $Z_1$, while holding
$P_1$ fixed, gives the mean formula. Expanding the squared fits gives
the corrections above. Under exact calibration $B_{\rm mean}=0$.
This distinction prevents a supplied fit's mean defect from being
silently called sampling covariance.

\subsection{Explicit majorants and a terminating radius choice}
Normalize $c=\max_jW_{jj}=1$. Then $\|W\|\le d$ and
$\|V\|\le w=d/\lambda_0$. Fix a budget and put
\[
 z=\sqrt{n/\lambda_0},\quad
 r_0=\min\{z/2,p/2,1\},\quad p_f=p/2,\quad z_+=z+r_0,
 \quad t_+=\sqrt{1+z_+^2}.
\]
For $h=\|E\|_F+\|\Delta P\|\le r_0$, the straight feature/penalty
path is full rank with penalty at least $p_fI$. Define
\begin{align*}
 b_0&=(2\sqrt p)^{-1}+1/2,&b_f&=(2\sqrt{p_f})^{-1}+1/2,\\
 C_D&=5/(4p_f)+5/4+1/(2p_f^{3/2}),&C_F&=t_+C_D+b_fz_+,\\
 L_{\rm law}&=2\sqrt d/\sqrt{p_f}+d/p_f,
 &A_M&=2b_ft_+C_F+b_f^2t_+^2L_{\rm law}/2,\\
 Q_B&=C_F+b_ft_+L_{\rm law},&L&=w(A_M+Q_B^2r_0).
\end{align*}
Then, uniformly over the normalized PSD queries,
\begin{equation}\label{eq:joint-majorant}
 \|M_1-M_0\|\le\eta(h):=Lh.
\end{equation}
Here this statement can refer either to centered covariance or Full-centered
MSE with a supplied penalty. To check it, the selected and Full derivative
bounds from Appendix~\ref{app:transfer} give
$\|\dot F_S\|\le C_Dh$ for $F_S=L_S-L_*$ and $\|F_S\|\le b_f$.
The root derivative gives $\|(F_ST)^{\boldsymbol\cdot}\|\le C_Fh$.
The feature and penalty parts of the log-determinant derivative are at
most $2\sqrt d\|E\|_F/\sqrt{p_f}$ and $d\|\Delta P\|/p_f$.
Thus total variation is at most $L_{\rm law}h/2$ by
\eqref{eq:law-tv}. The PSD MSE atoms give the bound $wA_Mh$.
At the base $\mathbb E_0F_S=0$, so
$\|\mathbb E_1F_{1,S}T_1\|\le Q_Bh$.
Centering subtracts a PSD mean square of norm at most $wQ_B^2h^2$.
This proves \eqref{eq:joint-majorant}, including along a path that is
not itself calibrated.

The endpoint remainder has an explicit common majorant. For $0\le h\le r_0$
put
\begin{align*}
 v_h&=h/\sqrt p,&u_h&=2v_h+h/p,&\delta_h&=u_h+v_h^2,\\
 \rho_h&=v_h^2+\sum_{j=2}^d\binom dj\delta_h^j,
 &r_h&=\frac{\delta_h(v_h+u_h)+v_h^2}{\sqrt p(1-\delta_h)},\\
 q_h&=\left(\frac{1+\delta_h}{1-\delta_h}\right)^d,
 &d_h&=(v_h+u_h)/\sqrt p,&j_h&=b_0+d_h.
\end{align*}
When $\delta_h<1$, define
\begin{equation}\label{eq:common-remainder-majorant}
 \bar e(h)=[1+(z+h)^2]w
 \left\{\frac{2\rho_hj_h^2}{(1-\delta_h)^d}
       +2\sqrt{q_h}\,j_hr_h+r_h^2\right\}.
\end{equation}
Then $e_{\rm end}\le\bar e(h)$, and $\bar e(h)=O(h^2)$.
Indeed $\|A^{-1}\|\le1/p$, $\|U\|\le1$, $\|Y\|\le v_h$,
$\operatorname{tr}YY^\top\le v_h^2$, and $\|B_1\|\le u_h$.
The Neumann bound now gives $\|R_S\|\le r_h$, while
$\mu\ge(1-\delta_h)^d$ and $|\bar w|\le\rho_h$.
Also $\|J_S\|\le j_h$ and the density is at most $q_h$.
These bounds control the three terms of \eqref{eq:common-error-identity}
by the three terms in braces. They also dominate the individual choices
in \eqref{eq:endpoint-error}: its pointwise alternative is no larger
than $r_h$, and its first PSD averaging alternative gives
$q_J\le q_hw j_h^2$. This proves the stated domination, not just the
existence of another second-order error estimate. Every scalar expression
in \eqref{eq:common-remainder-majorant} is nondecreasing on this domain.

Let $\mathcal E_I$ be one of the spaces in Theorem~\ref{thm:enlarged},
and let $P_0^I,P_1^I$ project onto $\mathcal E_I$ and
$K_1^{1/2}\mathcal E_I$. Since $K_0^{1/2}=I$ on every contrast,
the metric-root bound and its inverse identity give
\begin{equation}\label{eq:enlarged-projector}
 \|P_1^I-P_0^I\|\le f(h):=\frac{\omega(h)}{1-\omega(h)},
 \qquad \omega(h)=z_+h<1.
\end{equation}
This follows by resolving the parallel and perpendicular parts of
$K_1^{1/2}|_{\mathcal E_I}$; its least singular value is at least
$1-\omega(h)$. Write $h_0=\alpha_s$ and let $\gamma$ be the uniform
base top-to-complement gap. When $f(h)<1$, projection of a base top vector
onto the endpoint trial space has base Rayleigh quotient at least
$h_0(1-f(h)^2)$. A unit vector in its complement has at most $f(h)^2$
energy in $\mathcal E_I$, so its base quotient is at most
$h_0-\gamma+\gamma f(h)^2$. Therefore the finite quantities from
Theorem~\ref{thm:enlarged} satisfy
\begin{align}
 g&\ge\gamma-(h_0+\gamma)f(h)^2-2\eta(h)-4\bar e(h),
 \label{eq:enlarged-gap-majorant}\\
 B&\le\eta(h)+2h_0f(h)+2\bar e(h).
 \label{eq:enlarged-cross-majorant}
\end{align}
The four remainder terms in the first line account for two compressed
surrogate errors and the explicit $2e$ margin. For the second line,
subtract the zero base cross block. The two projector changes contribute
at most $2h_0f(h)$; the operator change and surrogate error give the rest.

For all contrasts use $\gamma=\alpha_s-\beta_s>0$.
For a fixed threshold $\vartheta>0$, use
$\gamma=\min\{\alpha_s\vartheta,\alpha_s-\beta_s\}>0$.
Choose the first integer $j\ge0$ such that $h_j=r_0 2^{-j}$ satisfies
\begin{equation}\label{eq:enlarged-dyadic}
 \delta_{h_j}<1,\quad \omega(h_j)<1,\quad f(h_j)<1,\quad
 (h_0+\gamma)f(h_j)^2+2\eta(h_j)+4\bar e(h_j)\le\gamma/2.
\end{equation}
The search terminates because every error term tends to zero.
Monotonicity makes this a sufficient closed ball in all joint directions.
On it $g\ge\gamma/2$, $B=O(h)$, and the main theorem gives quadratic
regret and leakage $O(h+\sqrt\rho)$ for every response of deficit at
most $\rho$. None of these constants uses a query diagonal gap.

Finally, use the calibration constants from Appendix~\ref{app:transfer}.
At the balanced base, set
\[
 r_{{\rm cal},s}=\frac12\sqrt{n/\lambda_0},\qquad
 L_{X,s}=2v\mathcal L_s,
\]
where $v$ and $\mathcal L_s$ are defined in
\eqref{eq:radius-constants}--\eqref{eq:composite-constants} of that appendix.
These are separate from this appendix's joint radius $r_0$.
On the feature ball of radius $r_{{\rm cal},s}$, exact calibration gives
$\|\Delta P\|\le\|\Delta P\|_F\le L_{X,s}\epsilon$.
Consequently the explicit positive feature radius
\begin{equation}\label{eq:enlarged-feature-radius}
 r_{\rm unif}=
 \min_{d\le s<m}\left\{r_{{\rm cal},s},\frac{h_{j,s}}{1+L_{X,s}}\right\}
\end{equation}
works for all budgets and all queries with $\max_jW_{jj}=1$.
For a general scale $c>0$, risk and error bounds scale by $c$ and
leakage uses the normalized deficit $\rho/c$.
This radius is not uniform as $n,d,\lambda_0$ vary, as the sector gap
vanishes, or as $\vartheta$ tends to zero. No arithmetic or calibration
solver error is implicit in it; those would require additional verified
error bounds. The finite endpoint formulas remain valid outside this
sufficient neighborhood, though their bounds may be loose.

\section{The gain under recalibration and exact local certificates}\label{app:gain}
Section~\ref{sec:transfer} proves Theorem~\ref{thm:gain} with explicit constants. The shared-subset operator identity keeps a factor $t$ in the
perturbation bound; a comparison of the two base deficits then controls
the difference of sharp risks. Section~\ref{sec:mixture} gives the separate positive mean-share result at the balanced base. The following subsections sharpen the constants for
the three exact certificates in Table~\ref{tab:gain-certificates}.
Throughout, $n,d\ge2$, $d\le s<nd$, and the base consists of signed unit
coordinate replicas. The physical target is the fixed matrix
$\lambda_0I$, $\lambda_0>0$. Every endpoint uses its unique exactly
calibrated penalty in both the original-row determinant law and the ridge
fit. Only the subset is random. The query is any fixed nonzero PSD matrix.
All unmarked norms are operator norms; derivatives of the calibration
map use Frobenius norm on their input and output.

\subsection{A sharper finite spectral bound}\label{app:gain-envelope}
To certify a radius $R$, we need bounds $L_R,L_D,L_M$ on the entire ball.
The spectral gates below state what those bounds must satisfy. We then
construct them in this order: \eqref{eq:gain-cal-gates} places the feature
ball inside the calibration region; \eqref{eq:gain-base-energy} computes
derivative energies at the base; and \eqref{eq:gain-energy-transport}
extends the derivative bounds over the ball. Equations
\eqref{eq:gain-stack-norms} and \eqref{eq:gain-stack-constants} turn
these bounds into $L_R,L_D,L_M$, which enter
\eqref{eq:gain-spectral-gates}.

The following refinement of the two-eigenvector argument will be used
for the exact certificates. Suppose constants $L_R,L_D,L_M=L_R+TL_D$
satisfy \eqref{eq:gain-three-bounds} on a proposed radius $R$.
They may be sharper than \eqref{eq:gain-basic-constants}.
Write $f(t)=D-Qt$ and $\gamma_0=\gamma(0)$. Choose
\begin{align*}
 0<g_T&\le\min_{[0,T]}(\alpha-\beta),&
 k_C&\ge\sup_{[0,T]} f/\alpha,\\
 a_-&\ge\sup_{[0,T]}\max\{0,-j_{A\beta},-j_{A\gamma}\},&
 b_+&\ge\sup_{[0,T]}\max\{0,j_{B\beta},j_{B\gamma}\},
\end{align*}
where
\begin{align*}
 j_{A\beta}&=B_1+B_2t+fB_0/A,&
 j_{A\gamma}&=(-2+t)\gamma_0+f\gamma_0/A,\\
 j_{B\beta}&=B_1+B_2t+f\beta/\alpha,&
 j_{B\gamma}&=(-2+t)\gamma_0+f\gamma/\alpha.
\end{align*}
Let $\Pi_M$ denote the fixed base mean-sector projector. With the base
deficits used above, direct subtraction gives
\[
 N=t(f/A)H_A+R_A=t(f/\alpha)H_B+R_B,
 \quad R_A\succeq-tca_-\Pi_M,\quad R_B\preceq tcb_+\Pi_M.
\]
For clarity, the remainders vanish on contrasts. On means they are
$R_A/t=j_{A\beta}P_W+j_{A\gamma}R_W$ and
$R_B/t=j_{B\beta}P_W+j_{B\gamma}R_W$; their bounds use
$P_W+R_W\preceq cI$.

Put $e_R=L_R\epsilon$, $e_M=L_M\epsilon$. If $e_R<g$ and $e_M<g_T$,
projecting the two endpoint top-eigenvector equations onto the base mean
sector gives
\[
 \|\Pi_Mu\|\le\frac{e_R}{g-e_R},\qquad
 \|\Pi_Mv\|\le\frac{e_M}{g_T-e_M}.
\]
Indeed the ridge eigenvalue is at least $(A-e_R)c$, whereas its base
mean block is at most $B_0cI$; the mixture argument uses $\beta(t)cI$.
The forcing terms have norms at most $e_Rc,e_Mc$. With
$\mathcal G=h_R-h_t-tfc$, the same two variational comparisons as before
are
\[
 \mathcal G\le\|E_B-E_A\|-u^\top Nu,\qquad
 \mathcal G\ge-\|E_B-E_A\|-v^\top Nv.
\]
In the first, discard the nonnegative proportional $H_A$ term. In the
second, use $v^\top H_Bv\le2e_Mc$. This proves
\begin{equation}\label{eq:gain-phi}
 |\mathcal G|\le tc\Phi(\epsilon),\quad
 \Phi(\epsilon)=L_D\epsilon+\max\left\{
 a_-\left(\frac{L_R\epsilon}{g-L_R\epsilon}\right)^2,
 2k_CL_M\epsilon+b_+\left(\frac{L_M\epsilon}{g_T-L_M\epsilon}\right)^2
 \right\}.
\end{equation}
For fixed constants, $\Phi(\epsilon)/\epsilon$ is nondecreasing: each
nonlinear term divided by $\epsilon$ is a nonnegative multiple of
$\epsilon/(g-L\epsilon)^2$. Therefore the gates
\begin{equation}\label{eq:gain-spectral-gates}
 L_RR<g,\qquad L_MR<g_T,\qquad \Phi(R)\le D/4
\end{equation}
give \eqref{eq:gain-modulus} on the entire ball with $C=\Phi(R)/R$
and a gain of at least $Dtc/4$.

All scalar choices can be made with rational polynomial operations.
Take $\alpha_{\min}=\min_{[0,T]}\alpha$, $g_T=\min_{[0,T]}(\alpha-\beta)$,
and $k_C=D/\alpha_{\min}$. A quadratic minimum uses its endpoints and
any interior vertex. The two $j_A$ are affine, so their endpoints give
$a_-$. For $b_+$, multiply each $j_B$ by $\alpha$, take the maximum of
zero and all Bernstein coefficients of the resulting polynomials on
$[0,T]$, then divide by $\alpha_{\min}$. Specifically, for
$p(t)=\sum_{j=0}^Np_jt^j$, those coefficients are
$\sum_{j=0}^ip_jT^j\binom ij/\binom Nj$, $0\le i\le N$.
The nonnegative Bernstein basis sums to one, so this is a uniform bound.

\subsection{A finite calibration derivative difference}\label{app:gain-cal-difference}
We supply the local calibration bound needed by the sharper certificate;
a bound on two derivative norms alone would not suffice. In this and the
remaining subsections, write $\zeta=n/\lambda_0$ (distinct from the
mean-loss coefficient $b$), $z=\sqrt\zeta$, $p_0=r/\lambda_0$, and
$q_0^{\rm cal}=\zeta/p_0$. Here $rI$ is the physical base penalty.
The marked-set distribution in Section~\ref{sec:calibration} has
cardinality $K$ with weights
\[
 \binom dk\binom{m-k}{s-k}(q_0^{\rm cal})^k,
 \quad 0\le k\le d.
\]
Define its positive curvatures and the base derivative coefficients by
\begin{align*}
 \lambda_{\parallel}&=\operatorname{Var}(K)/d,&
 \lambda_{\perp}&=\mathbb E[K(d-K)]/[d(d-1)],&
 \mu_*&=\min\{\lambda_{\parallel},\lambda_{\perp}\},\\
 t_j&=p_0/\zeta-p_0/[(1+\zeta)^2\lambda_j],&
 L_0^{\rm cal}&=\max_j|t_j|.&&
\end{align*}
The marked covariance has eigenvalues $\lambda_{\parallel}$ on scalar
directions and $\lambda_{\perp}$ on traceless diagonal directions.
Differentiating the moment equation gives
\begin{equation}\label{eq:gain-base-cal-derivative}
 DP(\zeta I)[F]=t_{\perp}\{F-\operatorname{tr}(F)I/d\}
                      +t_{\parallel}\operatorname{tr}(F)I/d.
\end{equation}
Orthogonal equivariance extends the diagonal formula to all symmetric
$F$. Positivity of both curvatures follows from the full support of the
marked-set law.

Choose a log radius $a>0$, upper bounds $E_a\ge e^a$ and
$\chi\ge e^{2\sqrt d a}$, and
\[
 0<\delta\le\min\{\zeta/2,\mu_*a(1+\zeta/2)^2/(2\chi)\}.
\]
Set $\ell=\zeta-\delta$, $u=\zeta+\delta$ and define
\begin{align*}
 R_{\max}&=(1+\ell)^{-2},&R_0&=(1+\zeta)^{-2},\\
 \Delta R_*&=\max\{R_{\max}-R_0,R_0-(1+u)^{-2}\},&
 \Delta P_*&=p_0\max\{E_au/\zeta-1,1-\ell/(E_a\zeta)\},\\
 J_*&=\frac{E_a-1}{q_0^{\rm cal}}
 +\frac{\Delta P_*\chi R_{\max}}{\mu_*}
 +\frac{p_0(\chi-1)R_{\max}}{\mu_*}
 +\frac{p_0\Delta R_*}{\mu_*},\\
 L_{\rm cal}&=L_0^{\rm cal}+J_*,\\
 p&=\max\{(s-d+1)/(m-d+1),\ell/(q_0^{\rm cal}E_a)\}.
\end{align*}
We claim, throughout $\|B-\zeta I\|_F\le\delta$,
\begin{equation}\label{eq:gain-cal-difference}
 P(B)\succeq pI,\qquad
 \|DP(B)-DP(\zeta I)\|_{F\to F}\le J_*.
\end{equation}
Here are details that also cover repeated spectra. Within the radius-$a$
log-parameter ball about $\log(q_0^{\rm cal})I$, each marked log weight
changes by at most $\sqrt d a$. Thus its normalized density ratio lies
between $\chi^{-1}$ and $\chi$. The identity
$\operatorname{Var}(v^\top\mathbf1_J)=\min_c\mathbb E(v^\top\mathbf1_J-c)^2$
places its covariance $\mathcal C$ between $\chi^{-1}\mathcal C_0$ and
$\chi\mathcal C_0$. Coordinate swaps, as in the proof of
\eqref{eq:uniform-curvature}, extend the lower Hessian bound
$\mu_*/\chi$ to symmetric matrix directions.
The target moment $B(I+B)^{-1}$ changes by at most
$\delta/(1+\zeta/2)^2\le\mu_*a/(2\chi)$. On the log-ball boundary,
the radial component of the gradient of the convex calibration objective
is therefore at least $(\mu_*/\chi)a^2-(\mu_*a/(2\chi))a>0$.
Its minimum on the closed ball is interior; uniqueness identifies it
with the calibrated solution. In particular every odds eigenvalue lies
in $[q_0^{\rm cal}/E_a,q_0^{\rm cal}E_a]$, which proves the penalty floor.

In Gram eigenvalue coordinates $b_j$, the calibration Jacobian is
\[
 J_{\rm cal}=\operatorname{diag}(1/q_j)
 -\operatorname{diag}(p_j)\mathcal C^{-1}
                         \operatorname{diag}((1+b_j)^{-2}).
\]
Write $J_{{\rm cal},0}$ for its value at the balanced base.
The covariance comparison gives
$\|\mathcal C^{-1}\|\le\chi/\mu_*$ and
$\|\mathcal C^{-1}-\mathcal C_0^{-1}\|\le(\chi-1)/\mu_*$.
Subtract the baseline Jacobian by changing, in order, $1/q_j$, $p_j$,
$\mathcal C^{-1}$, and the final diagonal factor. Their four bounds are
exactly the four terms of $J_*$, so
$\|J_{\rm cal}-J_{{\rm cal},0}\|\le J_*$.
For two distinct Gram eigenvalues, join their vector to its coordinate
swap. This segment stays in the Gram ball. With
$e=(e_i-e_j)/\sqrt2$, equivariance and integration give
\[
 \frac{p_i(b)-p_j(b)}{b_i-b_j}-t_{\perp}
   =\int_0^1 e^\top\{J_{\rm cal}(b_{\rm swap}+x(b-b_{\rm swap}))-J_{{\rm cal},0}\}e\,dx.
\]
Thus the off-diagonal derivative coefficients satisfy the same bound.
The diagonal and off-diagonal symmetric subspaces are orthogonal and
invariant, so their bounds combine by a maximum. The baseline map
\eqref{eq:gain-base-cal-derivative} is independent of eigenbasis, and
continuity covers repeated eigenvalues. This proves
\eqref{eq:gain-cal-difference}.

For a feature radius $R$ impose
\begin{equation}\label{eq:gain-cal-gates}
 R<z,\qquad (2z+R)R\le\delta.
\end{equation}
Let $v=z+R$, $L_X=2vL_{\rm cal}$, $L_{X0}=2zL_0^{\rm cal}$. The entire
feature ball is full rank, lies in the Gram ball, and, for every unit
Frobenius feature direction, satisfies
\begin{equation}\label{eq:gain-dot-cal-difference}
 \|\dot P\|_F\le L_X,\qquad
 \|\dot P-\dot P_0\|_F\le\Delta\dot P_*=2vJ_*+2L_0^{\rm cal}R.
\end{equation}
The last inequality adds and subtracts
$DP(\zeta I)[Z^\top E+E^\top Z]$ and uses the input difference bound
$2R$. Thus it covers noncommuting feature changes as well.

\subsection{Derivative energies of the probability-weighted atoms}
We now construct finite constants for \eqref{eq:gain-three-bounds}.
At a design $Z$ in the ball, keep $A_S,L_S,L,I_S,T_Z$ as defined above,
and include physical scaling in the atoms:
\[
 u_{R,S}=(L_S-L)T_Z/\sqrt{\lambda_0},\qquad
 u_{G,S}=(LI_S-L_S)T_Z/\sqrt{\lambda_0}.
\]
Dots in this subsection are derivatives in a unit Frobenius direction
$E$. If $a_S=\operatorname{tr}(A_S^{-1}\dot A_S)$ and
$\ell_S=a_S-\mathbb E_pa_S$, exact calibration gives
\begin{equation}\label{eq:gain-centered-score}
 \ell_S=2\operatorname{tr}((L_S-L)E)
  +\operatorname{tr}\{(A_S^{-1}-\mathbb E_pA_S^{-1})\dot P\},
 \quad \eta_i:=\partial\log\pi_i=\mathbb E_p[\ell_S\mid i\in S].
\end{equation}
Indeed the design part of $a_S$ is $2\operatorname{tr}(L_SE)$ and
$\mathbb E_pL_S=L$. Also $\dot\pi_i=\mathbb E_p[\mathbf1_i\ell_S]$
and $\dot I_S=-I_S\operatorname{diag}(\eta_i)$.

Stack $\sqrt{p_S}W^{1/2}u_{R,S}$ into a linear response map
$\mathcal B_R$ and define $\mathcal B_G$ likewise. Its derivative stacks
$\sqrt{p_S}W^{1/2}V_{R,S}$, where
\[
 V_{R,S}=\dot u_{R,S}+\ell_Su_{R,S}/2,
 \qquad V_{G,S}=\dot u_{G,S}+\ell_Su_{G,S}/2.
\]
Then $M_R=\mathcal B_R^\top\mathcal B_R$ and
$\mathcal J_t=\mathcal B_R^\top\mathcal B_G+\mathcal B_G^\top\mathcal B_R
+t\mathcal B_G^\top\mathcal B_G$. If their norms are bounded by
$\sqrt c\,r_*,\sqrt c\,g_*$ and derivative norms by
$\sqrt c\,v_R,\sqrt c\,v_G$, product differentiation proves
\begin{equation}\label{eq:gain-stack-constants}
 L_R=2r_*v_R,\qquad L_D=2(v_Rg_*+r_*v_G)+2Tg_*v_G,
 \qquad L_M=L_R+TL_D.
\end{equation}

Let $E_a$, $1\le a\le md$, be all coordinate feature directions. At the
base define the two positive matrices
\begin{equation}\label{eq:gain-base-energy}
 \begin{gathered}
 H_R=\sum_a\mathbb E_0[V_R(E_a)^\top V_R(E_a)],\qquad
 H_G=\sum_a\mathbb E_0[V_G(E_a)^\top V_G(E_a)],\\
 v_{R0}=\sqrt{d\|H_R\|},\qquad v_{G0}=\sqrt{d\|H_G\|}.
 \end{gathered}
\end{equation}
For a unit direction $E=\sum_ae_aE_a$, vector Cauchy--Schwarz gives
$V(E)^\top V(E)\preceq\sum_aV(E_a)^\top V(E_a)$. Together with
$W\preceq dcI$, this proves the claimed base derivative bounds.

These finite matrices are particularly simple. Put
$\Pi_M=Z_0Z_0^\top/\zeta$, $\Pi_C=I-\Pi_M$ and
$\theta_0=\sqrt{1+\zeta}$. Each $H$ in \eqref{eq:gain-base-energy} has
\begin{equation}\label{eq:gain-energy-sectors}
 H=h_C\Pi_C+h_M\Pi_M,\qquad \|H\|=\max\{h_C,h_M\}.
\end{equation}
To see this, first remove row signs by an orthogonal diagonal matrix.
Within-group row permutations fix the base and force each diagonal block
to be a linear combination of $I$ and $\mathbf1\mathbf1^\top$, and each
off-diagonal group block to be constant. Flip one whole row group and its
matching feature coordinate. This signed permutation fixes the base and
forces off-diagonal blocks of $H$ to vanish. Group permutations identify
the remaining coefficients. The complete coordinate-direction sum is
invariant under these orthogonal changes of direction basis. All row
operations here are signed permutations, so they preserve or relabel the
original-row law. Conjugating back restores arbitrary row signs.

For an explicit evaluation of \eqref{eq:gain-base-energy}, use
\eqref{eq:gain-base-cal-derivative}, \eqref{eq:gain-centered-score}, and
\begin{align*}
 \dot A_S&=E^\top D_SZ_0+Z_0^\top D_SE+\dot P,&
 \dot L_S&=A_S^{-1}(E^\top D_S-\dot A_SL_S),\\
 \dot L&=(I+\zeta I)^{-1}(E^\top-\dot BL),&
 \dot B&=Z_0^\top E+E^\top Z_0,\\
 \dot T&=\tfrac12\Pi_C\dot U\Pi_C
 +\frac{\Pi_C\dot U\Pi_M+\Pi_M\dot U\Pi_C}{1+\theta_0}
 +\frac{\Pi_M\dot U\Pi_M}{2\theta_0},&
 \dot U&=EZ_0^\top+Z_0E^\top.
\end{align*}
Here $\dot U=\partial(T_Z^2)$.
Differentiate each atom as a product and add its half-score term.
For the rational fixtures below, every squared energy lies in
$\mathbb Q(\theta_0)$. Indeed $Z_0=U/\sqrt{\lambda_0}$ for an integer
replica matrix $U$; the selected and Full maps have this same scalar
factor, while $\dot L_S,\dot L$ are rational and $\dot T$ has that
factor times a matrix in $\mathbb Q(\theta_0)$. Thus each $V$ is
$1/\sqrt{\lambda_0}$ times a matrix in that field.
No numerical derivative or eigenvalue is needed.

Record the finite base norms
\[
 l_0=\max_S\|L_{S,0}\|,\quad
 u_{R0}=\max_S\|u_{R,S,0}\|,\quad u_{G0}=\max_S\|u_{G,S,0}\|,
 \quad q_0=\max_S\sqrt{\sum_a\ell_{S,0}(E_a)^2}.
\]
The row Grams of these base fit and atom maps are diagonal because the
groups have disjoint supports. Their norms and $q_0$ therefore admit
exact square-root bounds. Finally the base stack norms are at most
$\sqrt{cA}$ and $\sqrt{c\max\{Q,B_2,\gamma_0\}}$.
For the latter, the $G$ mean block is $B_2P_W+\gamma_0R_W$ and its
contrast coefficient is $Q$. Both mean coefficients are nonnegative.
Once $v_R,v_G$ hold on the ball, integration gives valid stack bounds
\begin{equation}\label{eq:gain-stack-norms}
 r_* =\sqrt A+Rv_R,\qquad
 g_* =\sqrt{\max\{Q,B_2,\gamma_0\}}+Rv_G.
\end{equation}

\subsection{Finite transport of derivative energies}\label{app:gain-finite-transport}
The following formulas bound differences from the base at every point in
the radius-$R$ ball, in every unit feature direction. Constants are fixed
on that ball. In addition to \eqref{eq:gain-cal-gates}, use
\[
 \begin{gathered}
 \theta=\sqrt{1+v^2},\qquad \pi_0=s/m,\qquad
 q_{\rm score}=\sqrt d\{2/\sqrt p+L_X/p\},\\
 \pi_f=\pi_0e^{-2q_{\rm score}R},\qquad h=e^{2q_{\rm score}R}-1,
 \qquad d_p=e^{q_{\rm score}R}.
 \end{gathered}
\]
The determinant derivative bounds prove $|a_S|\le q_{\rm score}$,
$|\ell_S|\le2q_{\rm score}$ and
$e^{-2q_{\rm score}R}\le p_S/p_{S,0}\le e^{2q_{\rm score}R}$.
They also prove $\pi_i\ge\pi_f>0$. Set
\[
 D_R=5/(4p)+5/4+L_X/(2p^{3/2}),\quad
 f_0=z/(1+z^2),\quad b_{R0}=\sqrt{\lambda_0}u_{R0},\quad
 b_{G0}=\sqrt{\lambda_0}u_{G0}.
\]
The last two constants bound the raw base atoms since $\|T_0^{-1}\|\le1$.
Use the following nonnegative bounds, with $\Delta\dot P_*$ from
\eqref{eq:gain-dot-cal-difference}:
\begin{align}
 \Delta A&=(2z+R+L_X)R,&
 \Delta I_A&=\Delta A/(pp_0),\nonumber\\
 \Delta L&=(D_R-5/4)R,&\Delta F&=5R/4,\nonumber\\
 \Delta\dot L&=\Delta I_A[1+(2z+L_{X0})l_0]\nonumber\\
       &\quad+\{(2R+\Delta\dot P_*)l_0+(2v+L_X)\Delta L\}/p,\nonumber\\
 f_{\rm floor}&=1+(z-R)^2,&
 \Delta I_F&=(2z+R)R/[f_{\rm floor}(1+z^2)],\nonumber\\
 \Delta\dot F&=\Delta I_F(1+2zf_0)
          +\{2Rf_0+2v\Delta F\}/f_{\rm floor}.
 \label{eq:gain-fit-transport}
\end{align}
Here $\Delta I_A$ bounds the selected inverse difference, $\Delta L$
the selected fit difference, and $\Delta F$ the Full fit difference.
Resolvent subtraction gives $\Delta I_A$. Integrating the selected and
Full derivatives gives $\Delta L,\Delta F$. For the derivative difference,
write $\dot L_S=A_S^{-1}C_S$ with
$C_S=E^\top D_S-\dot A_SL_S$. Subtract as
$(A_S^{-1}-A_{S,0}^{-1})C_{S,0}+A_S^{-1}(C_S-C_{S,0})$.
The base factor has norm at most $1+(2z+L_{X0})l_0$ and the other
difference at most $(2R+\Delta\dot P_*)l_0+(2v+L_X)\Delta L$.
This proves the selected formula, preserving matrix order. The Full
formula is the same subtraction with constant penalty $I$, endpoint
inverse floor $f_{\rm floor}$, and base fit norm $f_0$.

For the centered score and inclusions, define
\begin{align}
 \Delta a&=\sqrt d\{\Delta I_A(2z+L_{X0})+(2R+\Delta\dot P_*)/p\},\nonumber\\
 \Delta\ell&=2\Delta a+hq_0,&q&=q_0+\Delta\ell,\nonumber\\
 \Delta\eta&=\Delta\ell+2hq_0\pi_0/\pi_f,&
 \Delta I&=h/\pi_f,\qquad B_I=1/\pi_f,\nonumber\\
 \Delta\dot I&=\Delta Iq_0+B_I\Delta\eta,&
 \Delta H&=\Delta F B_I+f_0\Delta I,\nonumber\\
 \Delta\dot H&=\Delta\dot F B_I+(5/4)\Delta I
                  +\Delta FqB_I+f_0\Delta\dot I.
 \label{eq:gain-inclusion-transport}
\end{align}
Trace Cauchy--Schwarz bounds $|a_S-a_{S,0}|$ by $\Delta a$.
Centering costs twice this bound. The remaining law change is
$(\mathbb E_p-\mathbb E_0)\ell_{S,0}$, whose absolute value is at most
$hq_0$. This proves the bound on $\Delta\ell$ and $|\ell_S|\le q$.
For $\eta_i=\mathbb E_p[\ell\mid i\in S]$, first change the score
inside the endpoint conditional expectation, at cost $\Delta\ell$.
Then change the conditional average of the base score. Its numerator
change is at most $h\pi_0q_0$, and its denominator change costs another
$h\pi_0q_0/\pi_f$ after division. This proves $\Delta\eta$; the entire
score difference is not divided by the inclusion floor.
Since $|\pi_i-\pi_0|\le h\pi_0$, reciprocals differ by at most
$\Delta I$. Subtracting $\dot I_S=-I_S\operatorname{diag}(\eta_i)$
gives $\Delta\dot I$. The last two formulas follow by subtracting
$LI_S$ and its derivative as ordered products. Here
$\|\dot I_S\|\le q/\pi_f$ and the base Full derivative is at most $5/4$.

Finally define the metric and atom differences
\begin{align}
 \Delta T&=vR,&\Delta\dot T&=R+z\Delta T,\nonumber\\
 \Delta F_R&=\Delta L+\Delta F,&\Delta G&=\Delta H+\Delta L,\nonumber\\
 \Delta u_R&=(\Delta F_R\theta+b_{R0}\Delta T)/\sqrt{\lambda_0},&
 \Delta u_G&=(\Delta G\theta+b_{G0}\Delta T)/\sqrt{\lambda_0},\nonumber\\
 d_{R0}&=5/4+5/(4p_0)+L_{X0}/(2p_0^{3/2}),&
 d_{G0}&=(5/4)/\pi_0+f_0q_0/\pi_0+d_{R0}-5/4,\nonumber\\
 \Delta\dot u_R&=\{(\Delta\dot L+\Delta\dot F)\theta
       +d_{R0}\Delta T+\Delta F_Rv+b_{R0}\Delta\dot T\}/\sqrt{\lambda_0},
       \nonumber\\
 \Delta\dot u_G&=\{(\Delta\dot H+\Delta\dot L)\theta
       +d_{G0}\Delta T+\Delta Gv+b_{G0}\Delta\dot T\}/\sqrt{\lambda_0},
       \nonumber\\
 x_R&=\Delta\dot u_R+(\Delta\ell/2)u_{R0}+(q/2)\Delta u_R,&
 x_G&=\Delta\dot u_G+(\Delta\ell/2)u_{G0}+(q/2)\Delta u_G,\nonumber\\
 v_R&=d_pv_{R0}+\sqrt d\,x_R,&
 v_G&=d_pv_{G0}+\sqrt d\,x_G.
 \label{eq:gain-energy-transport}
\end{align}
To check the metric bound, subtract the Sylvester equations for $\dot T$
and $\dot T_0$. Since both metrics are at least $I$, the Sylvester inverse
has norm at most $1/2$. The right side consists of the derivative of
$ZZ^\top-Z_0Z_0^\top$, of norm at most $2R$, and two products of the
metric difference with $\dot T_0$, of total norm at most $2z\Delta T$.
This gives $\Delta\dot T$. Integrating $\|\dot T\|\le v$ gives
$\Delta T$. No commutation is assumed.
Subtracting the products $F_RT_Z$, $GT_Z$ and their derivatives proves
the four atom bounds. The base HT derivative uses the actual base
inclusion derivative, bounded by $q_0/\pi_0$, in $d_{G0}$.
Subtracting the half-score product then gives $x_R,x_G$ as bounds for
$V_R-V_{R,0}$ and $V_G-V_{G,0}$.
In the derivative stack, split $\sqrt{p_S}V_S$ into its base-atom part
and its difference part. The first has norm at most $d_p$ times its
base stack norm since $p_S/p_{S,0}\le d_p^2$; the second is bounded
using $W\preceq dcI$. The triangle inequality proves the final two
formulas in \eqref{eq:gain-energy-transport}.

Equations \eqref{eq:gain-energy-transport}, \eqref{eq:gain-stack-norms}
and \eqref{eq:gain-stack-constants} now give finite constants satisfying
\eqref{eq:gain-three-bounds} on the entire ball. Together with
\eqref{eq:gain-cal-gates} and \eqref{eq:gain-spectral-gates}, they prove
both the two-sided gain modulus and strict improvement. This derivation
has included calibration, fits, law, actual inclusions, and the Full
metric at every point of the ball.

\subsection{Three exact radius certificates}\label{app:gain-certificates}
Table~\ref{tab:gain-certificates} gives three matched bases. Each row
uses its own interval $0<t\le T$ from \eqref{eq:mixture-step}, and all have
$R=1/1000$ in $\|H/\sqrt{\lambda_0}\|_F$. For every dense or sparse
change in that closed ball, every nonzero PSD query and every $0<t\le T$,
\[
 |h_R-h_t-t(D-Qt)c|\le C t\epsilon c,\qquad h_R-h_t\ge Dtc/4>0.
\]
The penalty at each endpoint is assumed exactly calibrated. The proof
does not certify an approximate calibration solver.

\begin{table}[ht]
\centering
\caption{Exact sufficient neighborhoods. The penalty $rI$ is the matched
base penalty, and $\lambda_0I$ is the fixed physical Full target.
The entries in the last column are strict rational upper bounds for
$\Phi(R)$, not rounded decimal estimates.}\label{tab:gain-certificates}
\begin{tabular}{cccccc}
\toprule
$(n,d,s)$ & $r$ & $\lambda_0$ & $T$ & $R$ & $\Phi(R)$ upper bound\\
\midrule
$(2,2,2)$ & $2/3$ & $5/3$ & $1089/3890$ & $1/1000$ & $291/100000$\\
$(2,2,3)$ & $5/4$ & $55/31$ & $28665/123952$ & $1/1000$ & $111/100000$\\
$(2,3,3)$ & $3/2$ & $363/101$ & $98972519/449051012$ & $1/1000$ & $209/100000$\\
\bottomrule
\end{tabular}
\end{table}

We describe the exact calculation so that the table is reproducible from
finite sums and inequalities. Use $a=1/100$ in
Section~\ref{app:gain-cal-difference}. Compute the physical mixture
coefficients by the count sums in Appendix~\ref{app:mixture}, and the
marked moments by its cardinality distribution above. The exact target
identity $\mathbb E[K_1/(r+K_1)]=n/(n+\lambda_0)$ checks matching at
the base. Sum over the $\binom ms$ original-row subsets and all $md$
coordinate directions using \eqref{eq:gain-centered-score} and
\eqref{eq:gain-base-energy}. Store each energy as a pair of rational
coefficients in $\mathbb Q(\sqrt{1+n/\lambda_0})$ and check its full
matrix form \eqref{eq:gain-energy-sectors}. The exact coefficients,
energy pairs, scalar polynomial bounds, all transport constants, and
rational gate slacks are supplied in \texttt{gain\_certificates.json}, with evaluator
\texttt{gain\_certificates.py} and separate exact check
\texttt{gain\_independent\_checks.py} in the reproducibility package. They are inputs and
outputs of the formulas above, not substituted proof premises.

For completeness, the enclosure rules use only integer and rational
arithmetic. For $x\ge0$, let
$k=\lfloor\sqrt{x}\,10^{18}\rfloor$, found by integer squares. Then
$k/10^{18}\le\sqrt x\le(k+1)/10^{18}$, with equality used when exact.
Respect the sign of the second coefficient when enclosing a quadratic
field element. For $0\le x<20$ use
\[
 e^x\le\sum_{j=0}^{18}\frac{x^j}{j!}
            +\frac{x^{19}/19!}{1-x/20}.
\]
Successive tail ratios are at most $x/20$, which proves this bound.
Round positive upper quantities upward and floors downward on the same
rational grid. Take
\[
 \chi\ge e^{2\sqrt d a},\quad E_a\ge e^a,\quad
 \delta\le\min\{\zeta/2,\mu_*a(1+\zeta/2)^2/(2\chi)\}
\]
with these directed bounds, and check \eqref{eq:gain-cal-gates} exactly.
The lower bound for $f_{\rm floor}=1+(z-R)^2$ uses a lower bound for
$z>R$. The upper bound for $f_0$ uses an upper bound for $z$ only in its
numerator and the exact denominator $1+\zeta$. Keep $J_*$ separately;
subtracting rounded derivative norms would not bound their difference.
Apply the same directed operations to
\eqref{eq:gain-fit-transport}--\eqref{eq:gain-energy-transport} and
\eqref{eq:gain-phi}. All calibration, rank, penalty, inclusion, leakage,
and gain gates are strictly satisfied for the three rows. In particular
the rational bounds in Table~\ref{tab:gain-certificates} are below each
row's exact $D/4$.

This evaluator enumerates $6+4+20$ labelled subsets and $8,8,18$
coordinate directions, respectively. For general parameters this is a
combinatorial calculation. With $N=\binom ms$, storing dense selectors
and base atoms uses $O(N(m^2+dm+d^2))$ scalar storage, in addition to
energy accumulation. Rational bit costs and evaluation of endpoint
inclusions remain separate. The certificates cover precisely these
matched bases and weights $0<t\le T$. They do not certify the earlier
feature grid or establish practical prediction gains, responsewise
dominance, or an optimal radius.

\section{Finite evaluations and reproducibility}\label{app:experiments}
\subsection{Finite evaluation of the mixture}\label{sec:experiments}
We evaluate the constructive mixture without changing the sampling law or
label budget. The outcome is sharp query risk normalized by Full ridge
loss. All calculations use finite sums or count formulas with 80-digit
arithmetic. We recompute four prespecified cases at 120-digit precision.
These are floating evaluations,
not interval certificates. The protocol below specifies the
inputs and controls; the supplement provides code and per-case results.

\paragraph{Structured designs.}
We vary $n\in\{2,3,4,8\}$, $d\in\{2,4,8\}$ and
$\lambda_0/n\in\{.01,.1,1,10\}$. Minimum, middle and leave-one-out sizes
give 140 distinct target-size points. The queries are identity, one
coordinate, flat rank one and tilted rank one. We also evaluate the
query-uniform criterion $\max\{\alpha(t),\beta(t)\}$.
Table~\ref{tab:gain} compares the theorem's weight $T$, a common weight,
and the query-uniform optimum within the mixture family. For each fixed
$(n,d,\lambda_0)$, the common weight is the smallest theorem weight $T_s$
among that triple's listed sample sizes. All three reported weights improve
on ridge for the four tested queries at every grid point.
The explicit weight is conservative: its median gain is 1.24\%, compared
with 4.50\% for the optimum.

\begin{table}[t]
\caption{Sharp-risk reduction from mixing on 140 structured target-size
points. The thresholds summarize this fixed grid.}
\label{tab:gain}
\centering
\begin{tabular}{lrrr}
\toprule
Weight & Median & $\ge1\%$ & $\ge5\%$\\
\midrule
Theorem $T$ & 1.24\% & 80 & 16\\
Common $T$ & 0.98\% & 70 & 12\\
Optimal mixture & 4.50\% & 121 & 62\\
\bottomrule
\end{tabular}
\end{table}

\paragraph{Strong controls.}
Improvement over ridge is not dominance over other procedures. Across the
560 physical query cases, $T$ wins 402 comparisons with same-law HT, 432
with uniform-subset HT, and 297 with query-aware quota sampling. The latter
two change the sampling law. For coordinate queries, quota sampling has
zero risk at 92 of the 140 points. Thus an all-query improvement within a fixed law
does not settle query-specific sampling design.

\paragraph{Perturbed designs.}
For a four-row, two-feature base, we perturb features in two directions
at signed magnitudes $.001$, $.01$ and $.1$. Three target penalties and
two sizes give 72 nonzero design points, each with four queries. We
recalibrate the whole procedure and retain the base mixture weights.
The base $T$ improves all 288 query cases; its smallest gain is 0.018\%.
The base optimal weight loses four cases, by up to 0.350\%.
The gain theorem in Section~\ref{sec:transfer} gives a separate local
guarantee. It does not certify this grid: its sufficient neighborhoods
and improving weight intervals have their own assumptions. The exact certificates
in Appendix~\ref{app:gain} use three different matched bases and targets.
These evaluations do not test the general-design LOO or ETF geometry.

\subsection{Protocol and complete comparisons}
The structured calculation uses 48 fixed-target families. For each $(n,d)$,
the size set is $\{d,d+\lfloor(nd-d)/2\rfloor,nd-1\}$ with duplicates
removed. There are 140 target-size points. The physical queries are
$I$, $e_1e_1^\top$, $\boldsymbol{1}\boldsymbol{1}^\top$ and $vv^\top$,
where $v_j=j/d$. Each has maximum diagonal entry one. The query-uniform
criterion is separate; it is not a fifth physical query.

The local weight is $T$ from Section~\ref{sec:mixture}. The common weight
is the minimum of these weights over the three-size grid of one target
family. The optimized weight minimizes $\max\{\alpha(t),\beta(t)\}$ over
$[0,1]$. It need not minimize risk for each known query. The same-law HT
control uses the original determinant sample. Uniform-subset HT samples
uniformly among all size-$s$ subsets. Quota sampling independently selects
$k_j$ rows uniformly within group $j$, with $1\le k_j\le n$ and
$\sum_jk_j=s$. Its allocation minimizes the sharp query risk over these
integer choices. Its risk is
\[
 \max_j W_{jj}\,\frac{a^2(n-k_j)}{nk_j(n-1)},
 \qquad a=\frac{n}{n+\lambda_0}.
\]
The code evaluates this optimization by a finite recurrence. A separate
threshold-feasibility calculation verifies every retained optimum.

\begin{table}[ht]
\caption{Mixture wins/ties/losses against strong controls. Each row has
140 paired target-size cases. Quota sampling uses the query and changes the law.}
\label{tab:controls}
\centering
\begin{tabular}{llrrr}
\toprule
Weight & Query & Same-law HT & Uniform HT & Quota\\
\midrule
$T$ & Identity & 82/0/58 & 93/0/47 & 116/0/24\\
& Coordinate & 82/0/58 & 93/0/47 & 30/0/110\\
& Flat & 140/0/0 & 140/0/0 & 116/0/24\\
& Tilted & 98/0/42 & 106/0/34 & 35/0/105\\
Optimal & Identity & 85/0/55 & 94/0/46 & 130/0/10\\
& Coordinate & 85/0/55 & 94/0/46 & 44/0/96\\
& Flat & 140/0/0 & 140/0/0 & 130/0/10\\
& Tilted & 103/0/37 & 111/0/29 & 51/0/89\\
\bottomrule
\end{tabular}
\end{table}

For perturbations, $X_0=(e_1,e_1,e_2,e_2)^\top$ and
$\lambda_0\in\{.2,2,20\}$, $s\in\{2,3\}$. Define
\[
 H_1=\frac{1}{\sqrt6}
 \begin{pmatrix}0&1\\1&-1\\-1&0\\1&1\end{pmatrix},\qquad
 H_2=\frac{1}{\sqrt5}
 \begin{pmatrix}1&0\\0&-1\\1&1\\-1&0\end{pmatrix}.
\]
We evaluate $X_0+\varepsilon H_i$ for
$\varepsilon\in\{\pm.001,\pm.01,\pm.1\}$. Both directions break exact
replication; $H_1$ preserves Gram isotropy while changing its scale, whereas
$H_2$ also changes anisotropy. Here $|\varepsilon|$ is the raw Frobenius
perturbation size; the target-whitened size in Theorem~\ref{thm:transfer}
is $|\varepsilon|/\sqrt{\lambda_0}$. At each point, the code reconstructs Full,
the calibrated positive definite penalty, every subset probability and fit,
the unequal row-inclusion probabilities, and the native normalization.
The same-law HT operator is
$(G+A_0)^{-1}X^\top\operatorname{diag}(\mathbf1_{i\in S}/\pi_i)$.
The raw mean-squared-error operator includes any finite mean defect;
covariance and squared bias are retained separately. Each risk is a largest
generalized eigenvalue. Physical queries and target penalties stay fixed.

The 72 nonzero points give 288 query cases per method. Six unperturbed
bases add 24 cases, kept separate from perturbation counts. The four
optimized-weight losses all have $\lambda_0=20$, $s=2$, flat query and
$|\varepsilon|=.1$. Their reductions are $-0.248281\%$ and $-0.179110\%$
for $H_1$, and $-0.349723\%$ and $-0.046204\%$ for $H_2$.
Against same-law HT, the base $T$ wins 108/288 cases and the base optimum
wins 124/288. The structured optimum lies at a contrast--mean crossing;
its finite perturbation losses do not establish a causal explanation
or rule out improvement on a smaller neighborhood.

\paragraph{Arithmetic and reproducibility.}
The calculations use Python 3.13.11 and mpmath 1.3.0 on a CPU, without
random seeds, model training or external datasets. Primary precision is
80 decimal digits. The scalar calibration residual gate is $10^{-60}$;
bisection has a 400-step cap. Four prespecified cases are replayed at 120
digits with agreement tolerance $10^{-55}$: structured $(n,d,\lambda_0/n,s)$
equal to $(2,2,.1,2)$ and $(8,8,10,63)$, and perturbed
$(\lambda_0/2,s,H,\varepsilon)$ equal to $(1,2,H_1,.1)$ and $(10,3,H_2,-.1)$.
The structured replays include their target families for the common weight.
All four replays pass. Wins and ties use
$10^{-30}\max\{1,|r_1|,|r_2|\}$. There are no statistical error bars.
The full run retains 6,148 method rows with no unresolved cases.

The supplementary package separates exact polynomial checks from these
floating calculations. It includes synthetic inputs in code, all per-case
risks, a complete rerun driver and a separate risk-reconstruction checker.
Long coefficient arrays and method-row records are supplied as machine-readable
files. No figure renderer is needed to reproduce the numerical claims.

\section{Source relations}\label{app:sources}
\paragraph{Law and mean.}
The identity
$\det(X_S^\top X_S+R)=\det(R)\det(I+X_SR^{-1}X_S^\top)$
identifies the law as a $k$-DPP on $I+XR^{-1}X^\top$
\citep[Section~5]{kulesza2012}. Conditioning a uniform-base regularized
DPP on its size gives the same law \citep[Definition~2]{derezinski2020design}.
The selected-inverse expectation of \citet[Lemma~1 and Corollary~1]{schreurs2020}
gives the fixed-penalty ridge mean by the Woodbury identity.
\paragraph{Inverse calibration.}
Our calibration theorem combines this mean identity with the interior
mean-map theorem for minimal exponential families
\citep[Proposition~3.2 and Theorem~3.3]{wainwright2008}.
The spectral expansion identifies a finite family with a truncated-cube
mean set. Its interior condition gives the effective-dimension threshold.
Uniqueness lifts from the mean vector to the shared matrix penalty.
The penalty bounds follow from binomial-coefficient ratios.
Thus calibration is a model-specific consequence of these established results.
Arbitrary SPD target scope follows by a change of coefficient coordinates.
\citet[Theorem~4.1]{loonis2019} construct fixed-size projection DPPs
with prescribed original-row inclusion probabilities.
That inverse-design result concerns row marginals.
Here calibration matches the selected ridge response operator.
\citet[Theorem~1]{anzai2026} use this mean-map principle to match item
inclusion probabilities in an unconditioned DPP with fixed diversity.
Here the matched moments are latent spectral activations under an exact
row budget, and they determine the ridge mean.
Our uniqueness concerns the penalty used jointly in sampling and fitting.
It is not identification of a generic $k$-DPP kernel from its law;
such kernels have scale and other invariances \citep{hino2026}.

\paragraph{Exact debiasing.}
\citet[Theorems~1 and~6]{derezinski2020debiasing} obtain exact ridge
debiasing with scaled regularization under surrogate sketching. Their
row count is random. We instead use exactly $s$ distinct original rows
and the same matrix penalty in sampling and fitting, with an arbitrary
SPD target. The effective-dimension threshold and debiasing goal have precedents.

\paragraph{Covariance geometry.}
\citet{derezinski2018} establish selected-OLS identities under ordinary
volume sampling and study a regularized reverse-deletion procedure.
At positive penalty, that sequential procedure generally has a different
endpoint law from our direct determinant law. The proof of their Theorem~18
also uses repeated coordinate rows and group counts for noise-averaged
ridge prediction risk. Our result instead concerns fixed-response
covariance under the direct calibrated subset law and its complete equality
space. \citet{rhee2026phase} studies
fixed-design OLS covariance attainment. The normalized worst-response
OLS design criterion of \citet[Definitions~6 and~8]{derezinski2019minimax}
is another close precedent. It divides by residual energy and optimizes
over designs and estimators. Here one law and estimator are fixed for each
comparison; the response is deterministic, covariance is centered over
the subset draw, $W$ is a fixed physical coefficient query, and positive
Full ridge loss normalizes every nonzero response. We give complete
equality spaces on balanced replicas at the stated budgets, an exact
attainability test for a general-design leave-one-out envelope, and a
flat-query criterion on existing real ETFs at two deletions.
For isotropic queries, we also give sharp risk at three deletions. The
leave-one-out envelope can be strict, and the ETF criterion does not
classify the risk of every nonflat query.

\paragraph{Frame erasures and ETF structure.}
\citet{goyal2001} study one-erasure average and worst reconstruction MSE
under zero-mean, equal-variance, uncorrelated quantization noise.
Their estimator uses linear dual-frame reconstruction. \citet{bodmann2005}
study frame reconstruction-error operator norms, including the
pair-uniform geometry of equiangular tight frames. Probability-weighted
erasure criteria for dual frames also appear in \citet{arati2024}.
These are different losses from the determinant-weighted, centered
covariance divided by Full ridge loss here. In particular, an ETF
calculation within a fixed frame class is not a claim that ETFs optimize
our risk over designs: \citet[Proposition~3.8]{bodmann2005} show that
two-uniform frames can maximize their average two-erasure error measure.
The independence of the projectors $x_ix_i^\top$
for noncollinear equiangular lines is established in
\citet[Section~5.1]{fickus2018}. The dimension of the flat-energy
physical-query space follows from this rank fact and tightness; our
two-deletion result instead concerns the covariance deficit and all its
equality responses. Flat energy $x_i^\top W x_i=\operatorname{tr}(W)/d$
is not the source's notion of a flat ETF.

\paragraph{Three ETF deletions.}
The triangle classes and their erasure spectra are classical
\citep{bodmann2005}. Their unequal determinant weights follow directly
from those spectra. Our theorem uses the weights in the mean and
centered second moment. It proves a strict normalized
sector gap and identifies every maximizing response.
It does not assume uniform triples or claim frame-design optimality.

\paragraph{Estimator mixing and restricted responses.}
Horvitz--Thompson weighting \citep{horvitz1952} and dependent vector
estimator averaging \citep{lavancier2016} are established methods.
\citet[Theorem~5.1]{loonis2019} optimize fixed linear survey weights
for specified auxiliary variables, using first- and second-order inclusion probabilities.
Our selected ridge inverse depends on the realized subset.
We use the contrast derivative and strict sector gap to choose one
feature-only interval of scalar mixture weights.
It improves every nonzero PSD query's sharp risk at the balanced base and
preserves the complete maximizing space.
The mean-share value follows by optimizing the contrast and mean blocks
separately.
The further comparison gives the exact all-query improvement phase on
the protected interval and an attaining adverse query.
It concerns deterministic responses and this mixture family.

\paragraph{Recalibrated transfer.}
Classical eigenspace bounds \citep{yu2015} apply after we control the
calibration, subset law, fit and Full normalization.
The full equality space then yields the raw-response restriction theorem.
The individual top-Ritz bound is also classical \citep[Theorem~2]{lili2005}.
We combine it with the moving Full metric and explicit remainders.
Retaining all base contrasts removes query-gap dependence at the cost of
a larger space. The gain theorem compares two moving sharp risks under
the same recalibrated procedure, including its actual inclusion probabilities.
Its coupled dependence on the mixture weight comes from bounding the
shared-sample difference. This is stronger than qualitative gain persistence
at one fixed positive weight, which follows from continuity.


\begin{thebibliography}{99}
\bibitem[Anzai and Hino(2026)]{anzai2026}
So Anzai and Hideitsu Hino.
Determinantal point process approximation under positive and negative dependence.
arXiv:2607.18849, 2026.

\bibitem[Arati et~al.(2024)Arati, Devaraj, and Mondal]{arati2024}
S. Arati, P. Devaraj, and Shankhadeep Mondal.
Optimal dual frames and dual pairs for probability modelled erasures using
weighted average of operator norm and spectral radius.
arXiv:2411.00502, 2024.

\bibitem[Bodmann and Paulsen(2005)]{bodmann2005}
Bernhard G. Bodmann and Vern I. Paulsen.
Frames, graphs and erasures.
\emph{Linear Algebra and its Applications}, 404:118--146, 2005.

\bibitem[Derezi\'nski et~al.(2019)Derezi\'nski, Clarkson, Mahoney, and Warmuth]{derezinski2019minimax}
Micha\l{} Derezi\'nski, Kenneth L. Clarkson, Michael W. Mahoney,
and Manfred K. Warmuth.
Minimax experimental design: Bridging the gap between statistical and
worst-case approaches to least squares regression.
In \emph{Proceedings of the Thirty-Second Conference on Learning Theory},
volume 99 of PMLR, pages 1050--1069, 2019.

\bibitem[Derezi\'nski and Warmuth(2018)]{derezinski2018}
Micha\l{} Derezi\'nski and Manfred K. Warmuth.
Reverse iterative volume sampling for linear regression.
\emph{Journal of Machine Learning Research}, 19(23):1--39, 2018.

\bibitem[Derezi\'nski et~al.(2020a)Derezi\'nski, Bartan, Pilanci, and Mahoney]{derezinski2020debiasing}
Micha\l{} Derezi\'nski, Burak Bartan, Mert Pilanci, and Michael W. Mahoney.
Debiasing distributed second order optimization with surrogate sketching
and scaled regularization.
In \emph{Advances in Neural Information Processing Systems}, volume 33, 2020.

\bibitem[Derezi\'nski et~al.(2020b)Derezi\'nski, Liang, and Mahoney]{derezinski2020design}
Micha\l{} Derezi\'nski, Feynman Liang, and Michael W. Mahoney.
Bayesian experimental design using regularized determinantal point processes.
In \emph{Proceedings of AISTATS}, volume 108 of PMLR, pages 3197--3207, 2020.

\bibitem[Fickus et~al.(2021)Fickus, Jasper, Mixon, and Peterson]{fickus2018}
Matthew Fickus, John Jasper, Dustin G. Mixon, and Jesse D. Peterson.
Hadamard equiangular tight frames.
\emph{Applied and Computational Harmonic Analysis}, 50:281--302, 2021.

\bibitem[Goyal et~al.(2001)Goyal, Kova\v{c}evi\'c, and Kelner]{goyal2001}
Vivek K. Goyal, Jelena Kova\v{c}evi\'c, and Jonathan A. Kelner.
Quantized frame expansions with erasures.
\emph{Applied and Computational Harmonic Analysis}, 10(3):203--233, 2001.

\bibitem[Hino and Yano(2026)]{hino2026}
Hideitsu Hino and Keisuke Yano.
From DPPs to $k$-DPPs: identifiability analysis via spectral decomposition.
arXiv:2605.25526, 2026.

\bibitem[Horvitz and Thompson(1952)]{horvitz1952}
D. G. Horvitz and D. J. Thompson.
A generalization of sampling without replacement from a finite universe.
\emph{Journal of the American Statistical Association}, 47(260):663--685, 1952.

\bibitem[Kulesza and Taskar(2012)]{kulesza2012}
Alex Kulesza and Ben Taskar.
Determinantal point processes for machine learning.
\emph{Foundations and Trends in Machine Learning}, 5(2--3):123--286, 2012.

\bibitem[Lam et~al.(2007)Lam, Postnikov, and Pylyavskyy]{lam2007schur}
Thomas Lam, Alexander Postnikov, and Pavlo Pylyavskyy.
Schur positivity and Schur log-concavity.
\emph{American Journal of Mathematics}, 129(6):1611--1622, 2007.

\bibitem[Lavancier and Rochet(2016)]{lavancier2016}
Fr\'ed\'eric Lavancier and Paul Rochet.
A general procedure to combine estimators.
\emph{Computational Statistics \& Data Analysis}, 94:175--192, 2016.
\href{https://doi.org/10.1016/j.csda.2015.08.001}{doi:10.1016/j.csda.2015.08.001}.

\bibitem[Li and Li(2005)]{lili2005}
Chi-Kwong Li and Ren-Cang Li.
A note on eigenvalues of perturbed Hermitian matrices.
\emph{Linear Algebra and its Applications}, 395:183--190, 2005.

\bibitem[Loonis and Mary(2019)]{loonis2019}
Vincent Loonis and Xavier Mary.
Determinantal sampling designs.
\emph{Journal of Statistical Planning and Inference}, 199:60--88, 2019.
\href{https://doi.org/10.1016/j.jspi.2018.05.005}{doi:10.1016/j.jspi.2018.05.005}.
Theorem numbers refer to the
\href{https://xmary.perso.math.cnrs.fr/Articles/Determinantal_sampling_designs.pdf}{expanded author manuscript}.

\bibitem[Rhee(2026)]{rhee2026phase}
Kihun Rhee.
A geometric phase boundary for volume-sampled linear readouts.
arXiv:2608.26877, 2026.

\bibitem[Schreurs et~al.(2020)Schreurs, Fanuel, and Suykens]{schreurs2020}
Joachim Schreurs, Micha\"el Fanuel, and Johan A. K. Suykens.
Ensemble kernel methods, implicit regularization and determinantal point processes.
arXiv:2006.13701, 2020.

\bibitem[Wainwright and Jordan(2008)]{wainwright2008}
Martin J. Wainwright and Michael I. Jordan.
Graphical models, exponential families, and variational inference.
\emph{Foundations and Trends in Machine Learning}, 1(1--2):1--305, 2008.

\bibitem[Yu et~al.(2015)Yu, Wang, and Samworth]{yu2015}
Yi Yu, Tengyao Wang, and Richard J. Samworth.
A useful variant of the Davis--Kahan theorem for statisticians.
\emph{Biometrika}, 102(2):315--323, 2015.
\end{thebibliography}
\end{document}